\documentclass[11pt, reqno]{amsart}
\usepackage{amsmath,amsthm,amsfonts,amssymb, wrapfig}
\usepackage[normalem]{ulem}
\usepackage{graphicx}
\usepackage{wrapfig}
\usepackage{mathrsfs}
\usepackage{enumitem}
\usepackage{wrapfig}
\usepackage{etoolbox}
\usepackage{comment}
\usepackage{xcolor}
\usepackage{esint}
\apptocmd{\sloppy}{\hbadness 10000\relax}{}{}
\apptocmd{\sloppy}{\vbadness 10000\relax}{}{}
\usepackage[colorlinks=true, linkcolor=blue, citecolor=blue, urlcolor=blue]{hyperref}
\usepackage[letterpaper,margin=1.1in]{geometry}
\usepackage{accents}
\usepackage{enumitem}
\usepackage{todonotes}
\setlist[enumerate]{label=\textup{(\arabic*)}}

\DeclareRobustCommand{\hvec}[1]{\accentset{\scriptstyle\rightharpoonup}{#1}}
\renewcommand{\vec}[1]{\hvec{#1}}

\allowdisplaybreaks

\numberwithin{equation}{section}
\theoremstyle{plain}

\newtheorem{theorem}{Theorem}[section]
\newtheorem{proposition}[theorem]{Proposition}
\newtheorem{porism}[theorem]{Porism}
\newtheorem{corollary}[theorem]{Corollary}

\newtheorem{lemma}[theorem]{Lemma}

\theoremstyle{definition}
\newtheorem{remark}[theorem]{Remark}
\newtheorem{definition}[theorem]{Definition}

\newcommand{\norm}[1]{\left\lVert#1\right\rVert}

\def\R{\mathbb{R}}
\def\SS{\mathbb{S}}

\def\e{\varepsilon}

\def\XXint#1#2#3{{\setbox0=\hbox{$#1{#2#3}{\int}$ }
\vcenter{\hbox{$#2#3$ }}\kern-.6\wd0}}

\newcommand{\dist}{\mathop\mathrm{dist}\nolimits}

\newcommand{\res}{\hbox{ {\vrule height .22cm}{\leaders\hrule\hskip.2cm} }}

\newcommand{\Per}{\operatorname{Per}}

\renewcommand{\tilde}{\widetilde}

\let\epsilon\varepsilon

\usepackage{alphalph}

\title{Quantitative Stability for Anisotropic Gaussian Isoperimetric Inequality}
\author{Sean McCurdy and Kuan-Ting Yeh}


\address{Instituto de Matemáticas\\ Universidad Nacional Autonoma de México\\ Cuidad de México, México}
\email{sean.mccurdy@im.unam.mx}

\address{Department of Mathematics\\ Purdue University\\
 150 N. University Street,
 West Lafayette, IN 47907-2067, USA}
 \email{yeh49@purdue.edu}

\begin{document}

\begin{abstract}  
We prove a sharp, dimension-free quantitative stability theorem for the Gaussian isoperimetric inequality in \emph{general} Gauss spaces $(\mathbb{R}^n,\gamma_A)$, where $A$ is an arbitrary positive-definite matrix. More precisely, we show that the anisotropic Gaussian perimeter deficit controls the square of the anisotropic Gaussian asymmetry with a dimension-free constant, recovering the sharp quadratic stability behavior known in the isotropic setting. The central difficulties are the loss of rotational symmetry and the lack of compactness of the isoperimetric sets. The main novelty of our work is a new symmetrization-based reduction to the small-perturbation regime, together with effective Fuglede-type estimates derived from the weighted constant mean curvature equation. These ingredients overcome the principal obstructions created by anisotropy and extend the proof-from-regularity approach to quantitative isoperimetric stability in the anisotropic setting.
\end{abstract}

\maketitle

\tableofcontents

\section{Introduction}
\subsection{Background on Gaussian Isoperimetry}\label{background} The classical isoperimetric problem admits a natural analogue in Gauss space $(\mathbb{R}^n,\gamma_n)$. Here, the \textit{isotropic} Gaussian measure $\gamma_n$ and its associated perimeter $\Per_{\gamma_n}$ are defined by
\begin{align*}
    \gamma_n(E):= \frac{1}{(2\pi)^{\frac{n}{2}}}\int_E e^{-\frac{1}{2}|x|^2}dx,\qquad \Per_{\gamma_n}(E):= \frac{1}{(2\pi)^{\frac{n-1}{2}}}\int_{\partial^ME}e^{-\frac{1}{2}|x|^2}d\mathcal{H}^{n-1},
\end{align*}
where $\partial^M E$ denotes the $(n-1)$-dimensional essential boundary of $E$ and $\mathcal{H}^{n-1}$ denotes the standard $(n-1)$-dimensional Hausdorff measure.

Unlike in the Euclidean setting, where isoperimetric sets are balls, the isoperimetric sets in Gauss space are half-spaces of the form
\begin{align*}
    H_n(\omega, \overline{s}):=\{x \in \mathbb{R}^n: \omega \cdot x \le \overline{s}\}
\end{align*}
for $\omega \in \mathbb{S}^{n-1}$ and $\overline{s} \in \mathbb{R}$. This was first proved independently by Borell \cite{Borell75_Gaussisoperi} and Sudakov and Cirel'son \cite{CirelsonSudakov74_Gisoperi}, based on L\'evy's isoperimetric inequality on spheres and Poincar\'e's limit theorem (orthogonal projections onto $\mathbb{R}^n$ of the uniform measure on higher-dimensional spheres $\sqrt{N}\,\mathbb{S}^{N-1}$ converge to the isotropic Gaussian measure on $\mathbb{R}^n$ as $N \rightarrow \infty$).  Later, many other proofs were discovered.  In particular, we highlight Ehrhard's geometric proof using the $\epsilon$-enlargement version of perimeter, which provided a direct analogue in Gauss space for the proof of the isperimetric inequality in Euclidean space using Steiner symmetrization \cite{Ehrhard82,Ehrhard83,Ehrhard84}.  The classification of perimeter minimizers was carried out by Carlen and Kerce in \cite{CarlenKerce01}, where they used semigroup methods (following the work of Ledoux \cite{Ledoux94}) to show that half-spaces are the unique minimizers of the Gaussian weighted perimeter in their mass class.

Once the equality cases had been classified, the next natural question for the isoperimetric inequality is \textit{quantitative stability}, \textit{i.e.,} whether or not it is possible to control the ``distance" of a set $E \subset \mathbb{R}^n$ from the class of perimeter minimizers by the error or \textit{deficit} in the isoperimetric inequality.  That is, in the Gaussian setting, one wants to prove an inequality of the form
\begin{align*}
    \delta_{\gamma_n}(E):= \text{Per}_{\gamma_n}(E) - \min_{\substack{F \subset \mathbb{R}^n\\\gamma_n(F) = \gamma_n(E)}}\text{Per}_{\gamma_n}(F) \ge B(E),
\end{align*}
where $B$ is some quantity that is zero when $E$ is a perimeter minimizer and gives some sense of how ``far" $E$ is from being a perimeter minimizer, in this case, a half-space.  Quantitative stability estimates (also called Bonnesen-style inequalities) have a long history, see for example \cite{Osserman79_BonnesenIneqs}. However, since Hall, Hayman, and Weitsman's introduction of Fraenkel's asymmetry in the context of Newtonian capacity in \cite{HallHaymanWeitsman91}, Fraenkel-type asymmetries have received the most attention in dimensions $n\ge 3$. In the Gaussian context, the \textit{Gaussian asymmetry} is defined by
\begin{align*}
    \lambda_{\gamma_n}(E):=\min_{\omega\in \mathbb{S}^{n-1}}\Big\{\gamma_{n}(E \Delta H_n(\omega,s)): \gamma_n(E)=\gamma_n(H_n(\omega,s))\Big\}.
\end{align*}

In the \textit{isotropic} Gaussian setting, Cianchi, Fusco, Maggi, and Pratelli \cite{CianchiFuscoMaggiPratelli11} were the first to establish a sharp quantitative stability estimate for the Gaussian isoperimetric inequality:
$$
\lambda_{\gamma_n}(E)^2 \leq C(n,\alpha)\delta_{\gamma_n}(E),
$$
for every measurable set $E \subset \mathbb{R}^n$ with $\gamma_A(E)=\alpha$. For $n\geq 2$, the exponent $2$ is sharp, as in the Euclidean setting. Their proof adapts the argument of \cite{FuscoMaggiPrateli08_QuantStabIsoIneq} to the Gaussian setting. After a careful analysis of the one-dimensional case, they use a reflection argument to reduce the problem to $(n-1)$-symmetric sets. They then analyze the two-dimensional case and, through a more involved comparison argument, reduce the higher-dimensional case to establishing an estimate of a slightly different form. The reduction to $(n-1)$-symmetric sets is particularly delicate, as it requires simultaneous control of the Gaussian asymmetry and the Gaussian deficit; this is also where the dimension dependence enters the estimates.

The dependence of the constant on the dimension is somewhat unnatural, since the perimeter of the minimizers, namely half-spaces, is independent of the dimension in Gauss space. The first \textit{dimension-independent} quantitative stability estimates for the isotropic Gaussian isoperimetric inequality were obtained by Mossel and Neeman \cite{MosselNeeman15}. Their argument uses semigroup techniques, building on the work of Ledoux \cite{Ledoux94} and Carlen and Kerce \cite{CarlenKerce01}. We note, however, that the dependence on $\delta_{\gamma_n}(E)$ obtained by Mossel and Neeman is not sharp.

In the same year, Eldan \cite{Eldan2015} introduced the \textit{barycenter excess} $\beta_{\gamma_n}(E)$. Since $\beta_{\gamma_n}(E)$ controls the \textit{Gaussian asymmetry}, quantitative stability estimates in terms of the barycenter excess are naturally stronger than those formulated in terms of the Gaussian asymmetry. Eldan obtained such a quantitative stability estimate as a consequence of a two-sided quantitative stability result for Gaussian noise stability. His argument relies on methods from stochastic calculus. The resulting estimate is a $\log$-factor away from being sharp.

The search for sharp, dimension-free quantitative stability estimates for the Gaussian isoperimetric inequality was ultimately resolved by Barchiesi, Brancolini, and Julin \cite{BarchiesiBrancoliniJulin17}. The core of their approach is the study of minimizers of a functional containing a penalization term. Using first- and second-variation formulas, they show that, under suitable conditions, the relevant minimizers are half-spaces, thereby reducing the problem to an analysis of half-spaces. More recently, Barchiesi and Julin \cite{BarchiesiJulin2020} obtained a full characterization of the minimizers of this functional, showing that the only minimizers are either half-spaces or symmetric strips.

\subsection{Main Result and Novelty}\label{Main Result} The methods discussed above rely crucially on the rotational invariance of the isotropic Gaussian measure. This naturally raises the question of what remains true when rotational invariance is lost, as in the anisotropic Gaussian setting. Recently, the anisotropic Gaussian isoperimetric problem was studied from a geometric perspective by the second named author in \cite{Yeh2024}. Let $n\ge 2$, let $A$ be a symmetric positive-definite $n\times n$ matrix, and let $\gamma_A$ denote the \emph{$A$-anisotropic Gaussian measure}
\begin{align}\label{e:gamma A mass def}\gamma_A(E):=\frac{\sqrt{\det A}}{(2\pi)^{\frac{n}{2}}}\int_{E} e^{-\langle Ax,x\rangle /2}\ dx.\end{align}
The perimeter associated with $\gamma_A$ is called the \emph{$A$-anisotropic Gaussian perimeter} (or simply the \emph{$\gamma_A$-perimeter}) and is defined by
\begin{align}\label{e:gamma A perimeter def}\Per_{\gamma_A}(E):=\frac{\sqrt{\det A}}{(2\pi)^{\frac{n-1}{2}}}\int_{\partial^M E} e^{-\langle Ax,x\rangle /2}\ d\mathcal{H}^{n-1}(x).
\end{align}
It was shown in \cite[Theorem 1.1]{Yeh2024} that the minimizers of the anisotropic Gaussian isoperimetric problem are half-spaces corresponding to the smallest eigenvalue of \(\sqrt A\), namely,
\begin{equation}\label{unique minimizer}
H_n\left(\omega,\frac{\phi^{-1}(\gamma_A(E))}{d_{\min}}\right)\mbox{\qquad for some unit vector $\omega\in V_{d_{\min}}(\sqrt{A})$},
\end{equation}
where $d_{\min}$ denotes the smallest eigenvalue of $\sqrt{A}$, $V_{d_{\min}}(\sqrt{A})$ is the corresponding eigenspace, and $\phi(s):=\frac{1}{\sqrt{2 \pi}} \int_{-\infty}^s e^{-\frac{1}{2} t^2} d t$. In particular, we have
\begin{align*}
\min_{\substack{F \subset \mathbb{R}^n\\\gamma_{A}(F) = \gamma_A(E)}}\text{Per}_{\gamma_A}(F)= e^{-[\phi^{-1}(\gamma_A(E))]^2/2}\frac{1}{\|(\sqrt{A})^{-1}\|}.
\end{align*}

Equation \eqref{unique minimizer} suggests that the natural formulation of quantitative stability in the anisotropic Gaussian setting should be expressed in terms of the following quantities. We define the \emph{$A$-anisotropic Gaussian perimeter deficit} (or simply the \emph{$\gamma_A$-deficit}) by
\begin{align}\label{e:gamma A deficit def}
\delta_{\gamma_A}(E)&:=\Per_{\gamma_A}(E)-\min_{\substack{F \subset \mathbb{R}^n\\\gamma_{A}(F) = \gamma_A(E)}}\text{Per}_{\gamma_A}(F)=\Per_{\gamma_A}(E)-e^{-[\phi^{-1}(\gamma_A(E))]^2/2}\frac{1}{\|(\sqrt{A})^{-1}\|},
\end{align}
and the \emph{$A$-anisotropic Gaussian asymmetry} (or simply the \emph{$\gamma_A$-asymmetry}) by
\begin{align}\label{e:gamma A asymmetry def}\lambda_{\gamma_A}(E):=\min _{\omega \in V_{d_{\min}}(\sqrt{A})\cap \SS^{n-1}} \gamma_A\left(E \Delta H_n\left(\omega, \frac{\phi^{-1}\left(\gamma_A(E)\right)}{d_{\min}}\right)\right). 
\end{align}
Our main result establishes a quantitative stability estimate for the isoperimetric inequality in \textit{general} Gauss spaces $(\mathbb{R}^n,\gamma_A)$.

\begin{theorem}[Quantitative Anisotropic Gaussian Isoperimetric Inequality]\label{t:main theorem}
Let $\alpha\in (0,1)$ and $n\geq 2$. Let $A$ be an $n \times n$ symmetric positive-definite matrix with two smallest eigenvalues $0<d_1^2\le d_2^2$, counting multiplicity.  There exists a constant $0<C(\alpha,d_1^2, d_2^2)<\infty$ such that 
\begin{align*}
    \lambda_{\gamma_A}(E)^2 \leq C(\alpha,d_1^2, d_2^2)\, \delta_{\gamma_A}(E)
\end{align*}
for every measurable set $E \subset \mathbb{R}^{n}$ with $\gamma_A(E)=\alpha$.
\end{theorem}
As a trivial extension of Theorem \ref{t:main theorem}, we have the following result for \textit{arbitrary} real positive-definite matrices.
\begin{corollary}
    Let $\alpha\in (0,1)$ and $n\geq 2$. Let $A$ be an $n \times n$ (real) positive-definite matrix such that the symmetric part of $A$ has two smallest eigenvalues $0<d_1^2\le d_2^2$, counting multiplicity.  There exists a constant $0<C(\alpha,d_1^2, d_2^2)<\infty$ such that 
\begin{align*}
    \lambda_{\gamma_A}(E)^2 \leq C(\alpha,d_1^2, d_2^2)\, \delta_{\gamma_A}(E)
\end{align*}
for every measurable set $E \subset \mathbb{R}^{n}$ with $\gamma_A(E)=\alpha$.
\end{corollary}

\begin{proof}
    For any square matrix $A$, 
$$\left\langle \left(\frac{A+A^T}{2}\right)x,x\right\rangle = \langle Ax,x\rangle.$$
Therefore, we may extend the definitions of the quantities $\gamma_{A}, \delta_{\gamma_A}, \lambda_{\gamma_A}$ to arbitrary real positive-definite matrices, taking care that in our formulas we use $\sqrt{\text{det}(\frac{A+A^T}{2})}$ in place of $\sqrt{\text{det}A}$ to correctly normalize them to probability measures.  The corollary then follows immediately from Theorem \ref{t:main theorem}.
\end{proof}

To the best of our knowledge, this is the first quantitative stability estimate for general \textit{anisotropic} Gauss spaces and, in the isotropic setting, it recovers and in some respects refines the results of \cite{CianchiFuscoMaggiPratelli11,MosselNeeman15,Eldan2015} by independent means. We also show that the exponent 2 is sharp in the anisotropic setting, extending the corresponding sharpness known in the isotropic case. We emphasize that the anisotropy considered here enters through the matrix defining the underlying Gaussian weight. In particular, our $A$-anisotropic Gaussian perimeter should be distinguished from the anisotropic perimeter commonly considered in the literature; see, for instance, \cite{FigalliMaggiPratelli10_Mass_transport_anisotropic_IP}.

At first sight, one might hope to derive quantitative stability results in anisotropic Gauss spaces from their isotropic counterparts by a change of variables. However, the anisotropy encoded by the matrix $A$ leads to a substantially more restrictive class of isoperimetric sets; see \eqref{unique minimizer}. Consequently, the natural asymmetry $\lambda_{\gamma_A}(E)$ introduced in \eqref{e:gamma A asymmetry def} is not directly comparable to its isotropic Gaussian counterpart. Thus, Theorem \ref{t:main theorem} cannot be obtained simply by applying a change of variables to the results of \cite{CianchiFuscoMaggiPratelli11,MosselNeeman15,Eldan2015,BarchiesiBrancoliniJulin17}. We refer to Section \ref{ss:anisotropic setting} for further discussion.

In broad outline, the proof of Theorem \ref{t:main theorem} is inspired by the ``proof from regularity'' strategy introduced by Cicalese and Leonardi \cite{CicaleseLeonardi2012_selectionprinciple} in the Euclidean setting, known as the \emph{selection principle}. The strategy is to first reduce to the small-perturbation regime, wherein one considers small $C^1$ or Lipschitz perturbations of a perimeter minimizer.  Then, one works to establish a quantitative estimate near a perimeter-minimizing configuration and then use qualitative stability to pass from the small asymmetry regime to the large asymmetry regime. The implementation of this strategy in the present setting is substantially different. In particular, the reduction to the small-perturbation regime, which constitutes the technical core of the argument, is completely novel. It is carried out using elementary geometric ideas together with the regularity theory for volume-constrained perimeter minimizers in two-dimensional Gauss space. This reduction is tailored to the anisotropic Gaussian structure and does not arise from a direct adaptation of the Euclidean selection principle.

\subsection{Comparison with Related Results.}\label{ss:anisotropic setting}
We now explain in more detail why the anisotropic stability estimate cannot be obtained directly from the isotropic theory by a change of variables. We shall assume that $A$ is symmetric and positive-definite, as in the hypotheses of Theorem \ref{t:main theorem}.  At the level of the deficit, the $\gamma_A$-deficit is comparable to the $\gamma_n$-deficit, i.e., $\delta_{\gamma_A}(E)\approx \delta_{\gamma_n}(\sqrt{A}E)$ by change of variables. The essential difference lies in the asymmetry. In general, the $\gamma_A$-asymmetry is \emph{not} comparable to the corresponding $\gamma_n$-asymmetry of $\sqrt{A}E$. Because of the asymmetry of $A$, the class of comparison sets is restricted to half-spaces whose defining vectors are restricted to $V_{d_{\min}}(\sqrt{A}) \cap \mathbb{S}^{n-1}$. This means that, if for example $E,H\subset \mathbb{R}^n$, $H$ is a half-space with $\gamma_A(E) = \gamma_A(H)$ and $\lambda_{\gamma_A}(E) = \gamma_A(E\Delta H)$, then by change of variables one has the trivial observation that $\lambda_{\gamma_A}(E) = \gamma_{n}\big(\sqrt{A}E \Delta \sqrt{A}H\big)$, where $\sqrt{A}H$ is another half-space by linearity.  However, it is not true that 
$$\lambda_{\gamma_A}(E)=\gamma_{n}\big(\sqrt{A}E \Delta \sqrt{A}H\big) \approx \lambda_{\gamma_n}(\sqrt{A}E).$$  Indeed, letting $E$ be a half-space, one could have $\lambda_{\gamma_n}(\sqrt{A}E)=0$ and $\lambda_{\gamma_A}(E) > 0$.  Thus, the anisotropic case is not reducible to the corresponding isotropic results by change of variables. Moreover, this calculation shows that, in general $\lambda_{\gamma_n}(\sqrt{A}E)\le \lambda_{\gamma_A}(E)$, so the inequality in Theorem \ref{t:main theorem} is stronger than the corresponding inequality induced by change of variables from the isotropic case.  

Since quantitative stability in the anisotropic Gaussian setting cannot be obtained directly by a change of variables, a natural next step is to ask whether Theorem \ref{t:main theorem} can be proved by adapting one of the elegant arguments available in the Euclidean or \textit{isotropic} Gaussian setting. The presence of anisotropy, however, makes such adaptations far from straightforward. For brevity, we consider the approaches of \cite{CianchiFuscoMaggiPratelli11,BarchiesiBrancoliniJulin17,CicaleseLeonardi2012_selectionprinciple} and indicate where the anisotropic structure prevents a direct adaptation of their methods. 

As mentioned in Section \ref{background}, one of the key steps in \cite{CianchiFuscoMaggiPratelli11} is the reduction to $(n-1)$-symmetric sets.  This argument relies upon finding a hyperplane such that one of the symmetric sets generated by reflection satisfies certain estimates controlling the growth in $\delta_{\gamma_n}$ and $\lambda_{\gamma_n}$. Besides introducing dimensional dependence, this step  uses the rotational symmetry of the isotropic Gaussian measure. In the anisotropic setting, however, the density $e^{-\frac{1}{2}\langle Ax, x\rangle}$ is not symmetric under reflection across arbitrary hyperplanes containing the origin. Consequently, this reduction does not carry over directly to the anisotropic setting. The approach developed in this paper therefore follows a different reduction strategy.

The proof of \cite{BarchiesiBrancoliniJulin17} also appears to rely in an essential way on the rotational invariance of isotropic Gauss space. They observe that the barycenter excess $\beta_{\gamma_n}(E)$ is minimized in the direction $\omega=-\frac{b(E)}{|b(E)|}$. Moreover, the barycenter excess can be expressed in this direction as the difference $|b(H_n(\omega, s))|-|b(E)|$. Combined with the fact that the minimizers of their regularizing functional are half-spaces for suitable choices of the parameters, this yields a quantitative estimate relating the barycenter excess to the Gaussian perimeter deficit. However, in the anisotropic setting, the $\gamma_A$-perimeter minimizers are only half-spaces whose unit normal vectors lie in the eigenspace $V_{d_{\text {min }}}(\sqrt{A})$ corresponding to the smallest eigenvalue of $\sqrt{A}$. It seems unclear how to carry forward their method if the barycenter direction $\omega=-\frac{b(E)}{|b(E)|}$ is not an element of the eigenspace $V_{d_{\text {min }}}(\sqrt{A})$, and therefore not ``aligned" with an isoperimetric half-space. A direct adaptation of \cite{BarchiesiBrancoliniJulin17} to the anisotropic Gauss setting would therefore require additional control of the angle between these two directions, which leads to a substantial technical difficulty. We note, however, that in the isotropic case our result does not recover the \textit{explicit} dimension-free estimates obtained in \cite{BarchiesiBrancoliniJulin17}.

The work of \cite{CicaleseLeonardi2012_selectionprinciple} was an important inspiration for this project. Nevertheless, it is not clear how to adapt their method even to the \textit{isotropic} Gaussian setting. Here, the principal difficulty is the loss of compactness. In Euclidean space, perimeter minimizers are balls. Thus, after establishing that their regularized sets $\widetilde E_k$ are uniformly bounded, Cicalese and Leonardi can simply apply Tamanini's regularity theory \cite{tamanini1984regularity} to reduce the problem to the small-perturbation regime. In the Gaussian setting, however, the situation is fundamentally different. Even if one were to introduce an analogous regularizing functional, one would not expect the corresponding sets $\widetilde E_k$ to be uniformly bounded, since they must approximate half-spaces. Thus, Tamanini's regularity theory would only give that in a very large rectangle $R$, one could reduce to considering sets $\widetilde{E}_k$ such that
\begin{align*}
    \widetilde{E}_k \cap R \quad \text{is a }C^1\text{-perturbation of a half-space,}
\end{align*}
but it does not allow us to say anything about what happens outside this rectangle $R$. \emph{A priori}, a substantial portion of both the deficit $\delta_{\gamma_A}(\widetilde E_k)$ and the asymmetry $\lambda_{\gamma_A}(\widetilde E_k)$ could escape outside such rectangle $R$ as $k\to\infty$. Thus, even if one could prove a local version of a Fuglede estimate in the rectangle $R$, it would not provide the global control required to prove Theorem \ref{t:main theorem}. Thus, a direct adaptation of this approach breaks down at this stage.

While these difficulties may not be insurmountable, our aim is to obtain a direct and geometric proof of sharp quantitative stability. This leads us to a different and more robust approach, whose main ideas are presented in the next section.

\subsection{Proof of Theorem \ref{t:main theorem}.}\label{ss:outline of proof}
As mentioned in Section \ref{Main Result}, the proof of Theorem \ref{t:main theorem} follows the general ``proof from regularity'' strategy introduced by Cicalese and Leonardi \cite{CicaleseLeonardi2012_selectionprinciple} in the Euclidean setting. More precisely, we study the quantity
\begin{align}\label{e:Q def}
    Q(E):= \inf \left\{\liminf_{k\to \infty}\frac{\delta_{\gamma_A}(E_k)}{\lambda_{\gamma_A}(E_k)^2}: E_k \overset{L^1(\mathbb{R}^n, \gamma_A)}{\rightarrow} E,\, \lambda_{\gamma_A}(E_k)>0, \, \gamma_{A}(E_k) = \gamma_A(E)\right\}.
\end{align}
Sequences of admissible sets $E_k$ such that $\lim_{k \rightarrow\infty}Q(E_k) = Q(E)$ are called \textit{recovery sequences} for $Q(E)$. Roughly speaking, the ``proof from regularity'' strategy can be outlined as follows.
\begin{enumerate}
    \item (\textit{Reduction to the Small Perturbation Regime}) Let $H$ be a $\gamma_A$-perimeter minimizer. For any recovery sequence $\{E_k\}_k$ for $Q(H)$, one must show that one can reduce to considering a different sequence of sets $\{\widetilde{E}_k\}_k$ with better properties.  In the anisotropic Gaussian case, we want $\widetilde{E}_k$ to be $W^{1,\infty}$-perturbations of the half space $H$ such that the $W^{1,\infty}$-norm vanishes as $k \rightarrow \infty$.
    \item (\textit{Estimates on the Small Perturbation Regime}) Next, one must obtain lower estimates on $\frac{\delta_{\gamma_A}(\widetilde{E}_k)}{\lambda_{\gamma_A}(\widetilde{E}_k)^2}$ for the sequence of sets $\widetilde{E}_k \subset \mathbb{R}^n$ obtained from the previous step. Typically, this follows the work of Fuglede \cite{Fuglede89}, which requires control on the barycenter.
    \item (\textit{Compactness}) Finally, one uses the lower semicontinuity of $Q$ under convergence in $L^1(\mathbb{R}^n,\gamma_A)$, together with a qualitative stability result, to obtain a universal lower bound for $Q(H)$.
\end{enumerate}

We now make these steps precise in the anisotropic Gaussian setting. We first state the auxiliary results needed for the proof of Theorem \ref{t:main theorem} and show how they combine to yield the main result. Their proofs are deferred to the subsequent sections.

Note that, for $n\geq 2$ and a symmetric positive-definite $n\times n$ matrix $A$ whose two smallest eigenvalues are $0<d_1^2\leq d_2^2$, we may assume without loss of generality that $A$ is diagonal. Furthermore, we write $A$ and $A_{1,2}$ as follows; see also Definition \ref{d:standing assumptions}(3)(4).
\begin{equation}\label{A and A_{1,2}}
A=\begin{pmatrix}
d_1^2 & 0 & \cdots & 0 \\
0 & d_2^2 & \ddots & \vdots \\
\vdots & \ddots & \ddots & 0 \\
0 & \cdots & 0 & d_n^2
\end{pmatrix},\qquad A_{1,2}=A_{d_1,d_2}= \begin{pmatrix}
        d_1^2 & 0\\
        0 & d_2^2
    \end{pmatrix},
\end{equation}
with $0<d_1\leq d_2\leq \cdots\leq d_n$. Here, the notations $A_{1,2}$ and $A_{d_1,d_2}$ are used interchangeably whenever there is no risk of confusion.
    
\begin{theorem}[Reduction to the Small Perturbation Regime]\label{t:technical theorem}
Let $n\ge 2$, $\overline{s}\in\mathbb R$, and let $A$ and $A_{1,2}$ be the matrices defined in \eqref{A and A_{1,2}}. Let $H_n(e_1,\overline{s})\subset\mathbb R^n$ be a
$\gamma_A$-perimeter-minimizing half-space, and let
$H_2(e_1,\overline{s})\subset\mathbb R^2$ be the corresponding
$\gamma_{A_{1,2}}$-perimeter-minimizing half-space. Assume that, for some
$\alpha\in(0,1)$,
\[
\gamma_A\bigl(H_n(e_1,\overline{s})\bigr)
=
\alpha
=
\gamma_{A_{1,2}}\bigl(H_2(e_1,\overline{s})\bigr).
\]
Then the following statements hold.
\begin{enumerate}
    \item \textup{(Symmetrization)} For any $G\subset \mathbb{R}^n$, there exists a set $\tilde{G}\subset \mathbb{R}^2$ such that
    \begin{enumerate}
        \item[a.] $\gamma_A(G) = \gamma_{A_{1,2}}(\tilde{G})=\alpha.$
        \item[b.] $\delta_{\gamma_A}(G)\ge \delta_{\gamma_{A_{1,2}}}(\tilde{G}).$
        \item[c.] $\lambda_{\gamma_A}(G) = \gamma_{A_{1,2}}(\tilde{G}\Delta H_2(e_1,\overline{s}))$.
        \item[d.] There exists $\omega \in \mathbb{S}^{n-1}$ such that for all $t \in \mathbb{R}$, 
        \begin{align*}
            \gamma_{A_{1,2}}(\tilde{G}\Delta H_2(e_1,t)) & = \gamma_{A}(G\Delta H_n(\omega,t))\\
            \gamma_{A_{1,2}}(\tilde{G}\Delta H_2(-e_1,-t)) & = \gamma_{A}(G\Delta H_n(-\omega,-t)).
        \end{align*}
    \end{enumerate}
\item In addition, for any recovery sequence $\{E_k\}_k$ for $Q(H_n(e_1, \overline{s}))$ there exists a sequence $\{\tilde{E}_k\}_k$ of sets $\tilde{E}_k\subset \mathbb{R}^2$ such that for all sufficiently large $k \in \mathbb{N}$,
    \begin{enumerate}
        \item[a.] $\gamma_{A_{1,2}}(\tilde{E}_k) = \alpha_k$ for $|\alpha - \alpha_k| \le o(\lambda_{\gamma_{A}}(E_k)^8)$. 
        \item[b.] $\delta_{\gamma_A}(E_k) \ge \delta_{\gamma_{A_{1,2}}}(\tilde{E}_k) - o(\lambda_{\gamma_A}(E_k)^8)$.
        \item[c.] $\gamma_{A_{1,2}}(\tilde{E}_k \Delta H_2(e_1, \overline{s})) \ge \frac{1}{5}\lambda_{\gamma_A}(E_k).$
    \item[d.] $\widetilde{E}_k$ can be represented as the subgraph of a Lipschitz function $\widetilde{u}_k$ on $\partial H_2(e_1,\overline{s})$, namely, $\widetilde{E}_k = \emph{subgraph}_{\partial H_2(e_1,\overline{s})}(\widetilde{u}_k)$, for some Lipschitz function $\widetilde{u}_k:\partial H_2(e_1,\overline{s})\to\mathbb{R}$ satisfying $$\limsup_{k\to\infty} \|\widetilde{u}_k\|_{W^{1,\infty}(\mathbb{R})} =0.$$
    \end{enumerate}
\end{enumerate}
\end{theorem}

In some sense, the main difference from the reduction argument in \cite{CianchiFuscoMaggiPratelli11} is that we do not attempt to estimate $\lambda_{\gamma_A}(\widetilde{E}_k)$ directly. Instead, we work with the simpler quantity $\gamma_{A_{1,2}}(\tilde{E}_k \Delta H_2(e_1, \overline{s}))$, which is preserved under $(n-1)$-symmetrization.  This allows us to obtain the dimension-free constant $\frac{1}{5}$ in Theorem \ref{t:technical theorem}(2c), above.  The cost of this improvement is that, \textit{a priori}, $(n-1)$-symmetrization destroys any hope of controlling the barycenters of the sets $\tilde{E}_k$.  Control of the barycenter is crucial to the standard Fuglede-type inequality, but surprisingly, we find in Theorem \ref{t:small perturbation regime} that we can use an ODE to recover this control despite the $(n-1)$-symmetrization (see Section \ref{s:The Small Perturbation Regime} for more details).

A direct consequence of Theorem \ref{t:technical theorem} is the following corollary, which provides a lower bound for $Q(H_n(e_1,\bar{s}))$.

\begin{corollary}\label{c:reduction ratio estimates}
Under the hypotheses of Theorem \ref{t:technical theorem}(2), the sets $\tilde{E}_k \subset \mathbb{R}^2$ satisfy
        \begin{align*}
            \frac{\delta_{\gamma_A}(E_k)}{\lambda_{\gamma_A}(E_k)^2} \ge \frac{1}{25}\frac{\delta_{\gamma_{A_{1,2}}}(\tilde{E}_k)}{\gamma_{A_{1,2}}(\tilde{E}_k\Delta H_2(e_1,\overline{s}))^2} -o(\lambda_{\gamma_A}(E_k)^6),
        \end{align*}
        and therefore
        \begin{align*}
            Q(H_n(e_1, \overline{s})) = \lim_{k \rightarrow \infty}\frac{\delta_{\gamma_A}(E_k)}{\lambda_{\gamma_A}(E_k)^2} \ge \frac{1}{25}\liminf_{k \rightarrow \infty}\frac{\delta_{\gamma_{A_{1,2}}}(\tilde{E}_k)}{\gamma_{A_{1,2}}(\tilde{E}_k \Delta H_2(e_1, \overline{s}))^2}.
        \end{align*}
\end{corollary}

Next, we state our estimates in the small-perturbation regime. The estimates in Theorem \ref{t:small perturbation regime} are obtained through Taylor expansion and a Fuglede-type estimate. See Section \ref{s:The Small Perturbation Regime} for the proof and further discussion.

\begin{theorem}[Estimates for the Small Perturbation Regime]\label{t:small perturbation regime}
    Let $A, A_{1,2}, d_1^2, d_2^2, \overline{s}, \alpha,$ $\{E_k\}_k, \{\tilde{E}_k\}_k$, and $\tilde{u}_k$ be as in Theorem \ref{t:technical theorem}. Then,
        \begin{align*}
        \liminf_{k\rightarrow \infty}\frac{\delta_{\gamma_{A_{1,2}}}(\tilde{E}_k)}{\gamma_{A_{1,2}}(\tilde{E}_k \Delta H_2(e_1, \overline{s}))^2} \ge \begin{cases}\frac{(2\pi)e^{\frac{d_1^2\overline{s}^2}{2}}}{d_1}\left[\frac{d_2^2-d_1^2}{2}+d_2^2(\delta_0-\frac{1}{2}\delta_0^2)\right]& 1 \le \frac{d_2^2}{d_1^2}\le 2\\
        \frac{(2\pi)e^{\frac{d_1^2\overline{s}^2}{2}}}{d_1}\left[\frac{d_2^2-d_1^2}{2}\right] & \frac{d_2^2}{d_1^2}>2,
        \end{cases}
    \end{align*}
    where $0<\delta_0<1$ is a universal constant given in Lemma \ref{l:close to linear and comparability of lambdas}.
\end{theorem}

For convenience, we define
\begin{align}\label{e:spr lb constant}
    C_{SPR}(d_1, d_2, \alpha):= \begin{cases}\frac{(2\pi)e^{\frac{d_1^2\overline{s}^2}{2}}}{d_1}\left[\frac{d_2^2-d_1^2}{2}+d_2^2(\delta_0-\frac{1}{2}\delta_0^2)\right]& 1 \le \frac{d_2^2}{d_1^2}\le 2\\
        \frac{(2\pi)e^{\frac{d_1^2\overline{s}^2}{2}}}{d_1}\left[\frac{d_2^2-d_1^2}{2}\right] & \frac{d_2^2}{d_1^2}>2.
        \end{cases}
\end{align}

Now, we state our qualitative stability result and the relevant compactness result.

\begin{lemma}[Qualitative Stability]\label{l:qualitative stability}
Let $A$ be a symmetric, positive-definite $n\times n$ matrix. For every $\varepsilon>0$ and $\alpha\in(0,1)$, there exists $\rho=\rho(\alpha,\varepsilon,A,n)>0$ such that, if $F\subset\mathbb{R}^n$ satisfies
    \begin{align*}
        \gamma_A(F)=\alpha, \qquad \delta_{\gamma_A}(F)< \rho(\alpha, \varepsilon, A, n),
    \end{align*}
    then $\lambda_{\gamma_A}(F)< \varepsilon$. Moreover, the threshold $\rho(\alpha,\epsilon, A, n)$ satisfies the following properties:
    \begin{enumerate}
        \item $\rho(\alpha,\epsilon, A, n)$ is monotone non-decreasing in $\epsilon$ and left continuous in $\epsilon$.
        \item $\rho(\alpha,\epsilon, A, n)$ is lower semicontinuous in $\alpha$ and $A$ jointly.
    \end{enumerate}
Here, to describe the lower semicontinuity with respect to $A$, we use the topology induced by the pseudometric
\begin{align*}
    \dist(A, \tilde{A})^2 =\sum_{i=1}^n (d_i^2 -(\tilde{d_i})^2)^2,
\end{align*}
where $\{d_i^2\}_{i=1}^n$ and $\{\widetilde{d}_i^{\,2}\}_{i=1}^n$ are the eigenvalues of $A$ and $\widetilde{A}$, respectively, listed in increasing order and counted with multiplicity. When $A$ is diagonal and has the form given in \eqref{A and A_{1,2}}, we will, for simplicity, write $\rho(\alpha,\epsilon,A,n)$ as
\[
\rho(\alpha,\epsilon,A,n)=\rho(\alpha,\epsilon,d_1,d_2,\ldots,d_n,n).
\]
\end{lemma}

\begin{definition}\label{d:epsilon threshold def}
Let $A_{1,2}$ be a diagonal, positive-definite $2\times 2$ matrix with eigenvalues
$0<d_1^2\leq d_2^2$, as in \eqref{A and A_{1,2}}, and let $\alpha\in(0,1)$. We define
$\epsilon_0(d_1,d_2,\alpha)$ as the supremum of all $0<\epsilon\leq 2\alpha$
with the following property: for every $\mathcal{L}^2$-measurable set
$F\subset\mathbb{R}^2$ satisfying $\gamma_{A_{1,2}}(F) =\alpha = \gamma_{A_{1,2}}(H_{2}(e_1, \overline{s}))$, if $0<\gamma_{A_{1,2}}(F \Delta H_2(e_1, \overline{s})) \le \epsilon$ then
\begin{align*}
\frac{\delta_{\gamma_{A_{1,2}}}(F)}{\gamma_{A_{1,2}}(F\Delta H_2(e_1, \overline{s}))^2} \ge  \frac{1}{100} C_{SPR}(d_1, d_2, \alpha),
\end{align*}
where $C_{SPR}(d_1, d_2, \alpha)$ is defined in \eqref{e:spr lb constant} and $\overline{s}=\frac{1}{d_1}\phi^{-1}(\alpha)$ with $\phi(s)=\frac{1}{\sqrt{2 \pi}} \int_{-\infty}^s e^{-\frac{1}{2} t^2} d t$.
\end{definition}

\begin{lemma}[Compactness]\label{l:continuity of epsilon}
    For $A_{1,2}$, $d_1^2$, $d_2^2$, $\alpha$, and $\epsilon_0(d_1,d_2,\alpha)$ as in Definition \ref{d:epsilon threshold def}, we have $\epsilon_0(d_1,d_2,\alpha)>0$,
and the function $(d_1,d_2,\alpha)\mapsto \epsilon_0(d_1,d_2,\alpha)$ is upper semicontinuous.
\end{lemma}
The proofs of Lemma \ref{l:qualitative stability} and Lemma \ref{l:continuity of epsilon} are given in Section \ref{s:qualitative stability}. Assuming Theorem \ref{t:technical theorem}, Theorem \ref{t:small perturbation regime}, Lemma \ref{l:qualitative stability}, and Lemma \ref{l:continuity of epsilon}, the proof of Theorem \ref{t:main theorem} proceeds as follows.\\

\noindent {\bf Proof of Theorem \ref{t:main theorem}}: Fix $\alpha\in(0,1)$. Let $n\geq 2$, and let $A$ be a symmetric positive-definite $n\times n$ matrix whose two smallest eigenvalues are $0<d_1^2\leq d_2^2$. Following the comment before \eqref{A and A_{1,2}}, we recall that we may assume that $A$ and $A_{1,2}$ have the form given in \eqref{A and A_{1,2}}. With a slight abuse of notation, we denote by $e_1$ and $e_2$ the corresponding standard basis vectors in both $\mathbb{R}^n$ and $\mathbb{R}^2$. Thus, $e_1$ is an eigenvector associated with $d_1^2$ and $e_2$ with $d_2^2$ for both $A$ and $A_{1,2}$. Let $\overline{s} \in\mathbb{R}$ such that $\gamma_{A_{1,2}}(H_2(e_1, \overline{s})) = \gamma_{A}(H_n(e_1, \overline{s}))=\alpha$.

For any recovery sequence $\{E_k\}_k$ for $Q(H_n(e_1, \overline{s}))$, by Theorem \ref{t:technical theorem}(2), Corollary \ref{c:reduction ratio estimates}, and Theorem \ref{t:small perturbation regime}, there exist sets $\tilde{E}_k \subset \mathbb{R}^2$ such that
\begin{align}
    Q(H_n(e_1, \overline{s})) & \ge\frac{1}{25} \liminf_{k\rightarrow \infty}\frac{\delta_{\gamma_{A_{1,2}}}(\tilde{E}_k)}{\gamma_{A_{1,2}}(\tilde{E}_k \Delta H_2(e_1, \overline{s}))^2}\ge \frac{1}{25}C_{SPR}(d_1, d_2, \alpha),    
\end{align}
where $C_{SPR}(d_1, d_2, \alpha)>0$ is defined in \eqref{e:spr lb constant}.

Let $\epsilon_0(d_1,d_2,\alpha)$ be as in Definition \ref{d:epsilon threshold def}.  Note that if $\epsilon_0(d_1, d_2, \alpha) = 2\alpha$, then there is nothing to prove, since in this case Theorem \ref{t:main theorem} holds with $C(d_1, d_2, \alpha) = \frac{1}{100}C_{SPR}(d_1, d_2, \alpha)$. Thus, we may assume that $0<\epsilon_0(d_1, d_2, \alpha)<2\alpha$.

For any $n \ge 2$ and any $G \subset \mathbb{R}^n$ with $\gamma_A(G)=\alpha$, if $\lambda_{\gamma_A}(G) < \epsilon_0(d_1,d_2,\alpha)$, then our reduction theorem (Theorem \ref{t:technical theorem}(1)) gives us a set $\tilde{G}\subset \mathbb{R}^2$ such that $\gamma_{A_{1,2}}(\tilde{G})=\alpha$ and 
$$\epsilon_0(d_1, d_2, \alpha) > \lambda_{\gamma_{A}}(G) =\gamma_{A_{1,2}}(\tilde{G}\Delta H_2(e_1, \overline{s})).$$
Hence, by the definition of $\epsilon_0(d_1,d_2,\alpha)$,
\begin{align}\label{e:G to G tilde epsilon estimate}
    \frac{\delta_{\gamma_A}(G)}{\lambda_{\gamma_A}(G)^2} \ge \frac{\delta_{\gamma_{A_{1,2}}}(\tilde{G})}{\gamma_{A_{1,2}}(\tilde{G}\Delta H_2(e_1, \overline{s}))^2} \ge \frac{1}{100}
C_{SPR}(d_1, d_2, \alpha).
\end{align}

Thus, it only remains to consider $G \subset \mathbb{R}^n$ such that $\gamma_A(G)=\alpha$ and $\lambda_{\gamma_A}(G)\ge \epsilon_0(d_1, d_2, \alpha)$.  We claim that for such a set $G$, the $(n-1)$-symmetrization $\tilde{G}\subset \mathbb{R}^2$ guaranteed by Theorem \ref{t:technical theorem}(1) also satisfies
\begin{align}\label{e:claim in proof of big theorem}
    \delta_{\gamma_{A_{1,2}}}(\tilde{G}) \ge \rho(\alpha,\epsilon_0(d_1,d_2,\alpha), d_1, d_2, 2),
\end{align}
where $\rho$ is the quantity guaranteed by Lemma \ref{l:qualitative stability}. We split the proof into two cases.

\noindent \emph{Case 1: The strictly anisotropic case.} Assume that $d_1<d_2$.  Let $G \subset \mathbb{R}^n$ such that $\gamma_A(G)=\alpha$ and $\lambda_{\gamma_A}(G)\ge \epsilon_0(d_1, d_2, \alpha)$ with associated $\tilde{G}\subset \mathbb{R}^n$, as above.  In $2$-dimensions, strictly anisotropic Gauss space has only two $\gamma_{A_{1,2}}$-perimeter minimizers: $H_{2}(e_1,\overline{s})$ and $H_{2}(-e_1,-\overline{s})$.  Thus, if $\lambda_{\gamma_A}(G)\ge \epsilon_0(d_1, d_2, \alpha)$, by Theorem \ref{t:technical theorem}(1),
\begin{align*}
    \lambda_{\gamma_{A_{1,2}}}(\tilde{G}) & = \min\{\gamma_{A_{1,2}}(\tilde{G}\Delta H_{2}(e_1,\overline{s})),\gamma_{A_{1,2}}(\tilde{G}\Delta H_{2}(-e_1,-\overline{s})) \}\\
    & = \min\{\gamma_{A}(G\Delta H_{n}(\omega,\overline{s})),\gamma_{A}(G\Delta H_{n}(-\omega,-\overline{s}))\}\\
    & \ge \lambda_{\gamma_{A}}(G) \ge \epsilon_0(d_1, d_2, \alpha).
\end{align*}
Therefore, by Lemma \ref{l:qualitative stability}, it must be that $\delta_{\gamma_{A_{1,2}}}(\tilde{G}) \ge \rho(\alpha,\epsilon_0(d_1,d_2,\alpha), d_1,d_2, 2)$.

\noindent \emph{Case 2: The isotropic case.}
Let $0<d_1 = d_2$.  Let $G \subset \mathbb{R}^n$ such that $\gamma_A(G)=\alpha$. Suppose that $\lambda_{\gamma_A}(G)\ge \epsilon_0(d_1, d_1, \alpha)$. Without loss of generality, we may assume that $H_n(e_1, \bar{s})$ attains the minimum in the definition of $\lambda_{\gamma_A}(G)$ . Indeed, if the minimum is attained by another half-space, we may apply a suitable rotation to both $G$ and the minimizing half-space so that the latter coincides with $H_n(e_1, \bar{s})$. Note that Theorem \ref{t:technical theorem} immediately gives that there exists a $\tilde{G} \subset \mathbb{R}^2$ such that
\begin{align}\label{e:symm diference ge epsilon}
    \gamma_{A_{1,2}}(\tilde{G}\Delta H_2(e_1, \overline{s})) = \gamma_{A}(G\Delta H_n(e_1, \overline{s})) = \lambda_{\gamma_A}(G)\ge \epsilon_0(d_1, d_1, \alpha),
\end{align}
but this is insufficient to obtain control on $\lambda_{\gamma_{A_{1,2}}}(\tilde{G})$. We need to use Lemma \ref{l:qualitative stability} and Lemma \ref{l:continuity of epsilon}. 

Consider the sequence of matrices $A_{d_1, d_{2, i}}$, where $d_{2, i} \searrow d_1$, and note that, by assumption, $A_{1,2}=$ $A_{d_1, d_1}$. Then, by the continuity of $\gamma_{A_{d_1, d_2}}(\tilde{G}\Delta H_2(e_1, \overline{s}))$ in the $d_2$-variable,
\begin{align}\label{e:limit of symm differences}
    \gamma_{A_{d_1, d_1}}(\tilde{G}\Delta H_2(e_1, \overline{s})) = \lim_{i \rightarrow \infty} \gamma_{A_{d_1, d_{2,i}}}(\tilde{G}\Delta H_2(e_1, \overline{s})).
\end{align}
For this sequence, by Theorem \ref{t:technical theorem}(1d) and the assumption that $H_n(e_1, \overline{s})$ is a best-approximating half plane for $G \subset (\mathbb{R}^n;\gamma_{A})$, after possibly reflecting $\tilde{G}$, we may find a subsequence $d_{2,j}$ such that for all sufficiently large $j \in \mathbb{N}$
\begin{align}\label{e:symm diferences are lambdas}
    \gamma_{A_{d_1, d_{2,j}}}(\tilde{G}\Delta H_2(e_1, \overline{s}_j)) = \lambda_{\gamma_{A_{d_1, d_{2,j}}}}(\tilde{G}),
\end{align}
where we use $\overline{s}_j$ to denote the height such that $\gamma_{A_{d_1, d_{2,j}}}(H_2(e_1, \overline{s}_j)) =\alpha_j := \gamma_{A_{d_1, d_{2,j}}}(\tilde{G})$.  Then by Lemma \ref{l:continuity of epsilon}, \eqref{e:symm diference ge epsilon}, \eqref{e:limit of symm differences}, and \eqref{e:symm diferences are lambdas} we have
\begin{align*}
    0 < \epsilon_0(d_1, d_1, \alpha) & \le \gamma_{A_{d_1, d_1}}(\tilde{G} \Delta H_2(e_1, \overline{s}))\\
    & = \lim_{j \rightarrow \infty}\gamma_{A_{d_1, d_{2,j}}}(\tilde{G}\Delta H_2(e_1, \overline{s}))\\
    & \le \lim_{j\rightarrow \infty}\lambda_{\gamma_{A_{d_1, d_{2,j}}}}(\tilde{G}) + \gamma_{A_{d_1, d_{2,j}}}(H_2(e_1, \overline{s}_j)\Delta H_2(e_1, \overline{s}))\\
    & = \lim_{j\rightarrow \infty}\lambda_{\gamma_{A_{d_1, d_{2,j}}}}(\tilde{G}),
\end{align*}
because $\gamma_{A_{d_1, d_{2,j}}}(H_2(e_1, \overline{s}_j)\Delta H_2(e_1, \overline{s})) \rightarrow 0$ as $j \rightarrow \infty$.

Thus, for any $C<1$ it must be that for all sufficiently large $j \in \mathbb{N}$, $$\lambda_{\gamma_{A_{d_1, d_{2,j}}}}(\tilde{G}) > C\epsilon_0(d_1, d_{1}, \alpha).$$

Therefore, by the lower semicontinuity of $\rho_0$ in the $(\alpha, d_1, d_1)$ coordinates (Lemma \ref{l:qualitative stability}), we see
\begin{align*}
\delta_{\gamma_{A_{d_1, d_1}}}(\tilde{G})& = \lim_{j\rightarrow \infty} \delta_{\gamma_{A_{d_1,d_{2,j}}}}(\tilde{G})\\
    &\ge \liminf_{j\rightarrow \infty} \rho(\alpha_j, C\epsilon_0(d_1, d_{1}, \alpha), d_1,d_{2,j}, 2)\\
    & \ge \rho(\alpha, C\epsilon_0(d_1, d_{1}, \alpha), d_1,d_{1}, 2).
\end{align*}
Letting $C\nearrow 1$ and using the left continuity of $\rho$ guaranteed by Lemma \ref{l:qualitative stability}, we obtain \eqref{e:claim in proof of big theorem}.

Finally, we have \eqref{e:claim in proof of big theorem} in all cases, Theorem \ref{t:technical theorem}(1) immediately gives
\begin{align*}
    \frac{\delta_{\gamma_A}(G)}{\lambda_{\gamma_{A}}(G)^2} \ge 
    \frac{\delta_{\gamma_{A_{1,2}}}(\tilde{G})}{4\alpha^2}\ge \frac{1}{4\alpha^2}\rho(\alpha,\epsilon_0(d_1,d_2, \alpha), d_1, d_2, 2).
\end{align*}

Therefore, for any $n \ge 2$ and any $G \subset \mathbb{R}^n$ with $\gamma_{A}(G)=\alpha\in (0,1)$,
\begin{align*}
  \frac{\delta_{\gamma_{A}}(G)}{\lambda_{\gamma_{A}}(G)^2} &\ge \min\left\{\frac{1}{100}C_{SPR}(d_1, d_2, \alpha), \frac{1}{4\alpha^2}\rho(\alpha,\epsilon_0(d_1,d_2, \alpha), d_1, d_2, 2)\right\}\\
 &=: C(\alpha, d_1^2, d_2^2)>0.
\end{align*}
This gives Theorem \ref{t:main theorem}.
\qed

\subsection{Organization of the Paper}

The remainder of the paper is devoted to the proofs of Theorem \ref{t:technical theorem}, Theorem \ref{t:small perturbation regime}, Lemma \ref{l:qualitative stability}, and Lemma \ref{l:continuity of epsilon}. We first outline the proof of Theorem Theorem \ref{t:technical theorem} at the end of Section \ref{Notation and Preliminaries}, as it constitutes our main technical contribution. Sections \ref{s:reduction_symmetrization}--\ref{s:truncation} are then devoted to its full proof. In Section \ref{s:The Small Perturbation Regime}, we derive a Fuglede-type estimate, establish the sharpness of the exponent $2$ in the anisotropic Gaussian setting, and prove Theorem \ref{t:small perturbation regime}. Finally, in Section \ref{s:qualitative stability}, we prove Lemma \ref{l:qualitative stability} and Lemma \ref{l:continuity of epsilon} by standard compactness-contradiction arguments.

\subsection{Additional Remarks} We conclude the introduction by mentioning a few related developments. Fusco and La Manna \cite{FuscoLaManna2023} extended \cite{CianchiFuscoMaggiPratelli11} from isotropic Gauss space to radial weights of the form $e^{-\omega(|x|)}$, while Alvino, Brock, Chiacchio, Mercaldo, and Posteraro \cite{AlvinoBrockChiacchioMercaloPosteraro2019} studied anisotropic power weights $|x_n|^\alpha$ in half-spaces. La Manna \cite{LaManna2019_BallsLocalMinimizerGaussianPerimeter} used the selection principle to prove local minimality of centered balls for the $\gamma_n$-perimeter. For additional work on the function $Q(E)$ in the Euclidean setting, Alvino, Ferone, and Nitsch give a second-order expansion of the perimeter deficit, centered at the ball $B_r(0)$, in terms of the Fraenkel asymmetry in $\mathbb{R}^2$ for a nice family of subsets in \cite{AlvinoFeroneNitsch2011_2DHigherOrderExpansion}.  In a follow-up to \cite{CicaleseLeonardi2012_selectionprinciple}, Cicalese and Leonardi use the selection principle to obtain higher-order expansions of the Euclidean perimeter, centered at the ball $B_r(0)$, as a polynomial of the asymmetry in $\mathbb{R}^2$ for all Borel Sets of positive and finite measure \cite{CicaleseLeonardi2013_TaylorExpansionPerimeteratBall}.

\section{Notation and Preliminaries}\label{Notation and Preliminaries}
For the remainder of this paper, we fix the following definitions and notation.

\begin{definition}[Standing assumptions]\label{d:standing assumptions}\mbox{}
    \begin{enumerate}
        \item Unless otherwise specified, we shall assume that $n$ is an integer satisfying $n \ge 2$.
        \item Let $S_{n \times n}^+(\mathbb{R})$ denote the set of all symmetric, positive-definite $(n \times n)$-matrices with real coefficients. We consider $S_{n \times n}^+(\mathbb{R})$ as a pseudometric space $(S_{n \times n}^+(\mathbb{R}),d_{\psi})$, where $d_{\psi}$ is the pseudometric
        \begin{align*}
    \dist(A, \tilde{A})^2 =\sum_{i=1}^n (d_i^2 -(\tilde{d_i})^2)^2,
\end{align*}
where $\{d_i^2\}_{i=1}^n$ and $\{(\tilde{d_i})^2\}_{i=1}^n$ are the eigenvalues of $A$ and $\tilde{A}$, respectively, listed with multiplicity in increasing order, as defined in Lemma \ref{l:qualitative stability}.
        
\item For the rest of the paper, we shall assume that $A \in (S_{n \times n}^+(\mathbb{R}), d_{\psi})$.  By rotation we may always assume that $A$ is diagonal in the standard basis.
$$
A=\begin{pmatrix}
d_1^2 & 0 & \cdots & 0 \\
0 & d_2^2 & \ddots & \vdots \\
\vdots & \ddots & \ddots & 0 \\
0 & \cdots & 0 & d_n^2
\end{pmatrix}
=\begin{pmatrix}
d_1^2 & 0 \\
0 & A'
\end{pmatrix}.
$$
        where $A'$ is the corresponding $(n-1)\times(n-1)$-submatrix.
\item For $A$ as above, we write
\begin{align*}
    A_{d_1,d_2}:= \begin{pmatrix}
        d_1^2 & 0\\
        0 & d_2^2
    \end{pmatrix},
\end{align*} so that $A_{d_1,d_2}' = d_2^2$. We also write $A_{d_1,d_2}$ as $A_{1,2}$ when there is no confusion.
\item  For any $t \in \mathbb{R}$ and $A'$ as above, we define the measure $\gamma_{A', t}$ on $\mathbb{R}^{n-1}$ by 
    \begin{align*}
    d\gamma_{A', t}(x'):=e^{\frac{-\langle A(t,x'),(t,x')\rangle}{2}}d\mathcal{H}^{n-1}(x').
    \end{align*}
\item For any $A$ as above, we use $\sqrt{A}$ to denote the diagonal matrix such that $\sqrt{A}\sqrt{A} = A$.  We denote the eigenspace associated to its smallest eigenvalue of $\sqrt{A}$ by $V_{d_{\min}}(\sqrt{A})=V_{d_{1}}(\sqrt{A})$.
\item (Half-spaces) For $n \ge 2$, any $\omega \in V_{d_{1}}(\sqrt{A}) \cap \mathbb{S}^{n-1}$, and $\overline{s} \in \mathbb{R}$, we define 
\begin{align*}
    H_n(\omega, \overline{s}):=\{x \in \mathbb{R}^n: \omega \cdot x \le \overline{s}\}.
\end{align*}
\item (Mass and Perimeter) We shall use $\alpha \in (0,1)$ to denote the $\gamma_{A}$-measure of a set.  For any fixed $\alpha$, we shall associate $\overline{s} \in \mathbb{R}$ such that $\overline{s}$ satisfies
\begin{align*}
    \gamma_A(H_n(\omega, \overline{s}))&= \frac{\det\sqrt{A}}{(2\pi)^{n/2}}\int_{H_n(\omega, \overline{s})}e^{-\frac{1}{2}\langle Ax, x\rangle}dx\\
    & = \frac{1}{(2\pi)^{n/2}}\int_{H_n\left(\frac{\sqrt{A}^{-1}\omega}{|\sqrt{A}^{-1}\omega|}, d_1\overline{s}\right)}e^{-\frac{1}{2}|x|^2}dx =\phi(d_1\bar{s})=\alpha,
\end{align*}
and
\begin{align*}\text{Per}_{\gamma_A}(H_n(\omega,\bar{s}))=e^{-\frac{1}{2} \frac{\bar{s}^2}{\left|(\sqrt{A})^{-1} \omega\right|^2}} \frac{1}{|(\sqrt{A})^{-1} \omega|}=e^{-[\phi^{-1}(\alpha)]^2/2}d_1.\end{align*}
        \item (Subgraphs) Let $u:\mathbb{R}^{2} \rightarrow \mathbb{R}$.  We shall use the notation 
        \begin{align}\label{d:E_u def}
        E_u:= \{(x',t): t\le \overline{s} + u(x')\}.
    \end{align}
    \end{enumerate}
\end{definition}

We shall also use the following notation. For a curve $\Gamma:(-\delta, \delta) \rightarrow \mathbb{R}^2$, $(\Gamma(t))_{e_i} := \Gamma(t) \cdot e_i$ for $i=1,2$.  Furthermore, for vector-valued functions $\vec{\eta}: (-\delta, \delta) \rightarrow \mathbb{R}^2$, we shall use $\eta_i := \vec{\eta}\cdot e_i$ for $i=1,2$.

\subsection{Definitions from Geometric Measure Theory}

We let $\mathcal{H}^n$ denote the $n$-dimensional Hausdorff measure.

\begin{definition}
    Let $E \subset \mathbb{R}^n$ be a Borel set. 
    \begin{enumerate}
    \item We let $E^{(1)}$ and $E^{(0)}$ denote the collection of points of Lebesgue density $1$ and $0$, respectively.
    \item $E$ is a \textit{set of locally finite perimeter} in Euclidean space if for all compact sets $K \subset \mathbb{R}^n$,
    \begin{align*}
        \sup\Big\{\int_E\text{div}(\phi)dx: \phi \in C^\infty_c(K;\mathbb{R}^n), \ \norm{\phi}_{\infty}\le 1\Big\}<\infty.
    \end{align*}
    \item If $E$ is a set of locally finite perimeter in Euclidean space, then there exists a Radon measure $\norm{\partial E}$ and a $\norm{\partial E}$-measurable vector valued function $\vec{\eta}_E$ such that
    \begin{enumerate}
        \item For any open set $U\subseteq \mathbb{R}^n$,
        \begin{align*}
        \norm{\partial E}(U)=\sup\Big\{\int_E\text{div}(\phi)dx: \phi \in C^\infty_c(U;\mathbb{R}^n), \quad \norm{\phi}_{\infty}\le 1\Big\}.
    \end{align*}
    \item For any $\phi \in C^\infty_c(\mathbb{R}^n:\mathbb{R}^n)$,
    \begin{align*}
        \int_E\text{div}(\phi)dx = \int \phi \cdot \vec{\eta}_E d\norm{\partial E}.
    \end{align*}
    \end{enumerate}
    \item For a set of locally finite perimeter in Euclidean space, the \textit{reduced perimeter} of $E$ is defined to be
    \begin{align*}
        \partial^*E:=\Big\{x \in \mathbb{R}^n: \lim_{r \rightarrow 0^+}\frac{\int_{B_r(x)}\vec{\eta}_E d\norm{\partial E}}{\norm{\partial E}(B_r(x))} = \vec{\eta}_E(x) \quad \text{and } |\vec{\eta}_E(x)|=1\Big\}.
    \end{align*}
    That is, $\partial^*E$ are points of density with respect to $\norm{\partial E}$ with norm $1$.
    \item For any Borel set $E \subset \mathbb{R}^n$, the \textit{measure-theoretic boundary} of $E$ is defined by
    \begin{align*}
        \partial^ME:= \Big\{x \in \mathbb{R}^n:  \lim_{r\rightarrow 0+}\frac{\mathcal{H}^n(E \cap B_r(x))}{\omega_nr^n} \not = 0,1\Big\}
    \end{align*} 
    That is, this set includes points where the limit doesn't exist and points where the limit exists, but is not $0$ or $1$.
\end{enumerate}
\end{definition}

\begin{remark}
    If $E\subset \mathbb{R}^n$ is a set of locally finite perimeter in Euclidean space, then $\norm{\partial E} = \mathcal{H}^{n-1} \res \partial^*E= \mathcal{H}^{n-1} \res \partial^ME$ by De Giorgi's Structure theorem and Federer's Theorem (see \cite{evansgariepy} Theorem 5.15 and Lemma 5.5).
\end{remark}

\begin{definition}
For any Borel set $E \subset \mathbb{R}^n$, we define
\begin{enumerate}
    \item $\gamma_A(E)$ as in \eqref{e:gamma A mass def}, 
    \item $\operatorname{Per}_{\gamma_A}(E)$ as in \eqref{e:gamma A perimeter def},
    \item $\delta_{\gamma_A}(E)$ as in \eqref{e:gamma A deficit def},
    \item $\lambda_{\gamma_A}(E)$ as in \eqref{e:gamma A asymmetry def},
    \item and, for an open set $U\subseteq \mathbb{R}^n$, define the \textit{relative $A$-anisotropic Gaussian perimeter} of $E$ in $U$ by
\begin{align*}
\operatorname{Per}_{\gamma_{A}}(E; U) := \frac{\sqrt{\det(A)}}{(2\pi)^{\frac{n-1}{2}}}\int_{\partial^ME \cap U}e^{-\frac{1}{2}\langle Ax, x \rangle}d\mathcal{H}^{n-1}(x).
\end{align*}
\end{enumerate} 
\end{definition}
    
\begin{remark}
    Because the weight $e^{-\frac{1}{2}\langle Ax, x \rangle}$ is smooth and locally strictly positive, $\operatorname{Per}_{\gamma_A}(E)<\infty$ implied that $E$ is a set of locally finite perimeter in Euclidean space, and hence we may speak of the reduced boundary $\partial^*E$ and
    \begin{align*}
        \operatorname{Per}_{\gamma_{A}}(E; U) = \frac{\sqrt{\det(A)}}{(2\pi)^{\frac{n-1}{2}}}\int_{U}e^{-\frac{1}{2}\langle Ax, x \rangle}d\norm{\partial E}(x)
    \end{align*} for all open sets $U\subseteq \mathbb{R}^n$.
\end{remark}

In Section \ref{s:truncation}, we shall find it convenient to rescale so that $1 =d_1 \le d_2$.  

\begin{lemma}\label{l:rescaling d_1 is ok}
    Let $A_{1,2}$, $\alpha$, $\overline{s}$ be as in Definition \ref{d:standing assumptions}.  Then, by the change of variables $y = d_1x$,
    \begin{align*}
    \frac{\delta_{\gamma_{A_{1,2}}}(E)}{\lambda_{\gamma_{A_{1,2}}}(E)^2} = d_1\frac{\delta_{\gamma_{\frac{1}{d_1}A_{1,2}}}(d_1E)}{\lambda_{\gamma_{\frac{1}{d_1}A_{1,2}}}(d_1E)^2}.
    \end{align*}
\end{lemma}

By applications of standard theorems for sets of locally finite perimeter in Euclidean space, we have the following basic results.

\begin{theorem}[Basic Results]\label{t:weak compactness, LSC}
    Let $A$ be as in Definition \ref{d:standing assumptions}.
    \begin{enumerate}
        \item Any collection $\{E_i\}_i$ of sets of finite $\gamma_A$-perimeter such that $\sup_i \emph{Per}_{\gamma_A}(E_i)<\infty$ is weakly pre-compact.
        \item If $E_i \rightarrow E_0$ in $\gamma_A$-measure, then
        \begin{align*}
            \emph{Per}_{\gamma_A}(E_0) \le \liminf_{i \rightarrow \infty}\emph{Per}_{\gamma_A}(E_i).
        \end{align*}
    \end{enumerate}
\end{theorem}

\subsection{Symmetrization}

We use the notation that $\mathcal{B}(\mathbb{R}^n)$ denotes the set of all Borel sets in $\mathbb{R}^n$ and in general, $\mathcal{B}(X)$ denotes the set of all Borel sets in a subset $X \subset \mathbb{R}^n$.

\begin{definition}\label{d:n-1 symmetrizations}($(n-1)$-symmetrizations)
    Let $A$ be a matrix as in Definition \ref{d:standing assumptions}. Let $e_{1}, e_2$ be the eigenvectors associated to the two smallest eigenvalues $0<d_1^2\le d_2^2$, respectively. Define, \begin{align*}
S_{e_1, e_{2}}^{n-1}:\mathcal{B}(\mathbb{R}^n) \rightarrow \mathcal{B}(\mathbb{R}^n)
\end{align*}
by the following operation.  Let
\begin{align*}
\pi_{e_1}:\mathbb{R}^n \rightarrow \text{span}(e_1)
\end{align*}
be orthogonal projection.  Let $
t \in \text{span}(e_1)$.  We define $S_{e_2}^t: \mathcal{B}(\pi_{e_1}^{-1}(t)) \rightarrow \mathcal{B}(\pi_{e_1}^{-1}(t))$
be the symmetrization map defined by
\begin{align*}
& S^t_{e_2}(B):= \{t\}\times H_{n-1}(e_2, \overline{s}), \quad \text{s.t.}\quad  \gamma_{A', t}(H_{n-1}(e_2, \overline{s})) = \gamma_{A', t}(B)
\end{align*}

Finally, we define the $(n-1)$-symmetrization defined by $(e_1, e_2)$ by
\begin{align*}
    S_{e_1, e_2}^{n-1}(B) := \bigcup_{t \in \text{span}(e_1)}S^t_{e_2}(B \cap \pi_{e_1}^{-1}(t)).
\end{align*}
\end{definition}

The key property that we need about $(n-1)$-symmetrizations is the following theorem.

\begin{theorem}\label{t:symmetrization}\emph{($(n-1)$-symmetrizations are perimeter-decreasing)}
    Let $n\ge2$ and $A$ be a diagonal matrix in the coordinate basis $\{e_i\}_{i=1}^n$.  Let $E \subset \mathbb{R}^n$ be a set of $\gamma_A$-finite perimeter.  Then,
    \begin{align*}
        \operatorname{Per}_{\gamma_{A}}(S^{n-1}_{e_1, e_2}(E)) \le \operatorname{Per}_{\gamma_A}(E).
    \end{align*}
\end{theorem}
The proof of this result is omitted, as it follows arguments analogous to those in the isotropic setting and closely parallels the techniques used for one-dimensional symmetrizations; see, for instance, Section 4 and Theorem 7.1 of \cite{CianchiFuscoMaggiPratelli11}.

\subsection{Regularity theory}

\begin{definition}
Let $n \ge 2$ and $\Omega \subset \mathbb{R}^n$ be an open set.  We shall say that a set $E \subset \mathbb{R}^n$ is a \textit{Euclidean perimeter almost-minimizer} in $\Omega$ if for all $A \subset \subset \Omega$ there exists a $T \in (0,\text{dist}(A, \partial \Omega))$ and a function $\alpha: (0, T) \rightarrow [0,\infty)$ such that
\begin{align*}
    \Per(E;B_{r}(x)) \le \Per(F;B_r(x)) + \alpha(r)r^{n-1}
\end{align*}
for all $x \in A$ and all $r \in (0,T)$ and all sets of finite perimeter $F \subset \mathbb{R}^n$ such that $E \Delta F \subset \subset B_r(x)$.

Let $0<r_0, \Lambda<\infty$.  We shall say a set $E \subset \mathbb{R}^n$ is a $(\Lambda, r_0)$-\textit{Euclidean perimeter almost-minimizer} in $\Omega$ if for all $0<r\le r_0$ and all sets of finite perimeter $F \subset \mathbb{R}^n$ such that $E \Delta F \subset \subset B_r(x) \cap \Omega$, 
\begin{align*}
    \Per(E;B_{r}(x)) \le \Per(F;B_r(x)) + \Lambda \mathcal{H}^n(E \Delta F).
\end{align*}
\end{definition}

\begin{theorem}[Tamanini Convergence, \cite{tamanini1984regularity} Section 1.9]\label{t:tamanini}
Suppose that $\{E_i\}_i$ is a sequence of Euclidean perimeter almost-minimizers in $\Omega$, with uniform $\alpha:(0, T)\rightarrow [0,\infty)$ such that
\begin{enumerate}
    \item \textup{(Monotonicity)} the function $t^{-1}\alpha(t)$ is non-increasing in $(0,T).$
    \item \textup{(Integrability)}
    \begin{align*}
    \int_{0}^{T}\alpha(t)^{\frac{1}{2}}\frac{dt}{t} < 0.
    \end{align*}
\end{enumerate}
Let $x_i \in \partial E_i$, $E_0 \subset \mathbb{R}^n$ such that $E_i \rightarrow E_0$ in measure, and assume that $x_i \rightarrow x_0 \in \Omega$. Then, $x_0 \in \partial E_0$. Furthermore, if $x_0 \in \partial^*E_0$, then
\begin{align*}
    x_i \in \partial^*E_i
\end{align*}
for all sufficiently large $i \in \mathbb{N}$ and $\vec{\eta}_{E_i}(x_i) \rightarrow \vec{\eta}_{E_0}(x_0)$.
\end{theorem}

\begin{theorem}[Regularity for volume-constrained perimeter minimizers in anisotropic Gauss space]\label{t:regularity for VCPMs}
    Let $n=2$ and $A_{1,2}$ be a matrix as in Definition \ref{d:standing assumptions}(4).  Let $\Omega \subset \mathbb{R}^2$ be an open set with Lipschitz boundary and $E \subset \mathbb{R}^2$ satisfy 
\begin{align*}
\operatorname{Per}_{\gamma_{A_{1,2}}}(E; U) \le \operatorname{Per}_{\gamma_{A_{1,2}}}(F; U) 
\end{align*}
for all $F \subset \mathbb{R}^2$ such that $E \Delta F \subset \subset U$ and $\gamma_A(F \cap U) = \gamma_A(E \cap U)$. Then,
\begin{enumerate}
    \item \textup{(Regularity)} For every $x \in \partial^M E \cap U$ there is a neighborhood $B_{r_x}(x) \subset \subset U$ such that $\partial^ME \cap B_{r_x}(x)$ is the graph of a $C^\infty$ function.
    \item \textup{(First Variation)} For any $U' \subset \subset U$ such that $\partial^ME \cap U'$ is a graph, we parametrize $\partial^ME \cap U'$ by a unit-speed curve
    \begin{align*}
        \Gamma:(-\epsilon,\epsilon) \rightarrow U', \qquad \Gamma((-\epsilon,\epsilon)) = \partial^ME \cap U'.
    \end{align*}
    Then, for any $\phi\in C^1_c((-\varepsilon,\varepsilon))$ such that 
    \begin{align*}
        \int_{B_\varepsilon(0)} \phi(t) e^{-\frac{1}{2}\langle A_{1,2} \Gamma(t), \Gamma(t)\rangle}dt = 0 \qquad \textup{(Infinitesimally Volume Preserving)},
    \end{align*}
the parametrizing curve $\Gamma$ satisfies
\begin{align*}
        \int_{B_\varepsilon(0)} \Big(-\langle \vec{\eta}_E(t), \Gamma''(t)\rangle -\langle A_{1,2}\Gamma(t), \vec{\eta}_E(t) \rangle \Big) \phi(t)e^{-\frac{1}{2}\langle A_{1,2} \Gamma(t), \Gamma(t)\rangle}dt=0.
    \end{align*}
    \item \textup{(Constant Mean Curvature Equations)} Using the notation, above, there is a constant $C$ such that 
    \begin{align*}
    \langle \Gamma''(t), \vec{\eta}_E(t)\rangle=C -\langle A_{1,2}\Gamma(t), \vec{\eta}_E(t)\rangle
    \end{align*}
    for all $t \in (-\epsilon, \epsilon)$.
\end{enumerate}
\end{theorem}

\begin{proof}
First, we prove local $C^{1,\alpha}$ regularity.  The proof follows from demonstrating that $\gamma_{A_{1,2}}$-volume constrained $\gamma_{A_{1,2}}$-perimeter minimizers are $(\Lambda, r_0)$-perimeter minimizers for the unweighted, Euclidean perimeter, in the sense of \cite{Maggi_book12} Chapter 21.  The regularity statement holds from \cite{Maggi_book12} Theorem 21.8. 

The reduction to the Euclidean case follows from two observations.  First, by mirroring the standard proofs for perimeter almost-minimizers in Euclidean space, one obtains that for every $x_0 \in \partial^M E \cap U$, there exists a $0<r_0<\Lambda<\infty$ possibly depending upon $x_0$ such that
\begin{align*}
\operatorname{Per}_{\gamma_{A_{1,2}}}(E; B_{r_0}(x_0)) \le \operatorname{Per}_{\gamma_{A_{1,2}}}(F; B_{r_0}(x_0)) + \Lambda \gamma_{A_{1,2}}(E \Delta F).
\end{align*}
for any $F \subset \mathbb{R}^2$ such that $E \Delta F \subset \subset B_{r_0}(x_0)$. Specifically, one needs a version of \cite{Maggi_book12} Lemma 17.21 (Volume-fixing variations) in anisotropic Gauss space. Then, one follows the argument in \cite{Maggi_book12} Example 21.3 that volume-constrained perimeter minimizers are locally $(\Lambda, r_0)$-minimizers. Though the convenient formulas no longer hold because of the non-constant weight, the argument goes through with trivial modification because the weight $\gamma_{A_{1,2}}$ is continuous.

Therefore, we have that
\begin{align}\label{e:gauss to euclidean}\nonumber
&\Big(\min_{x \in B_{r_0}(x_0)}\gamma_{A_{1,2}}(x) \Big)  \norm{\partial E}( B_{r_0}(x_0))\\ \nonumber
& \le \operatorname{Per}_{\gamma_{A_{1,2}}}(E; B_{r_0}(x_0))\\ \nonumber
& \le \operatorname{Per}_{\gamma_{A_{1,2}}}(F; B_{r_0}(x_0)) + \Lambda \gamma_{A_{1,2}}(E \Delta F)\\ \nonumber
& \le \Big(\min_{x \in B_{r_0}(x_0)}\gamma_{A_{1,2}}(x)\Big)\norm{ \partial F}(B_{r_0}(x_0))\\ \nonumber
& \qquad + \Big(\max_{x \in B_{r_0}(x_0)}\gamma_{A_{1,2}}(x) - \min_{x \in B_{r_0}(x_0)}\gamma_{A_{1,2}}(x)\Big)\norm{\partial F}(B_{r_0}(x_0)) + \Lambda \gamma_{A_{1,2}}(E \Delta F)\\
& \le \Big(\min_{x \in B_{r_0}(x_0)}\gamma_{A_{1,2}}(x)\Big)\norm{\partial F}(B_{r_0}(x_0)) + Cr_0\norm{\partial F}(B_{r_0}(x_0)) + \Lambda' \mathcal{H}^n(E \Delta F).
\end{align}

Second, by shadowing the proof of upper density bounds for Euclidean perimeter (see \cite{Maggi_book12} Theorem 21.11) we obtain
\begin{align*}
\operatorname{Per}_{\gamma_{A_{1,2}}}(E; B_r(x)) \le Cr^{n-1}
\end{align*}
for any $B_r(x) \subset B_{r_0}(x_0)$.

Thus, to see that there exists a $\Lambda''<0$ such that for any $0<r<r_0$ and any $F \subset \mathbb{R}^n$ with $F \Delta E \subset \subset B_r(x) \cap B_{r_0}(x_0)$, 
\begin{align}\label{e:lambda r minimizer}
    \Per(E;B_{r}(x)) \le \Per(F;B_{r}(x)) + \Lambda'' r^n,
\end{align}
we observe that we only need to argue \eqref{e:lambda r minimizer} for $F$ such that $\norm{\partial F}(B_r(x)) \le \norm{\partial E}(B_{r}(x)) \le Cr^{n-1}$, since otherwise \eqref{e:lambda r minimizer} holds trivially. But, for such sets $F$, we have from \eqref{e:gauss to euclidean} that
\begin{align*}
\Per(E;B_{r}(x)) \le \Per(F;B_{r}(x)) + Cr^n.
\end{align*}
This gives the claim.

To see higher regularity of $\partial E$, by local graphical representation we see the regularity for $\partial E$ follows from the regularity for solutions $f:\mathbb{R}^{n-1} \rightarrow \mathbb{R}$ such that
\begin{align*}
    \text{div}\Big(e^{-\frac{1}{2}\langle A_{1,2} (x', f(x')), (x', f(x')) \rangle }\frac{\nabla f}{\sqrt{1+|\nabla f|^2}}\Big) = -d_n^2f.
\end{align*}
Since $f \in C^{1,\alpha}$ and $e^{-\frac{1}{2}\langle A (x', f(x')), (x', f(x'))\rangle }$ is smooth and locally bounded above and away from zero, standard Schauder estimates give $C^{\infty}$-regularity.

Now, since $E$ is locally a volume-constrained perimeter minimizer in anisotropic Gauss space, we use the first variation formula of weighted perimeter to prove the constant mean curvature equations.  Define the normal variations $\Gamma_\tau(t):= \Gamma(t)+ \tau \phi(t)\vec{\eta}_E(t)$.
    \begin{align*}
        0&=\frac{d}{d\tau} \Big(\int |\Gamma_{\tau}'(t)|e^{-\frac{1}{2}\langle A_{1,2} \Gamma_{\tau}(t), \Gamma_{\tau}(t)\rangle}dt\Big)\Big|_{\tau =0}\\
        & = \int \Big(\langle \vec{\eta}'_E(t), \Gamma'(t)\rangle -\langle A_{1,2}\Gamma(t), \vec{\eta}_E \rangle \Big) \phi(t)e^{-\frac{1}{2}\langle A_{1,2} \Gamma(t), \Gamma(t)\rangle}dt \\
        & = \int \Big(-\langle \vec{\eta}_E(t), \Gamma''(t)\rangle -\langle A_{1,2}\Gamma(t), \vec{\eta}_E \rangle \Big) \phi(t)e^{-\frac{1}{2}\langle A_{1,2} \Gamma(t), \Gamma(t)\rangle}dt,
    \end{align*}
where we have used that $|\Gamma'(t)|=1$ in the second line and  $\frac{d}{dt}\langle \vec{\eta}_E(t), \Gamma'(t)\rangle \equiv 0$ in the last line. 

Choosing any two $t_1, t_2 \in (-\epsilon, \epsilon)$, we can choose $\phi_1, \phi_2 \in C^1_c((-\epsilon, \epsilon))$ and, without loss of generality, that $\int \phi_2(x) d\norm{\partial E}_{\gamma_A}(x)>0$.  Thus, we can define
\begin{align*}
    \phi= \phi_1 - \frac{\int \phi_1 d\norm{\partial E}_{\gamma_{A_{1,2}}}}{\int \phi_2 d\norm{\partial E}_{\gamma_{A_{1,2}}}}\phi_2. 
\end{align*}
By our first variation formula, we obtain
\begin{align*}
    &\frac{\displaystyle\int \Big(\langle \vec{\eta}_E(t), \Gamma''(t)\rangle +\langle A_{1,2}\Gamma(t), \vec{\eta}_E \rangle \Big) \phi_1(t) e^{-\frac{1}{2}\langle A_{1,2} \Gamma(t), \Gamma(t)\rangle}dt}{\displaystyle \int \phi_1(t) e^{-\frac{1}{2}\langle A_{1,2} \Gamma(t), \Gamma(t)\rangle}dt}\\
    & \qquad = \frac{\displaystyle\int \Big(\langle \vec{\eta}_E(t), \Gamma''(t)\rangle +\langle A_{1,2}\Gamma(t), \vec{\eta}_E \rangle \Big) \phi_2(t) e^{-\frac{1}{2}\langle A_{1,2} \Gamma(t), \Gamma(t)\rangle}dt}{\displaystyle \int \phi_2(t) e^{-\frac{1}{2}\langle A_{1,2} \Gamma(t), \Gamma(t)\rangle}dt}.
\end{align*}

Since the left-hand side and the right-hand side are independent of each other, these quantities must be constant for all $t_1, t_2 \in (-\epsilon, \epsilon)$.  Unwinding this, we see that $\Gamma$ satisfies the following ODE:
    \begin{align}\label{e:ODE def}
        \langle \vec{\eta}_E(t), \Gamma''(t)\rangle = C_{\vec{\eta}_E} - \langle A_{1,2} \Gamma(t), \vec{\eta}_E(t) \rangle,
    \end{align}
    where the constant $C_{\vec{\eta}_E}$ depends upon the choice of orienting normal $\vec{\eta}_E$ according to the rule $-C_{\vec{\eta}_E} =  C_{-\vec{\eta}_E}$.
\end{proof}

\begin{corollary}[Applying Theorem \ref{t:tamanini}]\label{c:tamanini convergence}
    If $\{E_i\}_i \subset \mathbb{R}^2$ is a sequence of volume-constrained minimizers in $(\mathbb{R}^n, \gamma_A)$ which converge in $\gamma_A$-measure to $H(\omega, \overline{s})$, then for all sufficiently large $i \in \mathbb{N}$ and all $0<R$,
    \begin{align*}
        \partial E_i \cap [-R,R]^2 = \emph{graph}_{[-R,R]}(u_i)
    \end{align*}
    and $u_i \rightarrow 0$ in $C^1([-R,R])$ as $i \rightarrow \infty$.
\end{corollary}

\subsection{Outline of the Proof of Theorem \ref{t:technical theorem}}\mbox{}\\
\noindent \textbf{Step 1: Reduction to ODEs.} Let $A$ be a symmetric, positive definite matrix as in Definition \ref{d:standing assumptions}.  Given a recovery sequence $\{E_k\}_k$ for $Q(H_n(e_1, \overline{s}))$, the first step is to replace $\{E_k\}_k$ by a sequence of better sets, keeping track at every stage that we decrease the $\gamma_A$-deficit and do not loose too much in the $\gamma_A$-asymmetry.  We accomplish this in the following steps.
\begin{enumerate}
    \item (Reduction to $n=2$) First, we use $(n-1)$-symmetrization to replace $\{E_k\}_k \subset \mathbb{R}^n$ by sets $\{F_k\}_k \subset \mathbb{R}^2$.  This is done in Lemma \ref{l:n-1 symmetrization} and Remark \ref{r:new assumptions, mid symmetrization}. This step gives Theorem \ref{t:technical theorem}(1abcd).  A key advantage of symmetrizing first is that 
    \begin{align*}
        \lambda_{\gamma_A}(E_k) = \gamma_{A_{1,2}}(F_k \Delta H_2(e_1, \overline{s})).
    \end{align*}
    From this point on we only keep track of the \textit{relative} asymmetries $\gamma_{A_{1,2}}(F_k \Delta H_2(e_1, \overline{s}))$, and we abandon entirely the Fraenkel-type anisotropic Gaussian asymmetry $\lambda_{\gamma_{A_{1,2}}}(F_k)$.  The trade-off for this simplicity is that we lose control on barycenter of the sets $F_k$, which complicates the Small Perturbation Regime argument later.
    \item (Reduction to subgraphs of monotone functions) Then, we symmetrize the $\{F_k\}_k$ again to produce $\{H_k\}_k \subset \mathbb{R}^2$ such that
    \begin{align*}
        H_k = \text{subgraph}_{\partial H_2(e_1, \overline{s})}(u_k), \qquad u_k \text{ monotonically increasing.}
    \end{align*}
    At this step, we must be careful that the relative asymmetry does not drop uncontrollably because of cancellation from the symmerization procedure (see Corollary \ref{c:second symmetrization}). The key to overcoming this is the difference in sharp powers between the $n=1$ and $n \ge 2$ case.
    \item (Reduction to relative volume-constrained perimeter minimizers) Finally, we replace our $\{H_k\}_k$ by sets $\{I_k\}_k \subset \mathbb{R}^2$ which are relative volume-constrained $\gamma_{A_{1,2}}$-perimeter-minimizers in $\{x_2 \le 0\}$ and $\{x_2 \ge 0\}$. Here the monotonicity of $u_k$ and control on the relative asymmetry obtained in the previous step play a decisive role in controlling the relative asymmetry of the $\{I_k\}_k$. 
\end{enumerate}
Section \ref{s:reduction_symmetrization} is dedicated to Step 1.

\noindent\textbf{Step 2: Studying the ODE gives Gradient Control.} The next step is to show that these $\{I_k\}_k$ are \textit{almost} the sets $\tilde{E}_k$ promised in Theorem \ref{t:technical theorem}.  This is based upon the study of the family of ODEs that the perimeters of the $\{I_k\}_k$ satisfy, coming from the constant mean curvature equations in weighted Euclidean space (see Theorem \ref{t:regularity for VCPMs}). More specifically, because the $\{I_k\}_k$ are volume-constrained $\gamma_{A_{1,2}}$-perimeter minimizers in $\{x_2 \le 0\}$ and $\{x_2 \ge 0\}$, the convergence results of \cite{tamanini1984regularity} imply that for any $0<\epsilon<R<\infty$, for all sufficiently large $k \in \mathbb{N}$
\begin{align*}
    I_k \cap \{(x_1, x_2): |x_2| \in (\epsilon, R)\} = \text{subgraph}_{\partial H_2(e_1, \overline{s})} u_k, \qquad \lim_{k \rightarrow \infty}\norm{u_k}_{C^1((-R, -\epsilon) \cup (\epsilon, R))}=0.
\end{align*}
Thus, the second step consists in obtaining height bounds and gradient bounds in the Center Strip $\{(x_1, x_2): |x_2|\le \epsilon\}$ and partial gradient control in the Unbounded Region $\{(x_1, x_2): |x_2|\ge R\}$.  The main ingredients are:
\begin{enumerate}
    \item (Preliminary Curvature Estimates) Employing some auxiliary vector fields to compare with tangent vectors to $\partial I_k$, we obtain some basic curvature estimates on $\partial I_k$ based upon the slope of $u_k$.  These vector fields give us a special ``height" $\ell_k$ such that the quantity $\frac{|x_1-\ell_k|}{|x_2|}$ plays an important role in our analysis.  This is the focus of Section \ref{s:preliminaries for gradient control}.  The main curvature estimates are given in Lemma \ref{l:curvature bounds by angle}. 
    \item (Gradient and Height Control in the Center Strip) We use the curvature estimates of Lemma \ref{l:curvature bounds by angle} and a topological argument to obtain Lipschitz control (and therefore height control) in the Center Strip.  See Lemma \ref{l:center strip control}.
    \item (Partial Gradient Control in the Unbounded Region) All our estimates on $u_k'$ in the unbounded regions are based upon a simple comparison argument.  If the slope is too far from our comparison vector fields, then the curvature estimates of Lemma \ref{l:curvature bounds by angle} force $\partial I_k$ to circle back upon itself and violate the fact that $I_k$ is the subgraph of a function.  First, in Lemma \ref{l:preliminary gradient control Comparison} and Corollary \ref{c:lipschitz control in cones} we prove that $u_k'$ is Lipschitz in certain conical regions $\Omega_{k, \varepsilon_0} \approx \{\frac{|x_1-\ell_k|}{|x_2|} \le \varepsilon_0\}$ (see Definition \ref{r:Omega def, existence of two sols}).  Next, in Porism \ref{porism graphicality}, we obtain one-sided estimates on the gradient outside the conical regions $\Omega_{k,\varepsilon_0}$.  Finally, we prove fine estimates on the gradient in very narrow strips $\{|x_1-\ell_k| <<1\}$ in Theorem \ref{t:finer estimates on gradient}.    
\end{enumerate}

\noindent \textbf{Step 3: Truncation.}
Despite all the estimates on $u_k'$ for sufficiently large $k \in \mathbb{N}$, Step 2 does not prove full Lipschitz control and height estimates in the Unbounded Region.  That is, it does not show that $I_k$ are eventually in the Small Perturbation Regime.  Therefore, to finish the reduction we need to truncate.  The estimates on the change in $\gamma_{A_{1,2}}$-mass (Lemma \ref{l:mass error is negligible}) and the change in $\gamma_{A_{1,2}}$-perimeter (Lemma \ref{l:perimeter error bound}) under the truncation operation rely essentially upon the results of Theorem \ref{t:finer estimates on gradient} and an auxiliary ODE (See Lemma \ref{l:separation of exit times}). 

We put the results of Section \ref{s:preliminaries for gradient control}, Section \ref{s:gradient control}, and Section \ref{s:truncation} together to prove Theorem \ref{t:technical theorem}(2abcd) in subsection \ref{ss:Ttheorem 2abcd}.

\section{Proof of Theorem \ref{t:technical theorem}: Reduction to the graphical case}\label{s:reduction_symmetrization}
    
First, we prove Theorem \ref{t:technical theorem}(1).

\begin{lemma}[Proof of Theorem \ref{t:technical theorem}(1)]\label{l:n-1 symmetrization}
    Let 
    \begin{enumerate}
        \item $A$ be a positive-definite, symmetric matrix as in Definition \ref{d:standing assumptions}(3).
        \item $S^{n-1}_{e_1, e_2}$ be the $(n-1)$-symmetrization defined in Definition \ref{d:n-1 symmetrizations}.
    \end{enumerate} 
    Then, for every $G\subset \mathbb{R}^n$, after a possible rotation, there exists a set $\tilde{G}\subset \mathbb{R}^2$ such that
    \begin{align*}
        S^{n-1}_{e_1, e_2}(G)= \tilde{G}\times \mathbb{R}^{n-2} 
    \end{align*} satisfies the claims of Theorem \ref{t:technical theorem}(1).
\end{lemma}

\begin{proof}
    That $\tilde{G}\subset \mathbb{R}^2$ exists follows from Theorem \ref{t:symmetrization}.  Fubini's theorem gives Theorem \ref{t:technical theorem}(1ad).  Theorem \ref{t:symmetrization} gives Theorem \ref{t:technical theorem}(1b).  If $A$ satisfies $d_1^2 <d_2^2$, then $H_n(e_1, \overline{s}), H_n(e_1, -\overline{s})$ are the only two competitors for $\lambda_{\gamma_{A}}(G)$.  Since by construction, $H_n(e_1, \pm\overline{s}) =  S^{n-1}_{e_1, e_2}(H_n(e_1, \pm\overline{s}))$ and
    \begin{align*}
        \gamma_A(G \cap H_n(e_1, \pm\overline{s})) = \gamma_{A}(S^{n-1}_{e_1, e_2}(G) \cap H_n(e_1, \pm\overline{s}))\\
        \gamma_A(G \cap H_n(e_1, \pm\overline{s})^c) = \gamma_{A}(S^{n-1}_{e_1, e_2}(G) \cap H_n(e_1, \pm\overline{s})^c),
    \end{align*}
    Fubini's theorem gives us Theorem \ref{t:technical theorem}(1c), as well.

    The only case that remains is $d_1^2 = d_2^2$.  In this case, we may need to rotate $G$ such that $H_n(e_1,\overline{s})$ is one of the best-approximating half-planes which realize $\lambda_{\gamma_A}(G)$. Then, Theorem \ref{t:technical theorem}(1c) follows immediately. 
\end{proof}

\begin{remark}\label{r:new assumptions, mid symmetrization}
    Now, we may assume that $\{E_k\}_k \subset \mathbb{R}^n$ is a recovery sequence for $Q(H_{n}(e_1, \overline{s}))$.  We let $\{F_k\}_k\subset \mathbb{R}^2$ be the sequence obtained from $\{E_k\}_k$ by applying Lemma \ref{l:n-1 symmetrization}. That is, we define $$F_k := \tilde{G}_k \quad \text{ for } G_k = E_k.$$  Moreover, by reflection if necessary, we may assume that
    \begin{align*}
        \lambda_{\gamma_A}(E_k) = \gamma_{A_{1,2}}(F_k \Delta H_2(e_1, \overline{s})).
    \end{align*}

Note that, by Fubini's theorem, the slice function
\begin{align*}
    m_{F_k}(t):= \gamma_{A', t}(F_k \cap \{(x_1 = t)\})
\end{align*}
satisfies $m_{F_k}\in BV(\mathbb{R}, \gamma_{A'_{1,2}})$ (see Definition \ref{d:standing assumptions}(5) for the notation).  By the structure theorem for BV functions in one dimension, if we write $F_k$ as the subgraph of a function $w_k: \{x_2=0\} \rightarrow [-\infty, \infty]$, then for every $s \in \mathbb{R}$, the one-sided limits
\begin{align*}
w_k(s^-):= \lim_{x_1 \rightarrow s^-} w_k(x_1), \quad w_k(s^+):= \lim_{x_1 \rightarrow s^+} w_k(x_1)
\end{align*}
exist, though they may be infinite.
\end{remark}    

The next step is to symmetrize the sets $F_k \subset \mathbb{R}^2$ again, this time in the $e_1$-direction (in the notation of Definition \ref{d:n-1 symmetrizations}, this is $S_{e_2, -e_1}^1$-symmetrization). However, here we need to be careful to guarantee that for sufficiently large $k\in \mathbb{N}$
\begin{align*}
    \gamma_{A_{1,2}}(F_k\Delta H_2(e_1, \overline{s})) \approx \gamma_{A_{1,2}}(S_{e_2,-e_1}^1(F_k) \Delta H_2(e_1, \overline{s})).
\end{align*}
Obviously, the danger is that mass in $F_k$ that contributes to $\gamma_{A_{1,2}}(F_k \setminus H_2(e_1, \overline{s}))$ may be pushed by the symmetrization to cancel out a hole that contributed to $\gamma_{A_{1,2}}(H_2(e_1, \overline{s}) \setminus F_k)$. In Lemma \ref{t:can symm vertically 1}, we prove that as $k \rightarrow \infty$, the amount of mass which cancels out in $S^1_{e_2, -e_1}$-symmetrization is asymptotically small compared to $\lambda_{\gamma_A}(E_k)$.

First, we make the following observation.

\begin{remark}\label{r:Q is finite and consequences}
   Let $A, A_{1,2}$ be as in Theorem \ref{t:technical theorem}. Let $\alpha \in (0,1/2]$.  Then, $Q(H_n(e_1, \overline{s}))<\infty$.  In particular, if $\{E_k\}_k$ is a recovery sequence for $Q(H_n(e_1, \overline{s}))$ and $\{F_k\}_k$ is the sequence obtained from $\{E_k\}_k$ by Lemma \ref{l:n-1 symmetrization} and Remark \ref{r:new assumptions, mid symmetrization}, then $$Q(H_n(e_1, \overline{s}))<\infty.$$  As an immediate consequence, we have that
   \begin{align}\label{e:weaker ratio bounds}
       \liminf_{k \rightarrow \infty}\frac{\delta_{\gamma_{A_{1,2}}}(F_k)}{\gamma_{A_{1,2}}(F_k \Delta H_2(e_1, \overline{s}))}=0.
   \end{align}
We will use \eqref{e:weaker ratio bounds} in Lemma \ref{t:can symm vertically 1}.  But, we defer the demonstration that $Q(H_n(e_1, \overline{s}))<\infty$ until Lemma \ref{l:proof of remark Q finite}.
\end{remark}

\begin{definition}(Valleys and Towers)
For $F_k\subset \mathbb{R}^2$, $w_k:\{x_2=0\} \rightarrow [-\infty, \infty]$ as in Remark \ref{r:new assumptions, mid symmetrization}, we make the following definitions.

Let $s^\pm \in \mathbb{R}$ be defined by
\begin{align*}
    s^+ &:= \inf \{s \in [\overline{s}, \infty)\,:\, \mathcal{H}^{1}(F_k^{(1)} \cap \{x_1 = s\} \cap \{x_2<0\}) = 0\}\\
    &= \inf \{s \in [\overline{s}, \infty)\,:\, \min\{u_k(s^-), u_k(s^+)\} \le \overline{s} \}\\
    s^- &:= \sup \{s \in (-\infty,\overline{s}]\,:\, \mathcal{H}^{1}(F_k^{(0)} \cap \{x_1 = s\} \cap \{x_2>0\}) = 0\}\\
    &= \sup \{s \in (-\infty,\overline{s}]\,:\, \min\{u_k(s^-), u_k(s^+)\} \ge \overline{s} \}.
\end{align*}

We define $T_k \subset F_k^{(1)}$ by
\begin{align*}
    T_k := F_k^{(1)} \cap \{x_1 > s^+\} \cap \{x_2 <0\}.
\end{align*}
We shall think of $T_k$ as the collection of \textit{towers} of $F_k$.

We define the subset $V_k \subset F_k^{(0)}$ by 
\begin{align*}
    V_k := F_k^{(0)} \cap \{x_1 < s^-\} \cap \{x_2 >0\}.
\end{align*}
We shall think of $V_k$ as the collection of \textit{valley} of $F_k$. Both towers and valleys are what we should think of as \textit{overhang mass}.  We decompose $F_k$ into 
\begin{align*}
    F_k := G_{k, 1} \cup G_{k,2} \cup T_{k},
\end{align*}
where
\begin{align*}
    G_{k,1}&:= \{(x_1,x_2) \in F_k^{(1)}:  x_2 >0 \}\\
    G_{k,2}&:= \{(x_1,x_2) \in F_k^{(1)} \setminus T_{k}:  x_2 < 0 \}.
\end{align*}
\end{definition}

\begin{lemma}[Strong separation from the $1$-dimensional case]
 \label{t:can symm vertically 1}
    Let $F_k\subset \mathbb{R}^2$ be as in Remark \ref{r:new assumptions, mid symmetrization}.  Then, for any $0<c<1/8$ there exists a $k(c) \in \mathbb{N}$ such that for all $k \ge k(c)$
    \begin{align*}
        \gamma_{A_{1,2}}(T_{k} \cup V_{k}) \le c \gamma_{A_{1,2}}(F_k \Delta H_2(e_1, \overline{s})).
    \end{align*}
\end{lemma}

\begin{proof}
We argue by contradiction.  Suppose that there existed a sequence $\{F_k\}_k$ obtained from a recovery sequence $\{E_k\}_k \subset \mathbb{R}^n$ for $Q(H_n(e_1, \overline{s}))$ by Lemma \ref{l:n-1 symmetrization} such that 
\begin{align*}
       \gamma_{A_{1,2}}(F_k \Delta H_2(e_1, \overline{s}))\ge     \gamma_{A_{1,2}}(T_{k} \cup V_{k}) > c \gamma_{A_{1,2}}(F_k \Delta H_2(e_1, \overline{s}))
\end{align*}
for some $0<c<1/8$ and all sufficiently large $k \in \mathbb{N}$.  

For ease of notation, we let $\lambda_k := \gamma_{A_{1,2}}(F_k \Delta H_2(e_1, \overline{s}))$. Recall that $\gamma_{A_{1,2}}(F_k)=\alpha = \gamma_{A_{1,2}}(H_2(e_1, \overline{s}))$ and, because of our $(n-1)$-symmetrization, $\gamma_{A_{1,2}}(F_k \cap \{x_2 >0\}) \ge \gamma_{A_{1,2}}(F_k \cap \{x_2\le 0\})$.  Therefore, under the above assumption, we make the estimate
\begin{align*}
    \gamma_{A_{1,2}}(G_{k,1} \cup V_{k})\in \left(\frac{\alpha}{2}, \frac{\alpha+2\lambda_k}{2}\right), \qquad 
    \gamma_{A_{1,2}}(G_{k,2}) \in \left(\frac{\alpha-2\lambda_k}{2}, \frac{\alpha}{2}\right). 
\end{align*}
Thus, by the solution to the \textit{relative Isoperimetric problem} in the half-spaces $\{x_2>0\}$ and $\{x_2<0\}$ (using \eqref{unique minimizer} and a standard reflection argument), we see that 
\begin{align}\label{e:G1 and B2}
\operatorname{Per}_{\gamma_{A_{1,2}}}(G_{k,1} \cup V_{k};\{x_2>0\})& \ge \frac{d_1}{2} \min\{e^{-\frac{1}{2}\overline{s}^2}, e^{-\frac{1}{2}\phi^{-1}(\alpha +2\lambda_k)}\}\\ \label{e:G2}
\operatorname{Per}_{\gamma_{A_{1,2}}}(G_{k,2};\{x_2<0\}) & \ge \frac{d_1}{2}\min\{e^{-\frac{1}{2}[\phi^{-1}(\alpha - 2\lambda_k)]^2}, e^{-\frac{1}{2}|\overline{s}|^2}\}\\ \label{e:B1 and B2}
\operatorname{Per}_{\gamma_{A_{1,2}}}( T_{k};\{x_2<0\}) + \operatorname{Per}_{\gamma_{A_{1,2}}}(V_{k};\{x_2>0\}) & \ge d_1 e^{-\frac{1}{2}[\phi^{-1}(1-\frac{c}{2}\lambda_k)]^2}.    
\end{align}

Now, we estimate the following ratio
\begin{align*}
\frac{\delta_{\gamma_{A_{1,2}}}(F_k)}{\lambda_k} & \ge \frac{ \operatorname{Per}_{\gamma_{A_{1,2}}}(G_{k,1} \cup V_{k}; \{x_2>0\}) -  \operatorname{Per}_{\gamma_{A_{1,2}}}(H_2(e_1, \overline{s});\{x_2>0\})}{\lambda_k}\\
    & \qquad + \frac{\operatorname{Per}_{\gamma_{A_{1,2}}}(G_{k,2};\{x_2<0\}) - \operatorname{Per}_{\gamma_{A_{1,2}}}(H_2(e_1, \overline{s});\{x_2<0\})}{\lambda_k}\\
   & \qquad + \frac{\operatorname{Per}_{\gamma_{A_{1,2}}}(T_{k};\{x_2<0\}) + \operatorname{Per}_{\gamma_{A_{1,2}}}(V_{k};\{x_2>0\})}{\lambda_k}.
\end{align*}
We take the limit of the first two terms. By \eqref{e:G1 and B2}
\begin{align*}
    &\liminf_{k \rightarrow \infty}\frac{\operatorname{Per}_{\gamma_{A_{1,2}}}(G_{k,1}\cup V_{k};\{x_2>0\}) - \operatorname{Per}_{\gamma_{A_{1,2}}}(H_2(e_1, \overline{s});\{x_2>0\})}{\lambda_k}\\
    & \qquad \ge \liminf_{k \rightarrow \infty}\frac{d_1}{2}\frac{\left( \min\{e^{-\frac{1}{2}\overline{s}^2}, e^{-\frac{1}{2}\phi^{-1}(\alpha +2\lambda_k)}\} - e^{-\frac{1}{2}\overline{s}^2}\right)}{\lambda_k}.
\end{align*}
If $\alpha<\frac{1}{2}$ then $\overline{s}<0$ and $\min\{e^{-\frac{1}{2}\overline{s}^2}, e^{-\frac{1}{2}\phi^{-1}(\alpha +2\lambda_k)}\} = e^{-\frac{1}{2}\overline{s}^2}$ for all sufficiently large $k \in \mathbb{N}$.  In particular, 
\begin{align*}
    &\liminf_{k \rightarrow \infty}\frac{\operatorname{Per}_{\gamma_{A_{1,2}}}(G_{k,1}\cup V_{k};\{x_2>0\}) - \operatorname{Per}_{\gamma_{A_{1,2}}}(H_2(e_1, \overline{s});\{x_2>0\})}{\lambda_k} \ge 0.
\end{align*} 
If $\alpha =0$, then $\overline{s}=0$ and $\min\{e^{-\frac{1}{2}\overline{s}^2}, e^{-\frac{1}{2}\phi^{-1}(\alpha +2\lambda_k)}\} = e^{-\frac{1}{2}\phi^{-1}(\alpha +2\lambda_k)}$.  Taking derivatives, we immediately see that, again, 
\begin{align*}
    &\liminf_{k \rightarrow \infty}\frac{\operatorname{Per}_{\gamma_{A_{1,2}}}(G_{k,1}\cup V_{k};\{x_2>0\}) - \operatorname{Per}_{\gamma_{A_{1,2}}}(H_2(e_1, \overline{s});\{x_2>0\})}{\lambda_k} \ge 0.
\end{align*} 

To estimate the next term, we note that since $\overline{s}\le 0$, $\min\{e^{-\frac{1}{2}[\phi^{-1}(\alpha - 2\lambda_k)]^2}, e^{-\frac{1}{2}|\overline{s}|^2}\} = e^{-\frac{1}{2}[\phi^{-1}(\alpha - 2\lambda_k)]^2}$.  Therefore, we use \eqref{e:G2} to estimate
\begin{align*}
& \liminf_{k \rightarrow \infty}\frac{\operatorname{Per}_{\gamma_{A_{1,2}}}(G_{k,2};\{x_2<0\}) - \operatorname{Per}_{\gamma_{A_{1,2}}}(H_2(e_1,\overline{s});\{x_2<0\})}{\lambda_k}\\
& \qquad \ge \lim_{k \rightarrow \infty}\frac{d_1}{2}\frac{e^{-\frac{1}{2}[\phi^{-1}(\alpha - 2\lambda_k)]^2}- e^{-\frac{1}{2}[\phi^{-1}(\alpha)]^2}}{\lambda_k}\\
&\qquad = -d_1 \cdot \frac{d}{dt}|_{t=\alpha}\Big(\frac{e^{-\frac{1}{2}[\phi^{-1}(t)]^2}}{2}\Big) = -d_1 \cdot \Big(\frac{e^{-\frac{1}{2}[\phi^{-1}(\alpha)]^2}}{2}\Big)\Big(-2[\phi^{-1}(\alpha)] \cdot \frac{d}{dt}|_{t=\alpha}\phi^{-1}(t) \Big) = \overline{s}d_1^2.
\end{align*}
In particular, since $\alpha \in (0,1/2]$, the quantity $\overline{s}d_1^2$ is finite.

In exactly the same manner, we use \eqref{e:B1 and B2} on the final term, from which we shall derive our contradiction.
\begin{align*}
    & \liminf_{k \rightarrow \infty} \frac{\operatorname{Per}_{\gamma_{A_{1,2}}}(T_{k};\{x_2>0\}) + \operatorname{Per}_{\gamma_{A_{1,2}}}(V_{k};\{x_2<0\})}{\lambda_k}\\
    & \qquad \ge d_1 \cdot \liminf_{k \rightarrow \infty}\frac{e^{-\frac{1}{2}[\phi^{-1}(1-\frac{c}{2}\lambda_k)]^2}}{\lambda_k}\\
    &\qquad =d_1 \cdot \lim_{\lambda \rightarrow 0^+}e^{-\frac{1}{2}[\phi^{-1}(1-\frac{c}{2}\lambda)]^2}\Big(-\phi^{-1}(1-\frac{c}{2}\lambda) \cdot [\phi^{-1}]'(1-\frac{c}{2}\lambda)(-\frac{c}{2})\Big)\\
    &\qquad =\frac{d_1 c}{2} \cdot \lim_{\lambda \rightarrow 0^+}\phi^{-1}(1-\frac{c}{2}\lambda) = +\infty,
\end{align*}
where to obtain the penultimate step we have used L'Hopital's Rule.

Thus, we have a contradiction.  By Remark \ref{r:Q is finite and consequences}, $\lim_{k \rightarrow \infty}\frac{\delta_{\gamma_{A}}(E_k)}{\lambda_{\gamma_A}(E_k)^2} := Q(H_n(e_1, \overline{s}))<\infty$. But, if $\gamma_{A_{1,2}}(T_{k} \cup V_{k}) \ge c\lambda_k,$ then
\begin{align*}
        \lim_{k \rightarrow \infty}\frac{\delta_{\gamma_{A}}(E_k)}{\lambda_{\gamma_A}(E_k)} \ge \liminf_{k \rightarrow \infty}\frac{\delta_{\gamma_{A_{1,2}}}(F_k)}{\gamma_{A_{1,2}}(F_k \Delta H_2(e_1,\overline{s}))}= +\infty.
\end{align*}
\end{proof}

As an immediate corollary, we have the following statement.

\begin{corollary}\label{c:small overhang measure}
    Let $A, A_{1,2},F_k$ be as in Remark \ref{r:new assumptions, mid symmetrization}.  Let $0<c<1/8$.  Then, one of the following holds for all sufficiently large $k \in \mathbb{N}$:
    \begin{enumerate}
        \item Either $\gamma_{A_{1,2}}(F_k \cap \{(x_1,x_2): x_1>\overline{s}, x_2<0\}) < c\lambda_{\gamma_{A}}(E_k)$.
        \item Or, $\gamma_{A_{1,2}}(F_k^c \cap \{(x_1,x_2): x_1<\overline{s}, x_2>0\}) < c\lambda_{\gamma_{A}}(E_k)$.
    \end{enumerate}
\end{corollary}

\begin{proof}
We claim that it is not possible for both
\begin{align*}
    \gamma_{A_{1,2}}\Big((F_k \setminus T_{k}) \cap \{x_1>\overline{s}, x_2<0\} \Big)>0, \quad \text{and }\quad  
    \gamma_{A_{1,2}}\Big((F_k^c\setminus V_{k}) \cap \{x_1<\overline{s}, x_2>0\} \Big)>0.
\end{align*}
Suppose, for example, that $F_k$ satisfies 
    \begin{align*}
        \gamma_{A_{1,2}}\Big((F_k \setminus T_{k}) \cap \{x_1>\overline{s}, x_2<0\} \Big) >0.
    \end{align*}
Then, the set $(F_k \setminus T_{k}) \cap \{x_1>\overline{s}, x_2<0\}$ is not composed of towers of $F_k$ and it must be that 
\begin{align*}
    \mathcal{H}^{n-1}\Big(\partial^*(F_k \setminus T_{k} \cap \{x_1>\overline{s}, x_2<0\}) \cap F_k^{(1)}\Big)>0.
\end{align*}
However, in this case, because $F_k$ is a $S_{e_1,e_2}^{n-1}$-symmetrized set, this implies that
    \begin{align*}
        \{x_1 = \overline{s}, x_2>0\} \subset F_k^{(1)}. 
    \end{align*}
    Thus, in particular,
    \begin{align*}
        F_k^c \cap \{x_1<\overline{s}, x_2>0\} = V_k,
    \end{align*}
and hence 
\begin{align*}
\gamma_{A_{1,2}}\Big((F_k^c \setminus V_{k}) \cap \{x_1>\overline{s}, x_2<0\} \Big) =0.
\end{align*}
A symmetric argument holds in the case $\gamma_{A_{1,2}}\Big((F_k^c \setminus V_{k}) \cap \{x_1>\overline{s}, x_2<0\} \Big)>0,$ implying that  
\begin{align*}
\gamma_{A_{1,2}}\Big((F_k \setminus T_{k}) \cap \{x_1>\overline{s}, x_2<0\} \Big) =0.
\end{align*}
\end{proof}

As a corollary of Lemma \ref{t:can symm vertically 1} and Corollary \ref{c:small overhang measure}, we may symmetrize the $F_k$ without losing too much asymmetry.

\begin{corollary}\label{c:second symmetrization}
Let $A, A_{1,2}, \{F_k\}_k$ be as in Remark \ref{r:new assumptions, mid symmetrization}.  We define the sets
\begin{align*}
    H_k:= S^1_{e_2, -e_1}(F_k).
\end{align*}
The sequence $\{H_k\}_k \subset \mathbb{R}^2$ enjoys the following properties.
\begin{enumerate}
    \item $\delta_{\gamma_A}(E_k) \ge \delta_{\gamma_{A_{1,2}}}(F_k) \ge \delta_{\gamma_{A_{1,2}}}(H_k)$.
    \item $\lambda_{\gamma_A}(E_k) \ge \gamma_{A_{1,2}}(H_k \Delta H_2(e_1, \overline{s})) \ge (\frac{1}{2}-2o(k))\lambda_{\gamma_{A}}(E_k),$ where we use the notation $o(k)$ to denote a function such that vanishes as $k \rightarrow \infty$.
    \item $H_k$ is the subgraph of a monotonically increasing function $h_k:\partial H_2(e_1, \overline{s})\rightarrow \mathbb{R}$.
\end{enumerate}
\end{corollary}

\begin{proof}
Claim (1) follows immediately from Theorem \ref{t:symmetrization} and Lemma \ref{l:n-1 symmetrization}.

Claim (2) follows from Corollary \ref{c:small overhang measure}.  That is, we have that for $0<c<<1/128$ and all sufficiently large $k$, either
\begin{align}\label{e:small overhang 1}
\gamma_{A_{1,2}}(F_k \cap (\{x_1>\overline{s}\} \cap \{x_2<0\})) \le c\lambda_{\gamma_{A}}(E_k)    
\end{align}
or 
\begin{align}\label{e:small overhang 2}
\gamma_{A_{1,2}}(F_k^c \cap (\{x_1<\overline{s}\} \cap \{x_2>0\})) \le c\lambda_{\gamma_{A}}(E_k).    
\end{align}

Note that for each such $k \in \mathbb{N}$, we also have that 
    \begin{align*}
        \gamma_{A_{1,2}}(F_k \cap \{x_1 > \overline{s}\}) = \frac{1}{2}\lambda_{\gamma_{A}}(E_k) = \gamma_{A_{1,2}}(F_k^c \cap \{x_1 < \overline{s}\}).
    \end{align*}
    Moreover, because $F_k$ is $S^{n-1}_{e_1, e_2}$-symmetrized, for any $(a,b) \subseteq \mathbb{R}$, the slice function
    \begin{align*}
        m_{F_k}(t; (a,b)):= \int_{\{(x_1, t):x_1 \in \mathbb{R}\} \cap \{x_1 \in (a,b)\}}\chi_{F_k}(x_1, t)e^{-\frac{1}{2}(d_1^2x_1^2 + d_2^2t^2)}dx_1
    \end{align*}
    is an increasing function in $t$.  Thus, we have
    \begin{align*}
        \gamma_{A_{1,2}}\Big(F_k \cap \{x_1>\overline{s}, x_2<0\} \Big) \in [0,\frac{\lambda_{\gamma_{A}}(E_k)}{4}]\\
    \gamma_{A_{1,2}}\Big(F_k^c \cap \{x_1<\overline{s}, x_2>0\} \Big) \in [0,\frac{\lambda_{\gamma_{A}}(E_k)}{4}].
    \end{align*}
By Corollary \ref{c:small overhang measure}, we have that
\begin{align*}
    \gamma_{A_{1,2}}\Big((F_k \setminus T_{k}) \cap \{x_1>\overline{s}, x_2<0\} \Big) \times
    \gamma_{A_{1,2}}\Big((F_k^c\setminus V_{k}) \cap \{x_1<\overline{s}, x_2>0\} \Big) & = 0\\
    \gamma_{A_{1,2}}\Big((F_k \setminus T_{k}) \cap \{x_1>\overline{s}, x_2<0\} \Big),
    \gamma_{A_{1,2}}\Big((F_k^c\setminus V_{k}) \cap \{x_1<\overline{s}, x_2>0\} \Big) & \in [0,\frac{1}{4}\lambda_{\gamma_{A}}(E_k)].
\end{align*}
    
Indeed, without loss of generality, we may assume that for some fixed $k\in \mathbb{N}$ sufficiently large \eqref{e:small overhang 1} holds.  Then,
    \begin{align*}
        \gamma_{A_{1,2}}(F_k \cap (\{x_1>\overline{s}\} \cap \{x_2<0\})) &\le c\lambda_{\gamma_{A}}(E_k)\\
        \gamma_{A_{1,2}}(F_k \cap (\{x_1>\overline{s}\} \cap \{x_2>0\})) &\ge (\frac{1}{2} - c)\lambda_{\gamma_{A}}(E_k)\\
        \gamma_{A_{1,2}}(F_k^c \cap (\{x_1<\overline{s}\} \cap \{x_2<0\})) &\in [\frac{1}{4}\lambda_{\gamma_{A}}(E_k), \frac{1}{2}\lambda_{\gamma_{A}}(E_k)]\\  
        \gamma_{A_{1,2}}(F_k^c \cap (\{x_1<\overline{s}\} \cap \{x_2>0\})) &\in [0, \frac{1}{4}\lambda_{\gamma_{A}}(E_k)].
    \end{align*}
Thus, $H_k$ satisfies,
\begin{align*}
\gamma_{A_{1,2}}(H_k \cap (\{x_1>\overline{s}\} \cap \{x_2>0\})) & \in [(\frac{1}{4}-c)\lambda_{\gamma_{A}}(E_k), (\frac{1}{2}-c)\lambda_{\gamma_{A}}(E_k)]\\
        \gamma_{A_{1,2}}(H_k^c \cap (\{x_1<\overline{s}\} \cap \{x_2<0\})) &\in [(\frac{1}{4}-c)\lambda_{\gamma_{A}}(E_k), (\frac{1}{2}-c)\lambda_{\gamma_{A}}(E_k)].
\end{align*}
Hence, $\gamma_{A_{1,2}}(H_k \Delta H_2(e_1, \overline{s})) \ge (\frac{1}{2}-2c)\lambda_{\gamma_{A}}(E_k)$. A symmetric argument works if \eqref{e:small overhang 2} holds in place of \eqref{e:small overhang 1}. The upper bound is immediate from the definition of $F_k,H_k$ and Theorem \ref{t:technical theorem}(1).

Claim (3) follows immediately from the fact that $m_{F_k}(t; \mathbb{R})$ is an increasing function of $t \in \mathbb{R}$. 
\end{proof}

The functions $\{h_k\}_k$ such that 
\begin{align*}
    H_k = \text{subgraph}_{\partial H_2(e_1, \overline{s})}h_k,
\end{align*}
are monotone, and therefore $BV(\mathbb{R};\gamma_{A_{1,2}'})$ functions.  However, to prove Theorem \ref{t:technical theorem}, we need greater regularity. To obtain functions with greater regularity, we replace the $H_k$ with solutions to a relative volume-constrained $\gamma_A$-perimeter minimization problem.

\begin{definition}\label{d:Ik}(The sets $I_k$)
For convenience, we define 
\begin{align*}
    \Omega_+&:=\{x_2 >0\}\\
    \Omega_-&:=\{x_2 <0\}.
\end{align*}
Now, we let $I_k^+$ be the solution to the following problem. For
\begin{align*}
    \mathcal{C}_{\Omega_+}(k):= \{B \in \mathcal{B}(\mathbb{R}^n)\,:\, & B\Delta H_k \subset \subset \Omega_+,\\
& \gamma_{A_{1,2}}(B \cap \Omega_+) = \gamma_{A_{1,2}}(H_k \cap \Omega_+) \},
\end{align*}
we let $I^+_k$ be the solution to the $\gamma_{A_{1,2}}$-perimeter minimization under $\gamma_{A_{1,2}}$-volume-constraint problem in $\Omega_+ $ given by $\mathcal{C}_{\Omega_+}(k)$. That is, let $I^+_k$ satisfy:
\begin{align*}
\begin{cases}
I^+_k \in \overline{\mathcal{C}_{\Omega_+}(k)}^{L^1(\mathbb{R}^n, \gamma_{A_{1,2}})}\\
\text{Per}_{\gamma_{A_{1,2}}}(I_k) \le \text{Per}_{\gamma_{A_{1,2}}}(B) \quad \forall B \in \mathcal{C}_{\Omega_+}(k). 
\end{cases}
\end{align*}
We then repeat this construction in $\Omega_-$.  That is, for
\begin{align*}
    \mathcal{C}_{\Omega_-}(k):= \{B \in \mathcal{B}(\mathbb{R}^n): & B\Delta I_k^+ \subset \subset \Omega_-,\\
& \gamma_{A_{1,2}}(B \cap \Omega_-) = \gamma_{A_{1,2}}(H_k \cap \Omega_-) \},
\end{align*}
we let $I_k$ be the solution to the $\gamma_{A_{1,2}}$-perimeter minimization under $\gamma_{A_{1,2}}$-volume-constraint problem in $\Omega_-$ given by $\mathcal{C}_{\Omega_-}(k)$. Explicitly, let $I_k$ satisfy:
\begin{align*}
\begin{cases}
I_k \in \overline{\mathcal{C}_{\Omega_-}(k)}^{L^1(\mathbb{R}^n, \gamma_{A_{1,2}})}\\
\text{Per}_{\gamma_{A_{1,2}}}(I_k) \le \text{Per}_{\gamma_{A_{1,2}}}(B) \quad \forall B \in \mathcal{C}_{\Omega_-}(k). \end{cases}
\end{align*}    
\end{definition}

\begin{definition}(Weak Graphicality)
        We shall say that a curve $\Gamma:(a, b) \rightarrow \mathbb{R}^2$ is \textit{weakly graphical} over $\partial H(e_1, \overline{s})$ if for every slice $\{x_2=c\}$, the intersection $\{x_2=c\} \cap \Gamma(a,b)$ is connected.
    \end{definition}

\begin{lemma}[Properties of $I_k$]\label{l:properties of Ik}
Let $A, A_{1,2}, H_k\subset \mathbb{R}^2$ be as in Corollary \ref{c:second symmetrization} and $I_k\subset \mathbb{R}^2$ be as in Definition \ref{d:Ik}.  Then, for all sufficiently large $k \in \mathbb{N}$ the following properties hold.
\begin{enumerate}
    \item $\delta_{\gamma_{A_{1,2}}}(I_k) \le \delta_{\gamma_{A_{1,2}}}(H_k)$.
    \item $\frac{1}{5}\lambda_{\gamma_{A}}(E_k)\le \gamma_{A_{1,2}}(I_k \Delta H_2(e_1, \overline{s}))$.
    \item Let $\Gamma_k^\pm:(0,\infty) \rightarrow \mathbb{R}^2$ be unit-speed parametrizations of $\partial I_k \cap \Omega_+$ and $\partial I_k \cap \Omega_-$, respectively, with $\Gamma_k^{\pm}(0) \cdot e_2 = 0$.  Then, $\Gamma_k^{\pm}$ is smooth and weakly graphical over $\partial H_2(e_1, \overline{s})$ in $\Omega^{\pm}$.
    \item  For any choice of normal $\vec{\eta}^{\pm}_k$, there exists a constant $C_{k, \vec{\eta}}^{\pm}$ such that for any $t \in \mathbb{R}$
    \begin{align}\label{e:ODE}
        \kappa_k^{\pm}(t) := \langle \vec{\eta}_k^\pm(t), (\Gamma_k^{\pm})''(t)\rangle = C_{k,\vec{\eta}}^{\pm} - \langle A\Gamma_k^{\pm}(t), \vec{\eta}_{k}^{\pm}(t)\rangle. 
    \end{align}
    \item  For any $0< \varepsilon$ and $0<<R<\infty$, the following holds for all sufficiently large $k$ (depending upon $\varepsilon, R$).  There exists a function $u_k:\partial H_2(e_1,\overline{s}) \rightarrow \mathbb{R}$ such that
        \begin{align*}
           \norm{u_k}_{C^1([-R, - \varepsilon] \cup [\varepsilon,R])} \le \varepsilon.
        \end{align*}
\end{enumerate}
\end{lemma}

\begin{proof}
Claim (1) follows from the existence of $I_k$ as a relative, volume-constrained perimeter minimizer and the fact that $H_k$ is in the comparison class which defines $I_k$. To see claim (2), we first note that because for all $k \in \mathbb{N}$, $\gamma_{A_{1,2}}(H_k) = \alpha$, we must have
\begin{align*}
    \gamma_{A_{1,2}}(H_k \cap \{x_1 >\overline{s}\}) = \gamma_{A_{1,2}}(H_k^c \cap \{x_1 <\overline{s}\}).
\end{align*}
From Corollary \ref{c:second symmetrization}, we know that $H_k$ are the subgraph of a monotone increasing function, hence 
\begin{align*}
    \gamma_{A_{1,2}}(H_k^c \cap \{x_2>0, x_1<\overline{s}\})\le \frac{1}{2}\gamma_{A_{1,2}}(H_k^c \cap \{x_1 <\overline{s}\})\\
    \gamma_{A_{1,2}}(H_k \cap \{x_2<0,x_1>\overline{s}\})\le \frac{1}{2}\gamma_{A_{1,2}}(H_k \cap \{x_1 >\overline{s}\}).
\end{align*}
In particular, we have that 
\begin{align*}
    \gamma_{A_{1,2}}(H_k \cap \{x_2>0\}) & = \frac{1}{2}\alpha + \gamma_{A_{1,2}}(H_k \cap \{x_1>\overline{s}, x_2>0\})\\
    & \qquad - \gamma_{A_{1,2}}(H_k^c \cap \{x_1 <\overline{s}, x_2>0\}).
\end{align*}
Now, let $h_k$ be as in Corollary \ref{c:second symmetrization}(3). If $h_k(0)\ge 0$, it must be that $\gamma_{A_{1,2}}(H_k^c \cap \{x_1 <\overline{s}, x_2>0\}) =0$ and the estimates in the proof of Corollary \ref{c:second symmetrization} give that
\begin{align}\label{e:excess mass I_k estimate}
    \gamma_{A_{1,2}}(H_k \cap \{x_2>0\}) & \ge \frac{1}{2}\alpha + (\frac{1}{4}-c)\lambda_{\gamma_{A}}(E_k).
\end{align}
On the other hand, if $h_k(0)<0$, by the monotonicity of $h_k$, it must be that $H_k \cap \{x_1>\overline{s}\} = H_k \cap \{x_1>\overline{s}, x_2>0\}$. Therefore,
\begin{align*}
    \gamma_{A_{1,2}}(H_k \cap \{x_2>0\}) & = \frac{1}{2}\alpha + \gamma_{A_{1,2}}(H_k \cap \{x_1>\overline{s}, x_2>0\})\\
    & \qquad - \gamma_{A_{1,2}}(H_k^c \cap \{x_1 <\overline{s}, x_2>0\})\\
    & \ge \frac{1}{2}\alpha + \gamma_{A_{1,2}}(H_k \cap \{x_1>\overline{s}\})\\
    & \qquad - \frac{1}{2}\gamma_{A_{1,2}}(H_k^c \cap \{x_1 <\overline{s}\})\\
    & = \frac{1}{2}\alpha + \frac{1}{2}\gamma_{A_{1,2}}(H_k \cap \{x_1>\overline{s}, x_2>0\}) \ge \frac{1}{2}\alpha + (\frac{1}{8} - \frac{c}{2})\lambda_{\gamma_{A}}(E_k).
\end{align*}
Together, these cases imply that 
\begin{align}\label{e:deficit in mass I_k estimate}
    \gamma_{A_{1,2}}(H_k \cap \{x_2<0\}) & \le \frac{1}{2}\alpha - (\frac{1}{8} - \frac{c}{2})\lambda_{\gamma_{A}}(E_k).
\end{align}
Thus, since $I_k$ is a relative, volume-constrained perimeter-minimizer,
\begin{align*}
\gamma_{A_{1,2}}(I_k \cap \{x_2>0\}) & \ge \frac{1}{2}\alpha + (\frac{1}{8} - \frac{c}{2})\lambda_{\gamma_{A}}(E_k)\\
\gamma_{A_{1,2}}\Big(I_k \Delta H_2(e_1, \overline{s}) \Big) & \ge  \gamma_{A_{1,2}}\Big([I_k \setminus H_2(e_1, \overline{s})]\cap \{x_2>0\} \Big) + \gamma_{A_{1,2}}\Big([H_2(e_1, \overline{s}) \setminus I_k]\cap \{x_2<0\} \Big) \\
& \ge 2(\frac{1}{8} - \frac{c}{2})\lambda_{\gamma_{A}}(E_k).
\end{align*}
This give (2).

To see (3), we note that smoothness follows immediately from Theorem \ref{t:regularity for VCPMs}(1).  The fact that $\Gamma^{\pm}_k$ are always weakly graphical over $\partial H_2(e_1,\overline{s})$ follows from the $I_k$ being relative volume-constrained $\gamma_A$-perimeter minimizers.  If they were not weakly graphical, by the smoothness of $\Gamma^{\pm}_k$ in $\Omega^{\pm}$, there would be an open interval $(a, b) \subset (-\infty, 0)(0,\infty)$ such that the slices $\{x_2 = c\}$ satisfy
    \begin{align*}
        \mathcal{H}^0(\{x_2=c\} \cap \Gamma^\pm_k) > 1 \qquad \text{for all }c\in(a,b).
    \end{align*}
    But, this guarantees that $S_{e_2,-e_1}^{1}$-symmetrization is strictly perimeter decreasing (this may be seen, for example, by following the proof of \cite{Maggi_book12} Theorem 14.4 with weights).  Therefore, the $\Gamma_k^{\pm}$ must be weakly graphical over $\partial H_2(e_1,\overline{s})$.
    
    The final two claims follow from straight-forward applications of Theorem \ref{t:regularity for VCPMs} and Corollary \ref{c:tamanini convergence}, respectively.
\end{proof}

\begin{lemma}[Continuity of the constant]\label{l:continuity of Ck}
    Let $C_{k,\vec{\eta}}^{\pm}$ as in \eqref{e:ODE}.  Let $\vec{\eta}_k$ be the outer unit normal to $I_k$.  Then,
    \begin{align*}
\lim_{k \rightarrow \infty}C_{k,\vec{\eta}}^{\pm} = d_1^2\overline{s}.       
    \end{align*}
\end{lemma}
\begin{proof}
    Note that for all sufficiently large $k(\varepsilon,R) \in \mathbb{N}$, by Lemma \ref{l:properties of Ik}(5) and Theorem \ref{t:regularity for VCPMs}(1), we may assume that $\Gamma^{+}_k\res[\varepsilon, R]=\text{graph}_{\partial H(\overline{s},e_1)}u_k$ for $u_k$ a $C^\infty$ function. Thus, for any such $k \in \mathbb{N}$ and any $0<\rho_1(k)<\rho_2(k)<\infty$ such that 
    \begin{align*}
        \Gamma_k^+([\rho_1, \rho_2]) \subset \{x_2 \in (\varepsilon,R)\}
    \end{align*}
    we can calculate 
    \begin{align*}
    (\Gamma^+_k)'(\rho_2(k)) - (\Gamma_k^+)'(\rho_1(k)) & = \int_{\rho_1(k)}^{\rho_2(k)}(\Gamma^+_k)''(t)dt\\
    & = \int_{\rho_1(k)}^{\rho_2(k)} C_{k,\vec{\eta}_k}^+ -\langle A\Gamma^+_k(t), \vec{\eta}_k(t)\rangle dt .
    \end{align*}
    Now, for any sequence of such $k\in \mathbb{N}$, we see by Lemma \ref{l:properties of Ik}(5) and the fact that $I_k \rightarrow H_2(e_1, \overline{s})$ that
\begin{align*}
    |(\Gamma^+_k)'(\rho_2(k)) - (\Gamma_k^+)'(\rho_1(k))|\overset{k \rightarrow \infty}{\rightarrow} 0\\
    \langle A\Gamma^+_k(t_k), \vec{\eta}_k(t_k) \rangle\overset{k \rightarrow \infty}{\rightarrow} d_1^2\overline{s}
\end{align*}
uniformly for any $t_k \in (\rho_1(k),\rho_2(k))$.  Thus, it must be that $C^{+}_{k,\vec{\eta}_k} \rightarrow d_1^2\overline{s}$, as claimed. A symmetric argument for $C^{-}_{k,\vec{\eta}_k}$ gives the same result.
\end{proof}

It will be convenient to rewrite the constants $C^{\pm}_{k,\vec{\eta}_k}$ in the following way.

\begin{definition}\label{d:ell_k}(The heights $\ell_k^{\pm}$)
    As above, let $\vec{\eta}_k$ be the outer unit normal to $\partial I_k$.  For each $k \in \mathbb{N}$, we define $\ell_k^{\pm}\in \mathbb{R}$ such that 
    \begin{align*}
       d_1^2\ell_k^{\pm} = C_{k,\vec{\eta}_k}^{\pm}.
    \end{align*}
\end{definition}

\section{Comparison Vector Fields}\label{s:preliminaries for gradient control}

Now that we have reduced to the sequence $\{I_k\}_k$, the rest of the proof rests upon analyzing curves $\Gamma$ that satisfy the ordinary differential equation \eqref{e:ODE}. In this section, we prove some preliminary results on such curves, which will be used to obtain gradient control in the next sections.  The main result is Lemma \ref{l:curvature bounds by angle}, which establishes curvature estimates based upon angle of the tangent/normal of $\Gamma$ with certain auxiliary reference vector fields $\vec{F}_k, \vec{G}_k$ (see Definition \ref{d:F,G vector field}).  These auxiliary vector fields are studied in Lemma \ref{r:description of vector fields} and  Lemma \ref{l:G derivatives bound}.

\begin{remark}\label{r:Gamma_k definition}
Following Lemma \ref{l:properties of Ik}(4), we shall use  $\Gamma_k^{\pm}:[0,\infty) \rightarrow \mathbb{R}^2$ to denote unit-speed curves parametrizing $\partial I_k \cap \Omega_-$ and $\partial I_k \cap \Omega_+$, respectively. Because the ordinary differential equations \eqref{e:ODE} which govern $\Gamma_k^{+}$ and $\Gamma_k^-$ are qualitatively identical, in the subsequent lemmata and definitions we shall only concern ourselves with $\Gamma_k^+$ in $\Omega_+= \{x_2>0\}$.  All of these results may be trivially extended to $\Gamma_k^-$ by reflection and identical arguments, \textit{mutatis mutandis} (e.g., with $C_k^-$ in place of $C_k^+$).
\end{remark}

\subsection{The vector fields $\vec{F}_k, \vec{v}_k, \vec{G}_k$}
We begin with the following definitions.

\begin{definition}\label{r:Omega def, existence of two sols}
    Let $k \in \mathbb{N}$ and $0<\varepsilon_0$.  Let $\ell_k^{\pm}$ be as in Definition \ref{d:ell_k}. We define the domains
    \begin{align*}
   \Omega_k^+&:= \{(x_1, x_2): x_2 \ge |\ell_k^+|+1\},\\
   \Omega_k^-&:= \{(x_1, x_2): x_2 \le -(|\ell_k^-|+1)\}\\
   \Omega_k&:= \Omega_k^+ \cup \Omega_k^-,\\
   \Omega_{k,\varepsilon_0}&:= \Omega_k \cap \Big(\{ x_2 >0\, ,\, \frac{|x_1-\ell_k^+|}{|x_2|} \le \varepsilon_0\} \cup \{x_2 <0\,,\,\frac{|x_1-\ell_k^-|}{|x_2|} \le \varepsilon_0\} \Big). 
\end{align*}

We define the function $f:\mathbb{R}^2 \times \mathbb{S}^1 \rightarrow \mathbb{R}$
\begin{align}\label{e:F defining eq}
        f_k(x, \vec{X}):= C_{k,\vec{\eta}_k}^+ - \langle A x, \vec{X}\rangle
\end{align}
\end{definition}

\begin{lemma}\label{l:existence of F, v}
    Let $\Omega_k \subset \mathbb{R}^n$ be as in Definition \ref{r:Omega def, existence of two sols}.  Then, the following properties hold.
    \begin{enumerate}
        \item $f_k:\Omega_k \times \mathbb{S}^1 \rightarrow \mathbb{R}$ is continuous.
        \item For any fixed $x=(x_1,x_2)\in \Omega_k$, $f_k(x, \vec{X})=0$ has exactly two solutions in $\mathbb{S}^1$,  one with $\vec{X} \cdot e_1 >0$ and another with $\vec{X} \cdot e_1 <0$.
    \end{enumerate}
\end{lemma}

\begin{proof}
    That $f_k$ is continuous is obvious from the definition \eqref{e:F defining eq}.  To see that for any $x=(x_1,x_2)\in \Omega_k^+$, $f_k(x, \vec{X})=0$ has at least two solutions in $\mathbb{S}^1$, we simply choose test vectors $\vec{X}=\pm e_2$, we observe that
    \begin{align*}
        f_k(x,e_2)& =\Big(d_1^2\ell_k^+ - d_2^2x_2\Big) <0\\ f_k(x, -e_2)& = \Big(d_1^2\ell_k^++d_2^2x_2\Big)>0.
    \end{align*}
    Hence, by continuity, for any $x\in \Omega_k^+$ there exist two solutions to $f_k(x,\vec{X})=0$, one with $\vec{X} \cdot e_1 >0$ and another with $\vec{X} \cdot e_1 <0$, as claimed.

    To see that for any $x=(x_1,x_2)\in \Omega_k^+$, $f_k(x, \vec{X})=0$ has at most two solutions in $\mathbb{S}^1$, we let $\vec{X}=(X_1, X_2)$ and consider $X_1 = \sqrt{1-X_2^2}$.  Note that as $X_2$ ranges from $-1$ to $1$, the quantity $-d_2^2 x_2X_2$ is strictly monotonically decreasing and the quantity $-d_1^2x_1\sqrt{1-X_2^2}$ is also strictly decreasing on half its domain.  Therefore, if there are multiple solutions on the upper half of the circle, they must occur for $X_2 \in [0,1)$.  In particular, there must be a point where
    \begin{align*}
        \frac{\partial}{\partial X_2}f_k(x, \vec{X}) & = \frac{\partial}{\partial X_2}(d_1^2\ell_k^+ -d_1^2x_1\sqrt{1-X_2^2} - d_2^2x_2X_2)\\
        & = \frac{d_1^2x_1X_2}{\sqrt{1-X_2^2}} - d_2^2x_2=0.
    \end{align*}
In particular, there is only one solution for $X_2 \in [-1, -1]$ and $X_1 > 0$.  Thus, if there is more that one solution in the upper half circle, then $f_k(x,(X_2, \sqrt{1-X_2^2}))$ decreases and then increases as $X_2$ ranges from $-1$ to $1$.  But, this contradicts that $f_k(x,e_2) =\Big(d_1^2\ell_k^+ - d_2^2x_2\Big) <0$.

An exactly similar argument applies for $X_1 = -\sqrt{1-X_2^2}$ and $f$ in $\Omega_k^-$.
\end{proof}

\begin{definition}\label{d:F,G vector field}(Auxiliary vector fields $\vec{F}_k$ and $\vec{G}_k$)
Let $f_k$ be as in \eqref{e:F defining eq} and $\Omega_k$ be as in Definition \ref{r:Omega def, existence of two sols}. For each $k \in \mathbb{N}$, we make the following definitions.
\begin{enumerate}
    \item We define vector fields $\vec{F}_k, \vec{v}_k: \Omega_k\rightarrow \mathbb{S}^1$ such that
    \begin{align}\label{e:F defining eq 1}
f(x, \vec{F}(x))= C_{k,\vec{\eta}_k}^+ - \langle A x, \vec{F}(x)\rangle \equiv 0;\qquad f(x, \vec{v}(x))= C_{k,\vec{\eta}_k}^+ - \langle A x, \vec{v}(x)\rangle\equiv 0.
    \end{align} 
    Moreover, we choose $\vec{F}_k$ to be the solution such that $ \vec{F}_k\cdot e_1 \ge 0$ in $\Omega_k$.  We let $\vec{v}_k$ be the solution such that $\vec{v}_k\cdot e_1 \le 0$ in $\Omega_k$.
    \item Writing $\vec{F}_k = (F_1, F_2)$, we define
    \begin{align*}
        \vec{G}_k(x_1,x_2) := \begin{cases}(-F_2(x_1, x_2), F_1(x_1, x_2)) \quad \text{in} \quad \Omega_k^+\\
        (F_2(x_1, x_2), -F_1(x_1, x_2)) \quad \text{in} \quad \Omega_k^-.\end{cases}
    \end{align*}
\end{enumerate}
\end{definition}

\begin{remark} We define $\vec{G}_k$ as above so that we may prove results on $\vec{G}_k$ by reflection. 

The vector field $\vec{G}_k(x_1, x_2)$ represents the \textit{tangent vector} for a curve satisfying the ODE \eqref{e:ODE} which has zero curvature at $(x_1, x_2)$. The vector fields $\vec{F}_k, \vec{v}_k$ represent \textit{normal vectors} to these curves.  

We will use the vector field $\vec{F}_k, \vec{G}_k$ as a reference frame.  The difference between the frame $(\vec{\eta}_k(t), \Gamma_k'(t))$ and the frame $(\vec{F}_k(\Gamma_k(t)), \vec{G}_k(\Gamma_k(t)))$ will give us a way to estimate the scalar curvature of $\Gamma_k$. See Lemma \ref{l:curvature bounds by angle}.
\end{remark}

\begin{lemma}[Basic Properties of $\vec{F}_k, \vec{G}_k$]\label{r:description of vector fields}
    Let $A, A_{1,2}$ be as in Definition \ref{d:standing assumptions}.  Let $\vec{F}_k(x_1, x_2) = (F_1(x_1,x_2), F_2(x_1,x_2))$ as in Definition \ref{d:F,G vector field}. Let $\Omega_k$ as in Remark \ref{r:Omega def, existence of two sols}. Let $C_{F_2}:=C_{F_2}(x_1,x_2)$ and  $C_{F_1}:=C_{F_1}(x_1,x_2)$ such that $C_{F_1}, C_{F_2}\in [1/2, 1]$ and $\sqrt{1-F_i(x_1,x_2)^2} = 1-C_{F_i}F_i(x_1,x_2)^2$ for $i=1, 2$.
    
    Then, the vector fields $\vec{F}_k$ satisfy the following properties. For $\ell_k^{\pm}$ as in Definition \ref{d:ell_k},
\begin{enumerate}
    \item {\rm (Coordinate formula)}
    \begin{align*}
        F_2(x_1,x_2) &=
            \frac{1- \sqrt{1-4(C_{F_2}\frac{x_1}{x_2}\frac{d_1^4}{d_2^4}\frac{\ell_k^\pm-x_1}{x_2})}}{2C_{F_2}\frac{d_1^2x_1}{d_2^2x_2}}.\\
     F_1(x_1,x_2) &=\frac{\frac{d_1^2}{d_2^2}\frac{x_1}{x_2} + \sqrt{\frac{d_1^4}{d_2^4}\frac{x_1^2}{x_2^2}-4C_{F_1}(\frac{d_1^2}{d_2^2}\frac{\ell_k^\pm}{x_2}-1)}}{2C_{F_1}}.
\end{align*}
\item {\rm (Estimates I)} For any $0<\varepsilon_0\le 10^{-1}$, if we additionally impose that $(x_1, x_2)$ such that $|\frac{\ell_k^+ - x_1}{x_2}| \le \varepsilon_0$, then there is a function $C(\varepsilon_0) \in (1, \infty)$ such that $C(\varepsilon_0) \searrow 1$ as $\varepsilon_0 \rightarrow 0^+$ and 
\begin{align}\label{e:F_2 estimate}
    \frac{d_1^2}{d_2^2} \frac{|\ell_k^+ - x_1|}{x_2}\le |F_2(x_1,x_2)| \le \frac{d_1^2}{d_2^2}\frac{|\ell_k^+ - x_1|}{x_2}C(\varepsilon_0)
\end{align}
\item {\rm(Estimates II)} The following statements hold.
\begin{enumerate}
    \item If $\ell_k^+ <0$, then for all $x \in \Omega_k^+ \cap \{x_1 \in [\ell_k^+, 0]\}$, $\frac{d_1^2}{d_2^2}\frac{\ell_k^+}{x_2} \le F_2(x_1, x_2)$.
    \item For all $x \in \Omega_k^+ \cap \{x_1 \ge 0\}$, we have
\begin{align*}
 \cos\Big(\pi - \left(\frac{-d_1^2\ell_k^+ + d_2^2x_2}{d_1^2x_1 + d_2^2x_2 \frac{\pi}{4}}\right)\Big)\le F_2(x_1, x_2).
\end{align*}
\item For $x \in \Omega_k^+\cap \{x_1 < \ell_k^+\}$, we have
\begin{align*}
F_2(x_1, x_2) \le 
 \cos\left(\frac{d_2^2x_2 - d_1^2\ell_k^+}{d_2^2x_2 \frac{\pi}{4}-d_1^2x_1}\right).
\end{align*} Corresponding estimates with the inequalities reversed hold in $\Omega^-_k$ with $\ell_k^-$ in place of $\ell_k^+$.
\end{enumerate}
\end{enumerate}
\end{lemma}

\begin{proof}
By definition of $\ell_k^{\pm}$ and $\vec{F}_k$, in $\Omega_k^{\pm}$
\begin{align*}
    0 & = d_1^2(\ell_k^\pm - x_1F_1) - d_2^2x_2F_2\\
    & = d_1^2(\ell_k^\pm - x_1\sqrt{1-F_2^2}) - d_2^2x_2F_2.
\end{align*}
Now, since $1-x \le \sqrt{1-x} \le 1-\frac{1}{2}x$ for all $x \in [0,1]$, we have that $\sqrt{1-F_2^2} = 1-C_{F_2}F_2^2$ for $C_{F_2}\in [1/2, 1]$.  Therefore, in $\Omega_k^{\pm}$
\begin{align*}
0 & = d_1^2(\ell_k^{\pm} - x_1\sqrt{1-F_2^2}) - d_2^2x_2F_2 \\
    & = d_1^2(\ell_k^{\pm} - x_1) + C_{F_2}d_1^2x_1F_2^2 -     d_2^2x_2F_2\\
   0 & = \frac{d_1^2}{d_2^2}\frac{\ell_k^{\pm}-x_1}{x_2} - F_2 + C_{F_2}\frac{d_1^2x_1}{d_2^2x_2}F_2^2.
\end{align*}
Thus, by the quadratic formula
\begin{align*}
    F_2(x_1, x_2) &= \frac{1\pm \sqrt{1-4(C_{F_2}\frac{d_1^2x_1}{d_2^2x_2}\frac{d_1^2}{d_2^2}\frac{\ell_k^{\pm}-x_1}{x_2})}}{2C_{F_2}\frac{d_1^2x_1}{d_2^2x_2}}.
\end{align*}
Now, since by definition of $\vec{F}_k$ it must be that $F_2 \equiv 0$ on $\Omega_k \cap \big(\{x_2<0\,,\, x_1 =\ell_k^{-}\} \cup \{x_2>0\,,\, x_1 =\ell_k^{+}\}\big)$, we have
\begin{align*}
        F_2(x_1,x_2) &=
            \frac{1- \sqrt{1-4(C_{F_2}\frac{d_1^2x_1}{d_2^2x_2}\frac{d_1^2}{d_2^2}\frac{\ell_k^{\pm}-x_1}{x_2})}}{2C_{F_2}\frac{d_1^2x_1}{d_2^2x_2}}.
\end{align*}

By an exactly similar argument, we calculate for $F_1(x_1, x_2)$
\begin{align*}
    0 & = d_1^2(\ell_k^{\pm} - x_1F_1) - d_2^2x_2F_2\\
 & = d_1^2(\ell_k^{\pm} - x_1F_1) - d_2^2x_2\sqrt{1-F_1^2} \\
    & = d_1^2\ell_k^{\pm} - d_1^2x_1F_1 + C_{F_1}d_2^2x_2F_1^2 - d_2^2x_2\\
  0  & = (\frac{d_1^2}{d_2^2}\frac{\ell_k^{\pm}}{x_2} - 1) -\frac{d_1^2}{d_2^2}\frac{x_1}{x_2}F_1 + C_{F_1}F_1^2.
\end{align*}
Thus, using the quadratic formula and the fact that on $\Omega_k \cap \big(\{x_2<0\,,\, x_1 =\ell_k^{-}\} \cup \{x_2>0\,,\, x_1 =\ell_k^{+}\}\big)$, $F_1 \equiv 1$, we may let $|x_2| \rightarrow 0$ to see that
\begin{align*}
    F_1(x_1,x_2) &=\frac{\frac{d_1^2}{d_2^2}\frac{x_1}{x_2} + \sqrt{\frac{d_1^4}{d_2^4}\frac{x_1^2}{x_2^2}-4C_{F_1}(\frac{d_1^2}{d_2^2}\frac{\ell_k^{\pm}}{x_2}-1)}}{2C_{F_1}}.
\end{align*}

The estimates on $F_2$ in \eqref{e:F_2 estimate}, follow immediately by Taylor expansion.

To see the remaining estimates, we observe that for $(x_1, x_2) \in \Omega_k^+$ and $x_1 \in (\ell_k^{+},0]$ (when this interval is non-empty), $F_2(x_1, x_2)<0$ and is monotonically decreasing as $x_1 \nearrow 0$.  Therefore, we easily estimate that for $x_1 \in [\ell_k^{+},0]$
\begin{align*}
    F_2(x_1, x_2) \ge \frac{d_1^2}{d_2^2} \frac{\ell_k^{+}}{x_2}.
\end{align*}
Now, consider $x_1 \in (0,\infty)$.  Identifying $\mathbb{S}^{1}$ with $[0, 2\pi)$ such that $e_2 \mapsto 0$ and $e_1 \mapsto \pi/2$, etc., we wish to find a $c(x_1, x_2) \in (0,\frac{\pi}{2})$ such that $f((x_1, x_2), \pi - c(x_1, x_2))<0$.  This will ensure that $\cos(\pi-c(x_1, x_2))<F_2(x_1, x_2)$. That is, we want a $c(x_1, x_2)$ such that
\begin{align*}
    f(x, \pi) - f(x, \pi-c) = \int_{\pi-c}^\pi \partial_{\theta}f(x, \theta)\, d\theta = \int_{\pi-c}^\pi -d_1^2 x_1 \cos(\theta) + d_2^2x_2 \sin(\theta) \, d\theta \le f(x, \pi).
\end{align*}
Now, we calculate
\begin{align*}
    \int_{\pi-c}^\pi -d_1^2 x_1 \cos(\theta) + d_2^2x_2 \sin(\theta) \, d\theta & \le d_1^2x_1 c + d_2^2x_2 \int_{\pi-c}^\pi \pi-\theta\, d\theta \\
    & \le d_1^2x_1 c + d_2^2x_2 \frac{1}{2}c^2\\
    & \le d_1^2x_1 c + d_2^2x_2 \frac{\pi}{4}c,
\end{align*}
where we have used our assumption that $c \le \frac{\pi}{2}$ in the last line. Now, we observe that
\begin{align*}
d_1^2x_1 c + d_2^2x_2 \frac{\pi}{4}c \le -d_1^2x_1\ell_k^{+} + d_2^2x_2 \quad \longleftrightarrow \quad  c \le \frac{-d_1^2x_1\ell_k^{+} + d_2^2x_2}{d_1^2x_1 + d_2^2x_2 \frac{\pi}{4}}.
\end{align*}
In particular, we may take $c= \frac{-d_1^2x_1\ell_k^{+} + d_2^2x_2}{d_1^2x_1 + d_2^2x_2 \frac{\pi}{4}}$ and obtain that $f(x, \pi-c)>0$. This gives the claim.

The final claim follows by the same argument.  We wish to find a $c(x_1, x_2) \in (0,\frac{\pi}{2})$ such that
\begin{align*}
 f(x, c) - f(x, 0)= \int_{\pi-c}^\pi \partial_{\theta}f(x, \theta)\, d\theta = \int_{\pi-c}^\pi -d_1^2 x_1 \cos(\theta) + d_2^2x_2 \sin(\theta) \, d\theta \le -f(x, 0).
\end{align*}
This will guarantee that $f(x, c)\le 0$ and give us a lower bound on $F_2(x)$.  Proceeding as above, we calculate
\begin{align*}
    \int_{0}^c -d_1^2 x_1 \cos(\theta) + d_2^2x_2 \sin(\theta) \, d\theta & \le -d_1^2x_1 c + d_2^2x_2 \int_{0}^c \theta\, d\theta \\
    & \le -d_1^2x_1 c + d_2^2x_2 \frac{1}{2}c^2 \le -d_1^2x_1 c + d_2^2x_2 \frac{\pi}{4}c.
\end{align*}
Hence, we may take $c(x_1, x_2) \le \frac{d_2^2x_2 - d_1^2\ell_k^{+}}{d_2^2x_2 \frac{\pi}{4}-d_1^2x_1}$. 

The arguments for $ \Omega_k^-$ are symmetric, with inequalities reversed.
\end{proof}

\begin{lemma}[Technical Lemma on $\vec{G}_k$]\label{l:G derivatives bound}
    Let $A, A_{1,2}$ as in Definition \ref{d:standing assumptions}.  Let $0< \varepsilon_0 \le 10^{-1}$ and $\Omega_k, \Omega_{k,\varepsilon_0}$ be as in Definition \ref{r:Omega def, existence of two sols}.  Let $\vec{G}_k$ be the vector field defined in Definition \ref{d:F,G vector field}.  
    
    For all $k\in \mathbb{N}$ sufficiently large and for all $0<\varepsilon_0$ sufficiently small, we have that for all $\vec{w} \in \mathbb{S}^1$ and $x= (x_1, x_2) \in \Omega_{k, \varepsilon_0}$
\begin{align}\label{e:G derivative estimate}
    \Big|\frac{\partial}{\partial w} \vec{G}_k(x_1, x_2)\Big| \le  5 \frac{d_1^2}{d_2^2}\left(\frac{1}{x_2}+C(d_1, d_2)\frac{|x_1|}{|x_2|}\Big|\frac{\ell_k - x_1}{x_2}\Big|\right),
\end{align}
where we may take $C(d_1, d_2) \le 4(\frac{d_1^4}{d_2^4} + \frac{d_1^6}{d_2^6})(1+ 2\frac{1}{|x_2|})$.
\end{lemma}

\begin{proof}
Fix $k \in \mathbb{N}$ and write $\vec{G}_k = (G_1, G_2)$.  Since $|G_1(x_1,x_2)| = |F_2(x_1,x_2)|$ in $\Omega_k^\pm$ by Definition \ref{d:F,G vector field}, we have by Lemma \ref{r:description of vector fields} that for all $0<\varepsilon_0 \le 10^{-1}$ sufficiently small and any $x \in \Omega_{k, \varepsilon_0}$
\begin{align}\label{e:techical lemma for G, G1 F2 small est}
     \frac{d_1^2}{d_2^2}|\frac{\ell_k^{\pm} - x_1}{x_2}|\le |G_1(x_1,x_2)| \le \frac{d_1^2}{d_2^2}\varepsilon_0C(\varepsilon_0)<2\varepsilon_0< \frac{1}{5}.
\end{align}
Since $|\vec{G}_k|\equiv 1$ in $\Omega_k$, we have that for any $\vec{w} \in \mathbb{S}^{1}$, $\frac{\partial}{\partial \vec{w}}\vec{G}_k \cdot \vec{G}_k = 0$.  In particular, 
\begin{align*}
    \frac{|\left(\frac{\partial}{\partial \vec{w}}G_2\right)|}{|\left(\frac{\partial}{\partial \vec{w}}G_1\right)|}<2\varepsilon_0< \frac{1}{5}
\end{align*}
for all $\vec{w}\in \mathbb{S}^{1}$, all $x \in \Omega_{k, \varepsilon_0}$, and all $0<\varepsilon_0<10^{-1}$ sufficiently small.  Therefore, under these conditions, we may estimate 
\begin{align*}
    |\frac{\partial}{\partial \vec{w}}\vec{G}_k(x_1, x_2)| \le 2|\frac{\partial}{\partial \vec{w}}G_1(x_1,x_2)| \qquad \forall \vec{w} \in \mathbb{S}^{1}.
\end{align*}

Now, we reduce to considering $e_1$, instead of any $\vec{w}\in \mathbb{S}^1$. Let $x \in \Omega_{k, \varepsilon_0}$ and $\vec{w}(x, t) \in \mathbb{S}^1$ such that for all sufficiently small $0<|t|$
\begin{enumerate}
    \item $x+t\vec{w}\in \Omega_{k, \varepsilon_0}$,
    \item $\vec{F}_k(x+t\vec{w}) = \vec{F}_k(x)$.
\end{enumerate}
Then, we calculate $\vec{w}(x) := \lim_{t \rightarrow 0}\vec{w}(x,t)$ as follows.  Since
\begin{align*}
    0 & = d_1^2(\ell_k - x_1F_1(x)) - d_2^2x_2F_2(x)\\
    & = d_1^2(\ell_k - (x_1 + tw_1(x,t))F_1(x)) - d_2^2(x_2+t w_2(x,t))F_2(x),
\end{align*}
we may take the difference quotient and let $t \rightarrow 0$ to obtain
\begin{align*}
0 & = \Big[d_1^2w_1(x)F_1(x) + d_2^2 w_2(x)F_2(x)\Big].
\end{align*}
Since isoclines of $\vec{F}_k$ are the isoclines of $F_2$, we have that the isocline of $G_1$ at $x \in \Omega_{k,\varepsilon_0}$ is tangent to $\vec{w}(x)$. Therefore, using our previous estimates on the ratio of $F_1$ and $F_2$, by taking $0<\varepsilon_0$ sufficiently small (now depending upon $\frac{d_2}{d_1}$), we may further guarantee that 
\begin{align*}
    \frac{|w_1(x)|}{|w_2(x)|} &= \frac{d_2^2|F_2(x)|}{d_1^2|F_1(x)|} \le \frac{d_2^2}{d_1^2} \varepsilon_0C(\varepsilon_0)\frac{1}{\sqrt{1-4\varepsilon_0^2}} \le \frac{1}{2}.
\end{align*}
In particular, 
\begin{align*}
    \frac{|\vec{w}^{\perp}(x) \cdot e_2|}{|\vec{w}^{\perp}(x) \cdot e_1|} = \frac{|w_1(x)|}{|w_2(x)|} \le \frac{1}{2}.
\end{align*}
This, and the fact that $\vec{w}$ is tangent to the isoclines, gives that for any $\vec{u} \in \mathbb{S}^1$ 
\begin{align*}
    |\frac{\partial}{\partial \vec{u}}G_1(x)| \le |\frac{\partial}{\partial \vec{w}^{\perp}(x)}G_1(x)| \le 2|\frac{\partial}{\partial e_1}G_1(x)|.
\end{align*}

Now that we have reduced our problem to calculating $|\frac{\partial}{\partial e_1}G_1(x_1,x_2)|$, we make the following preliminary calculation.
\begin{align}\nonumber
    \frac{\partial}{\partial e_1}C_{F_2} & = \frac{\partial}{\partial e_1}\Big(\frac{1-\sqrt{1-F_2^{2}}}{F_2^2}\Big)\\ \nonumber
    & =\frac{\frac{F_2}{\sqrt{1-F_2^2}}F_2^{2} - (1-\sqrt{1-F_2^2})2F_2}{F_2^{4}}\frac{\partial}{\partial e_1}F_2\\ \label{e:CF derivative}
    & = \frac{1}{F_2\sqrt{1-F_2^2}} + \frac{2C_{F_{2}}}{F_2}\frac{\partial}{\partial e_1}G_1.
\end{align}
Next, we calculate directly that
\begin{align*}
        \frac{\partial}{\partial e_1}G_1(x_1,x_2) &=
            \frac{\partial}{\partial x_1}\Big(-\frac{1- \sqrt{1-4(C_{F_2}\frac{x_1}{x_2}\frac{d_1^4}{d_2^4}\frac{\ell_k^{\pm}-x_1}{x_2})}}{2C_{F_2}\frac{x_1}{x_2}\frac{d_1^2}{d_2^2}}\Big)\\
            & = \frac{\frac{1}{2}(1-4C_{F_2}\frac{x_1}{x_2}\frac{d_1^4}{d_2^4}(\frac{\ell_k^{\pm}-x_1}{x_2}))^{-\frac{1}{2}}\Big[-4\frac{d_1^4}{d_2^4}C_{F_2}(\frac{1}{x_2}(\frac{\ell_k^{\pm} - x_1}{x_2}) - \frac{x_1}{x_2^2}) + (-4\frac{d_1^4}{d_2^4}\frac{x_1}{x_2}(\frac{\ell_k^{\pm}-x_1}{x_2}))\frac{\partial}{\partial x_1}C_{F_2}\Big]}{2C_{F_2}\frac{x_1}{x_2}\frac{d_1^2}{d_2^2}}\\
            & \qquad - \frac{-(1- \sqrt{1-4(C_{F_2}\frac{x_1}{x_2}\frac{d_1^4}{d_2^4}\frac{\ell_k^{\pm}-x_1}{x_2})})(2C_{F_2}\frac{1}{x_2}\frac{d_1^2}{d_2^2} + 2\frac{\partial}{\partial x_1}C_{F_2}\frac{x_1}{x_2}\frac{d_1^2}{d_2^2})}{4C_{F_2}^2\frac{x_1^2}{x_2^2}\frac{d_1^4}{d_2^4}}.
\end{align*}
Recalling \eqref{e:CF derivative}, we have 
\begin{align*}
   \frac{\partial}{\partial e_1}G_1(x_1,x_2) & = \text{A}(\frac{2C_{F_2}}{F_2})\frac{\partial}{\partial e_1}G_1(x_1,x_2) + \text{A}(\frac{1}{F_2\sqrt{1-F_2^2}}) + \text{B}\\
   & \qquad + \text{C}(\frac{2C_{F_2}}{F_2})\frac{\partial}{\partial e_1}G_1(x_1,x_2) + \text{C}(\frac{1}{F_2\sqrt{1-F_2^2}})+ \text{D},
\end{align*}
where
\begin{align*}
   \text{A}& := \frac{-4\frac{d_1^4}{d_2^4}\frac{x_1}{x_2}(\frac{\ell_k^{\pm}-x_1}{x_2})}{4C_{F_2}\frac{x_1}{x_2}\frac{d_1^2}{d_2^2}\sqrt{1-4C_{F_2}\frac{x_1}{x_2}\frac{d_1^4}{d_2^4}(\frac{\ell_k^{\pm}-x_1}{x_2})}} ,\\
   \text{B}& :=\frac{-4\frac{d_1^4}{d_2^4}C_{F_2}(\frac{1}{x_2}(\frac{\ell_k^{\pm} - x_1}{x_2}) - \frac{x_1}{x_2^2})}{4C_{F_2}\frac{x_1}{x_2}\frac{d_1^2}{d_2^2}\sqrt{1-4C_{F_2}\frac{x_1}{x_2}\frac{d_1^4}{d_2^4}(\frac{\ell_k^{\pm}-x_1}{x_2})}},\\
   \text{C}& := \frac{(1- \sqrt{1-4(C_{F_2}\frac{x_1}{x_2}\frac{d_1^4}{d_2^4}\frac{\ell_k^{\pm}-x_1}{x_2})})(2\frac{x_1}{x_2}\frac{d_1^2}{d_2^2})}{4C_{F_2}^2\frac{x_1^2}{x_2^2}\frac{d_1^4}{d_2^4}},\\
   \text{D}& := \frac{(1- \sqrt{1-4(C_{F_2}\frac{x_1}{x_2}\frac{d_1^4}{d_2^4}\frac{\ell_k^{\pm}-x_1}{x_2})})(2C_{F_2}\frac{1}{x_2}\frac{d_1^2}{d_2^2})}{4C_{F_2}^2\frac{x_1^2}{x_2^2}\frac{d_1^4}{d_2^4}}
\end{align*}

Thus, elementary algebra gives that
\begin{align*}
    \frac{\partial}{\partial e_1}G_1(x_1,x_2) = \frac{(\text{A}+\text{C})(\frac{1}{F_2\sqrt{1-F_2^2}}) + (\text{B} + \text{D})}{1-(\text{A}+\text{C})(\frac{2C_{F_2}}{F_2})}.
\end{align*}
Thus, we reduce to estimating each term directly. 

By Taylor expansion, we estimate that for all $0<\varepsilon_0$ sufficiently small
\begin{align*}
    \left|\frac{\text{A} + \text{C}}{F_2}\right| & = \left|\frac{-\frac{d_1^2}{d_2^2}(\frac{\ell_k^{\pm}-x_1}{x_2})}{F_2C_{F_2}\sqrt{1-4C_{F_2}\frac{x_1}{x_2}\frac{d_1^4}{d_2^4}(\frac{\ell_k^{\pm}-x_1}{x_2})}} + \frac{(1- \sqrt{1-4(C_{F_2}\frac{x_1}{x_2}\frac{d_1^4}{d_2^4}\frac{\ell_k^{\pm}-x_1}{x_2})})}{2F_2C_{F_2}^2\frac{x_1}{x_2}\frac{d_1^2}{d_2^2}}\right|\\
    & \le \Big|\frac{-\frac{d_1^2}{d_2^2}(\frac{\ell_k^{\pm}-x_1}{x_2})}{F_2C_{F_2}}\Big(1+2C_{F_2}\frac{x_1}{x_2}\frac{d_1^4}{d_2^4}(\frac{\ell_k^{\pm}-x_1}{x_2})\Big) + \frac{\frac{d_1^2}{d_2^2}(\frac{\ell_k^{\pm}-x_1}{x_2})}{F_2C_{F_2}}\Big(1 + 2C_{F_2}\frac{x_1}{x_2}\frac{d_1^6}{d_2^6}(\frac{\ell_k^{\pm}-x_1}{x_2})\Big)\Big|\\
    & \qquad + \frac{\frac{d_1^2}{d_2^2}\Big|\frac{\ell_k^{\pm}-x_1}{x_2}\Big|}{|F_2|C_{F_2}}\Big(\Big|\frac{x_1^2}{x_2^2}(\frac{\ell_k^{\pm}-x_1}{x_2})(2\varepsilon_0)\Big| + \Big|\frac{x_1}{x_2}(\frac{\ell_k^{\pm}-x_1}{x_2})(2\varepsilon_0)\Big|\Big)\\
    & \le \Big|\frac{\frac{d_1^2}{d_2^2}(\frac{\ell_k^{\pm}-x_1}{x_2})}{F_2C_{F_2}}\Big| \Big(\frac{|x_1|}{|x_2|}\frac{d_1^4}{d_2^4}|\frac{\ell_k^{\pm}-x_1}{x_2}| + \frac{|x_1|}{|x_2|}\frac{d_1^6}{d_2^6}|\frac{\ell_k^{\pm}-x_1}{x_2}|+2(\frac{|x_1|}{|x_2|} +\frac{|x_1|^2}{|x_2|^2})|\frac{\ell_k^{\pm}-x_1}{x_2}|(2\varepsilon_0)\Big).
\end{align*}
Now, we use Lemma \ref{r:description of vector fields}\eqref{e:F_2 estimate} to estimate that for all $0<\varepsilon_0$ sufficiently small
\begin{align*}
\Big|\frac{\frac{d_1^2}{d_2^2}(\frac{\ell_k^{\pm}-x_1}{x_2})}{F_2C_{F_2}}\Big| \le 2.
\end{align*}

Thus, we estimate that for all $k \in \mathbb{N}$ sufficiently large (because $\Omega_{k,\varepsilon_0} \subset \{|x_2| \ge |x_1|\}$ for all $k$ sufficiently large) and all sufficiently small $0<\varepsilon_0$ 
\begin{align}\label{e:denominator estimate technical lemma on G}
    1-\text{A}(\frac{2C_{F_2}}{F_2}) -\text{C}(\frac{2C_{F_2}}{F_2}) \ge 1 - 4C_{F_2}\varepsilon_0.
\end{align}

Similarly, by \eqref{e:techical lemma for G, G1 F2 small est} and our previous estimates, we have 
\begin{align*}
    \left|\frac{\text{A+C}}{F_2\sqrt{1-F_2^2}}\right|\le \frac{2(\frac{|x_1|}{|x_2|})|\frac{\ell_k^{\pm}-x_1}{x_2}|(\frac{d_1^4}{d_2^4} + \frac{d_1^6}{d_2^6}) (1+ C\varepsilon_0)}{\sqrt{1-F_2^2}} \le 3\frac{|x_1|}{|x_2|}|\frac{\ell_k^{\pm}-x_1}{x_2}|(\frac{d_1^4}{d_2^4} + \frac{d_1^6}{d_2^6}).
\end{align*}

For $|\text{B} +\text{D}|$, we calculate that for all $0<\varepsilon_0$ sufficiently small Taylor expansion gives 
\begin{align*}
    |\text{B} + \text{D}| & = \left|\frac{-\frac{d_1^2}{d_2^2}(\frac{1}{x_1}(\frac{\ell_k^{\pm} - x_1}{x_2}) - \frac{1}{x_2})}{\sqrt{1-4C_{F_2}\frac{x_1}{x_2}\frac{d_1^4}{d_2^4}(\frac{\ell_k^{\pm}-x_1}{x_2})}} + \frac{(1- \sqrt{1-4(C_{F_2}\frac{x_1}{x_2}\frac{d_1^4}{d_2^4}\frac{\ell_k^{\pm}-x_1}{x_2})})(\frac{1}{x_2})}{2C_{F_2}\frac{x_1^2}{x_2^2}\frac{d_1^2}{d_2^2}}\right|\\
    & \le \Big|-\frac{d_1^2}{d_2^2}(\frac{1}{x_1}(\frac{\ell_k^{\pm} - x_1}{x_2}) - \frac{1}{x_2})\Big(1+2C_{F_2}\frac{x_1}{x_2}\frac{d_1^4}{d_2^4}(\frac{\ell_k^{\pm}-x_1}{x_2}) + o(\varepsilon_0)\Big)\\
    & \qquad +\frac{1}{x_1}\frac{d_1^2}{d_2^2}\frac{\ell_k^{\pm} - x_1}{x_2} + \frac{1}{2C_{F_2}\frac{x_1^2}{x_2^2}\frac{d_1^2}{d_2^2}}(\frac{1}{x_2})o(\varepsilon_0)\Big|\\
    & \le \Big|\frac{d_1^2}{d_2^2}\frac{1}{x_2} + 2C_{F_2}\frac{d_1^6}{d_2^6}\frac{x_1}{x_2^2}(\frac{\ell_k^{\pm} - x_1}{x_2}) +o(\varepsilon_0)\Big|.
\end{align*}
Combining these estimates, we obtain the claim of the lemma.
\end{proof}

\subsection{Curvature Estimates}

\begin{definition}\label{d:pos neg curvature sets quantitative}
Let $\Omega_k\subset \mathbb{R}^2$ be as in Remark \ref{r:Omega def, existence of two sols} and $\vec{F}_k, \vec{v}_k$ as in Definition \ref{d:F,G vector field}. For a point $x= (x_1, x_2) \in \Omega_k \cap \{x_2 \ge 4|x_1|\}$, we define 
    \begin{enumerate}
        \item $\text{Pos}_x$ to be the component of $\mathbb{S}^1\setminus\{\vec{F}(x),\vec{v}(x)\}$ that contains $-e_2$.
        \item $\text{Neg}_x$ to be the component of $\mathbb{S}^1\setminus\{\vec{F}_k(x),\vec{v}_k(x)\}$ that contains $e_2$.
    \end{enumerate}
    For $\tau_0>0$, we define
    \begin{enumerate}
        \item $\text{Pos}_{x,\tau_0}:= \Big\{\vec{w} \in \text{Pos}_x: \min\{|w_2 - v_2(x_1,x_2)|, |w_2 - F_2(x_1,x_2)|\}\ge \tau_0\Big\}$.
        \item $\text{Neg}_{x,\tau_0}:= \Big\{\vec{w} \in \text{Neg}_x: \min\{|w_2 - v_2(x_1,x_2)|, |w_2 - F_2(x_1,x_2)|\}\ge \tau_0\Big\}$.
    \end{enumerate}
\end{definition}

This notation is justified by the following lemma.

\begin{lemma}[Curvature Estimates]\label{l:curvature bounds by angle}
Let $A, A_{1,2}$ be as in Definition \ref{d:standing assumptions}.  Let $0<\delta$ and $\Gamma:(-\delta, \delta)\rightarrow \Omega_k^+$ be any smooth curve which satisfies the following properties:
    \begin{itemize}
        \item $|\Gamma'|\equiv 1$.
        \item $\Gamma'(t) \cdot e_2 \ge 0$ for all $t \in (-\delta,\delta)$. Moreover, we choose $\vec{\eta}$ to be the choice of normal vector to $\Gamma$ such that $\vec{\eta}(t)\cdot e_1 \ge 0$.
        \item $\Gamma$ satisfies \eqref{e:ODE} for $C_{k, \vec{\eta}}^{+} = d_1^2\ell_k$.
    \end{itemize}
Then, $\Gamma$ satisfies the following qualitative behavior.
    \begin{enumerate}
        \item If $\vec{\eta}(t) \in \emph{Pos}_{\Gamma(t)}$, then $\Gamma''(t) >0$.
        \item If $\vec{\eta}(t) \in \emph{Neg}_{\Gamma(t)}$, then $\Gamma''(t) <0$.
    \end{enumerate}

Moreover, if we further assume that $\Gamma \subset \Omega_{k, \varepsilon_0} \cap \{x_2 \ge M|x_1|\}$ for $M\ge 4$, then the following quantitative estimates hold.
    \begin{enumerate}
        \item If $\vec{\eta}(t)\in \emph{Pos}_{\Gamma(t)}$, \begin{align}\label{e:quantitative curvature bounds positive}
    \Gamma''(t) & \ge (1-4/M)d_2^2 (\Gamma(t))_{e_2}\Big|\eta_2(t)-F_2(\Gamma(t))\Big|
\end{align}
\item If $\vec{\eta}(t) \in \emph{Neg}_{\Gamma(t)}$ then \begin{align}\label{e:quantitative curvature bounds negative}
    \Gamma''(t) & \le -(1-4/M)d_2^2 (\Gamma(t))_{e_2}\Big|\eta_2(t)-F_2(\Gamma(t))\Big|.
\end{align}
    \end{enumerate}
\end{lemma}

\begin{proof}
    Since the quantitative behavior follows immediately from the definition of $\text{Pos}_x, \text{Neg}_x$. Now, we prove the quantitative inequalities.  We plug in to \eqref{e:ODE}.
    \begin{align}\nonumber
        \Gamma''(t) & = C_{k\vec{\eta}}^{+} - \langle A \Gamma(t), \vec{\eta}(t)\rangle\\ \nonumber
        & = d_1^2 \ell_k - d_1^2 (\Gamma(t))_{e_1} \eta_1(t) - d_2^2 (\Gamma(t))_{e_2}\eta_2(t)\\ \nonumber
        & = d_1^2 \ell_k - d_1^2 (\Gamma(t))_{e_1}F_1(\Gamma(t)) - d_1^2 (\Gamma(t))_{e_1} \Big[\eta_1(t) - F_1(\Gamma(t))\Big]\\ \nonumber
        & \qquad - d_2^2 (\Gamma(t))_{e_2}F_2(\Gamma(t)) - d_2^2 (\Gamma(t))_{e_2}\Big[\eta_2(t) - F_2(\Gamma(t))\Big]\\ \label{e:curvature estimate 1}
        & = - d_1^2 (\Gamma(t))_{e_1} \Big[\eta_1(t) - F_1(\Gamma(t))\Big] - d_2^2 (\Gamma(t))_{e_2}\Big[\eta_2(t)-F_2(\Gamma(t))\Big].
    \end{align}
    
    Now, we invoke our first assumption.  Suppose that $\vec{\eta}(t) \in \text{Pos}_{\Gamma(t)}$, then by assumption $-1\le \eta_2(t) \le F_2(\Gamma(t))$, and hence 
    \begin{align*}
        -d_2^2(\Gamma(t))_{e_2}[\eta_2(t) - F_2(\Gamma(t))] = d_2^2(\Gamma(t))_{e_2}\Big|\eta_2(t) - F_2(\Gamma(t))\Big|.
    \end{align*}
On the other hand, we may write   
\begin{align*}
|\eta_1(t) - F_1(\Gamma(t))| & = \Big|\sqrt{1-\eta_2(t)^2} - \sqrt{1-F_2(\Gamma(t))^2}\Big|\\
&  \le C|\eta_2(t)^2 - F_2(\Gamma(t))^2| \le 2C|\eta_2(t) - F_2(\Gamma(t))|,
\end{align*}    
where $C \le 2$ is a constant such that if $x \in [0,\frac{1}{10}]$ and $y \in [0,1]$, then
\begin{align*}
    \Big|\sqrt{1-x} - \sqrt{1-y}\Big| \le C |x-y|. 
\end{align*}
This is admissible, because $x \in \Omega_{k,\varepsilon_0}$ allows us to take $\varepsilon_0$ sufficiently small in \eqref{e:F_2 estimate} to give $|F_2(\Gamma(t))|\le \frac{1}{10}$.  Thus, we have from \eqref{e:curvature estimate 1} that 
\begin{align*}
    \Gamma''(t) \ge  -4d_1^2 |(\Gamma(t))_{e_1}|\Big|\eta_2(t) - F_2(\Gamma(t))\Big| + d_2^2 (\Gamma(t))_{e_2}\Big|\eta_2(t)-F_2(\Gamma(t))\Big|.
\end{align*}
Imposing the further constraint that we only consider curves in $\{x_2 \ge M|x_1|\}$ for $M\ge 4$, the last computation tells us that
\begin{align*}
    \Gamma''(t) & \ge  -4d_1^2 |(\Gamma(t))_{e_1}|\Big|\eta_2(t) - F_2(\Gamma(t))\Big| + d_2^2 (\Gamma(t))_{e_2}\Big|\eta_2(t)-F_2(\Gamma(t))\Big|\\
    & \ge (1-4/M)d_2^2 (\Gamma(t))_{e_2}\Big|\eta_2(t)-F_2(\Gamma(t))\Big|.
\end{align*}

By an exactly analogous argument, if $\vec{\eta}(t) \in \text{Neg}_{\Gamma(t)}$, then $\eta_2(t) \ge F_2(\Gamma(t))$ and 
\begin{align*}
    \Gamma''(t) \le  4d_1^2 |(\Gamma(t))_{e_1}|\Big|\eta_2(t) - F_2(\Gamma(t))\Big| - d_2^2 (\Gamma(t))_{e_2}\Big|\eta_2(t)-F_2(\Gamma(t))\Big|\le 0.
\end{align*}
And, under the same condition that $\Gamma \subset \{x_2 \ge M|x_1|\}$ for $M\ge 4$,
\begin{align*}
    \Gamma''(t) &\le  4d_1^2 |(\Gamma(t))_{e_1}|\Big|\eta_2(t) - F_2(\Gamma(t))\Big| - d_2^2 (\Gamma(t))_{e_2}\Big|\eta_2(t)-F_2(\Gamma(t))\Big|\\
    & \le -(1-4/M)d_2^2 (\Gamma(t))_{e_2}\Big|\eta_2(t)-F_2(\Gamma(t))\Big|.
\end{align*}
\end{proof}

\begin{remark}
    By reflection $(x_1, x_2) \mapsto(x_1, -x_2)$, we obtain analogous results in $\Omega_k^-$.
\end{remark}

\section{Comparison Results and Partial Gradient Control}\label{s:gradient control}

In this section, we use the comparison vector fields $\vec{F}_k, \vec{G}_k$ and the curvature estimates in Lemma \ref{l:curvature bounds by angle} to obtain some limited control on the boundaries $\partial I_k$.

\subsection{The Center Strip}

We begin with obtaining Lipschitz control in the center strip $\{(x_1,x_2): |x_2| \le \varepsilon_0\}$.

\begin{theorem}[Full Lipschitz control in the center strip]\label{l:center strip control}
Let $A, A_{1,2}$ be as in Definition \ref{d:standing assumptions}. Let $I_k$ be as in Definition \ref{d:Ik}.  For any $0<\varepsilon_0$, there exists a $k(\varepsilon_0,\overline{s})\in\mathbb{N}$ such that for all $k \ge k(\varepsilon_0, \overline{s})$,
    \begin{align*}
        \partial I_k \cap \{x_2 \in (-\varepsilon_0,\varepsilon_0)\}& = \emph{graph}_{\partial H_2(e_1,\overline{s}) \cap \{x_2 \in (-\varepsilon_0,\varepsilon_0)\}}u_k,\\
        \emph{Lip}(u_k\res (-\varepsilon_0,\varepsilon_0)) & \le \varepsilon_0.
    \end{align*}
\end{theorem}

\begin{proof}
As above, we reduce to arguing in $\Omega^+$ by reflection. Let $0<\varepsilon_0$ be given. We choose $0<\epsilon'< \frac{1}{2}\varepsilon_0$.  We will show that for $\epsilon'$ sufficiently small, $|u_k'| \le 2\epsilon'$, which will give the claim of the lemma.

For each $k\in \mathbb{N}$ sufficiently large, we define:
\begin{enumerate}
    \item $t_{k,1} \in \text{Domain}(\Gamma_k^+)$ such that $\Gamma^+_k(t_{k,1}) \cdot e_2 = 1$.
    \item $t_{k,CS}\in \text{Domain}(\Gamma_k^+)$ by
\begin{align*}
    t_{k,CS}:= \begin{cases}
        0 & \text{if}\quad \{t \in (0,t_{k,1}): \frac{|(\Gamma_k^+)'(t)\cdot e_1|}{|(\Gamma_k^+)'(t)\cdot e_2|}\ge \epsilon'\} = \emptyset\\
        \sup \{t \in (0,t_{k,1}): \frac{|(\Gamma_k^+)'(t)\cdot e_1|}{|(\Gamma_k^+)'(t)\cdot e_2|}\ge \epsilon'\} & \text{otherwise}.
    \end{cases}
\end{align*}
\end{enumerate}
Note that for $k \ge k_0(\frac{1}{100}|\epsilon'|^2, 2)$ from Lemma \ref{l:properties of Ik}(5) we have that $\norm{u_{k}}_{C^1([\frac{1}{100}|\epsilon'|^2, 2])}\le \frac{1}{100}|\epsilon'|^2$, so $t_{k,1}, t_{k, CS}$ are well-defined. Further, we may assume that there exists a sequence $k_j \rightarrow \infty$ such that $0<t_{k_j,CS}<t_{k_j,1}$, since otherwise there is nothing to prove. 

Consider the collection $ A_{k_j} \subset (0,t_{k_j,CS}] \subset \text{Domain}(\Gamma_{k_j}^+)$ of all $t$ which satisfy the following three conditions:
\begin{enumerate}
\item $t$ satisfies 
\begin{align}\label{e:preliminary height bounds}
    \Gamma_{k_j}^+(t) \in \overline{B_{\frac{1}{25}|\epsilon'|^2}(\{x_1 = \ell_{k_j}\})} \cap \{x_2 \in (0,\frac{1}{100}|\epsilon'|^2]\}
\end{align}
\item $|t-t_{k,CS}|\le \epsilon'$.
\item $\frac{|\eta_{k_j,2}(t)|}{|\eta_{k_j,1}(t)|} \in [\frac{1}{2}\epsilon', 2\epsilon'].$
\end{enumerate}

\textbf{Claim 1:} We will show that $t_{k_j,CS} \in A_{k_j}$.  Trivially $t = t_{j_k, CS}$ satisfies (2) and (3) by the smoothness of $\Gamma_{k_j}^{\pm}$ in $\Omega_{k_j}^{\pm}$ in $\Omega_{k_j}^{\pm}$ respectively (Lemma \ref{l:properties of Ik}(3)).  To check property (3), we note that for all $k_j \ge k_0(\frac{1}{100}|\epsilon'|^2, 2)$ from Lemma \ref{l:properties of Ik}(5), we must have $\Gamma_{k_j}^+(t_{k_j, CS}) \cdot e_2 \in (0, \frac{1}{100}|\epsilon'|^2)$. But, by definition, $\partial I_k$ is graphical over $\{\Gamma_k^{+}(t_{k,CS}) \cdot e_2 \le x_2 \le 2 \}$, so we may extend $u_k$ to this range. Therefore, we may estimate by the fundamental theorem of calculus and the triangle inequality 
\begin{align}\nonumber
    |(\Gamma_{k_j}^+(t_{k_j, CS})\cdot e_1) - \ell_{k_j}^+| & \le |\ell_{k_j}^+ - u_k(\frac{1}{100}|\epsilon'|^2)| + \int_{\pi_{2}(\Gamma_{k_j}(t_{k_j, CS}))}^{\frac{1}{100}|\epsilon'|^2} |\tilde{u}_{k_j}'(x_2)|\,dx_2\\
    & \le |\overline{s} - u_{k_j}(\frac{1}{100}|\epsilon'|^2)| + |\overline{s} - \ell_{k_j}^+| + \int^{0}_{\frac{1}{100}|\epsilon'|^2} \epsilon' \,dt\\ \label{e:center strip height control}
    & \le \frac{1}{100}|\epsilon'|^2 (1+ \epsilon')+ |\overline{s} - \ell_{k_j}^+|. 
\end{align}
Since $\ell_k^+ \rightarrow \overline{s}$ by Lemma \ref{l:continuity of Ck}, $t_{k_j, CS}$ satisfies (3) and hence $t_{k_j,CS} \in A_{k_j}$ for all sufficiently large $k \in \mathbb{N}$.

\textbf{Claim 2:} We claim that $A_{k_j} =(0,t_{k_j,CS}]$ for all sufficiently large $k_j \in \mathbb{N}$ and sufficiently small $\epsilon'$. Note that this is equivalent to the claims of the theorem, since in this case we will have shown that $\frac{|(\Gamma_{k_j}^{+})'(t) \cdot e_2|}{|(\Gamma_{k_j}^{+})'(t) \cdot e_1|}= \frac{|\eta_{k_j,1}(t)|}{|\eta_{k_j,2}(t)|} \le 2\epsilon'$ for all $t \in (0,t_{k,1}]$. Thus, for all sufficiently large $k \in \mathbb{N}$, the function $u_k$ from Lemma \ref{l:properties of Ik} may be extended to $(0,1]$.  The smoothness of $\Gamma_{k}^{+}$ in $\Omega_{k}^{+}$ from Lemma \ref{l:properties of Ik}(3) then implies that $u_k$ is smooth in $(0,1)$. The uniform control on the gradient in $(0, \varepsilon_0]$ given by the claim implies a one-sided Lipschitz condition at $x_2=0$.  Obtaining a symmetric result by reflection for $\Gamma_k^-$ on $[-1, 0)$ give the theorem. 

To prove Claim 2, we will show that the component $U \subset A_{k_j}$ which contains $t_{k_j,CS}$ is both relatively open and relatively closed in $(0,t_{k_j,CS}]$. Since $\Gamma_{k_j}^+$ is smooth in $\{x_2 >0\}$ and the three conditions are closed conditions, $U$ is relatively closed. Now, we show that $U$ is relatively open. By the smoothness of the $\Gamma_{k_j}^{+}$, we have $\frac{|(\Gamma_{k_j}^+)'(t_{k_j, CS}) \cdot e_1|}{|(\Gamma_{k_j}^+)'(t_{k_j, CS}) \cdot e_2|} = \epsilon'$. Therefore, by that same smoothness and fact that condition (3) which defines $A_{k_j}$ allows $\frac{|(\Gamma_{k_j}^+)'(t_{k_j, CS}) \cdot e_1|}{|(\Gamma_{k_j}^+)'(t_{k_j, CS}) \cdot e_2|} \in [\frac{1}{2}\epsilon', 2\epsilon']$, there exists some $0<r$ such that $B_r(t_{k_j,CS}) \cap (0,t_{k_j,CS}] \subset A_{k_j}$.  Thus, by our ODE \eqref{e:ODE}, for any $t\in U$
\begin{align*}
|(\Gamma_{k_j}^+)''(t)|& =\Big|d_1^2(\ell_{k_j}^+ - \Gamma_{k_j}^+(t)\cdot e_1) + d_1^2(\Gamma_{k_j}^+(t)\cdot e_1)(1-\eta_{k_j,1}(t)) + d_2^2 (\Gamma_{k_j}^+(t)\cdot e_2) \eta_{k_j,2}(t)\Big|\\
& \le d_1^2\frac{1}{25}|\epsilon'|^2 + d_1^2(|\overline{s}|+1)(2|\epsilon'|^2) + d_2^2\frac{1}{100}|\epsilon'|^2(2\epsilon')\\
& \le C(d_1, d_2, |\overline{s}|)|\epsilon'|^2,
\end{align*}
as long as $0<\epsilon'$ is sufficiently small and $k_j$ is sufficiently large so that $B_{\frac{1}{25}|\epsilon'|^2}(\ell_k) \subset B_1(\overline{s})$.

With this uniform curvature estimate, we calculate that for any $t \in U \subset A_{k_j}$ 
\begin{align*}
    \Big|(\Gamma_k^+)'(t_{k,CS}) - (\Gamma_k^+)'(t)\Big| & = \Big|\int_t^{t_{k,CS}}(\Gamma_k^+)''(r)\vec{\eta}(r)dr\Big|\\
    & \le |t-t_{k,CS}| \times C(d_1, d_2, \overline{s})|\epsilon'|^{2}\le C(d_1, d_2, \overline{s})|\epsilon'|^{3}.
\end{align*}
In particular, by possibly taking $\epsilon'$ smaller, we have that 
\begin{align*}
    \frac{|(\Gamma_k^+)'(t)\cdot e_1|}{(\Gamma_k^+)'(t)\cdot e_2}\in \Big[\epsilon' - C(d_1, d_2, \overline{s})|\epsilon'|^{3},\epsilon' +C(d_1, d_2, \overline{s})|\epsilon'|^{3}\Big] \subset \Big[\frac{3}{4}\epsilon', \frac{3}{2}\epsilon'\Big].
\end{align*}
for all such $t$. To finish the argument, we use another argument as in \eqref{e:center strip height control} to obtain the required height control.

In particular, it cannot be that $\min \{t: t \in U\}>0$, because if such a $t$ existed, there would be a small ball around $\Gamma_{k_j}^+(t)$ such that $B_{r}(\Gamma_{k_j}^+(t))\subset B_{\frac{1}{25}|\epsilon'|^2}(\{x_1 = \ell_k^+\})$ and $|(\Gamma_{k_j}^+)'(\tau)|\le 2\epsilon'$ for all $\tau \in B_r(t)$. Furthermore, the gradient control shows that $|t-t_{k_j, CS}|\le \frac{1}{50}|\epsilon'|^2$.  Thus, $A_{k_j}=U = (0,t_{k_j, CS}]$ for $0<\epsilon'$ sufficiently small (depending only on $d_1, d_2, |\overline{s}|$) and all $k \in \mathbb{N}$ sufficiently large. Applying this argument to $\Gamma_k^-$ by reflection gives the theorem.
\end{proof}

\subsection{The Unbounded Regions} In this subsection, we obtain some limited results on the derivatives $(\Gamma_k^\pm)'$.

\begin{lemma}[First Comparison Result]\label{l:preliminary gradient control Comparison}
Let $A\in S^+_{2\times2}(\mathbb{R})$ and $\alpha \in (0,1)$.  Let $\Gamma_k^{\pm}:\mathbb{R} \rightarrow \mathbb{R}^2$ be as in Lemma \ref{l:properties of Ik}. For all $0<\varepsilon_0$ sufficiently small there exists an $0<R_0(\varepsilon_0,\overline{s})<\infty$ such that if $B_2(\Gamma_k^\pm(t)) \subset \Omega_{k,\varepsilon_0} \cap \{|x_2| \ge R\}$, then
    \begin{align*}
        |\eta_{k,2}(t) - F_2(\Gamma_k^\pm(t))| \le \frac{3}{|x_2|}
    \end{align*}
    for all sufficiently large $k \in \mathbb{N}$.
\end{lemma}

\begin{proof}
As usual, we shall reduce to arguing for $\Gamma_k^+$ and $\Omega_+$ by reflection.

We take $0<\varepsilon_0$ sufficiently small such that $C(\varepsilon_0) \le \frac{4}{3}$ in \eqref{e:F_2 estimate}.
We let $R_0:=\frac{|\overline{s}|+4}{\varepsilon_0}$. Note that this implies $\Omega_{k,\varepsilon_0} \cap \{x_2 \ge R_0\} \subset \{\frac{|x_1|}{|x_2|} \le 2 \varepsilon_0\}.$

We argue by contradiction.  Suppose that, for $\varepsilon_0, R_0$ as above, there exists a $\Gamma_k^+(t)$ such that $B_2(\Gamma_k^+(t)) \subset \Omega_{k,\varepsilon_0} \cap \{x_2 \ge R\}$ and 
    \begin{align*}
        |\eta_{k,2}(t) - F_2(\Gamma_k^+(t))| \ge \frac{3}{x_2}.
    \end{align*}

Then, for all $t' \in [t, t+2)$ such that $\vec{\eta}_k(t') \in \text{Pos}_{\Gamma_k^+(t'), \frac{3}{x_2}} \cup \text{Neg}_{\Gamma_k^+(t'), \frac{3}{x_2}}$, we have by \eqref{e:quantitative curvature bounds positive} and \eqref{e:quantitative curvature bounds negative} of Lemma \ref{l:curvature bounds by angle}
    \begin{align}\label{e:curvature pay off}
        |(\Gamma_k^+)''(t')| \ge 3(1-2\varepsilon_0)>2,
    \end{align}
and, from Lemma \ref{l:G derivatives bound}, that for all $(x_1, x_2)\in B_2(\Gamma_k^{+}(t)) \subset \Omega_{k,\varepsilon_0}$
    \begin{align*}
|\nabla F_2| = |\nabla G_1| & \le 5 \frac{d_1^2}{d_2^2}\left(\frac{1}{x_2}+C(d_1, d_2)\frac{|x_1|}{|x_2|}\Big|\frac{\ell_k^+ - x_1}{x_2}\Big|\right)\\
        & \le 5\frac{d_1^2}{d_2^2}\left(\frac{1}{R_0} + C\varepsilon_0^2 \right) << 2,
\end{align*}
where we have used that $\Omega_{k,\varepsilon_0} \cap \{x_2 \ge R_0\} \subset \{\frac{|x_1|}{|x_2|} \le 2 \varepsilon_0\}.$

We claim that for all $t'\in [t, t+2)$, $$\vec{\eta}_k(t') \in \text{Pos}_{\Gamma_k^+(t'), \frac{3}{x_2}} \cup\text{Neg}_{\Gamma_k^+(t'), \frac{3}{x_2}}.$$ The cases $\vec{\eta}_k(t) \in \text{Pos}_{\Gamma_k^+(t), \frac{3}{x_2}}$ and $\vec{\eta}_k(t) \in \text{Neg}_{\Gamma_k^+(t), \frac{3}{x_2}}$ are symmetric.  We shall only argue for the case $\vec{\eta}_k(t) \in \text{Pos}_{\Gamma_k^+(t), \frac{3}{x_2}}$.

\textbf{Case 1:} Suppose that $t' \in [t, t+2)$ and $(\Gamma_k^+)'(t') \cdot e_2 \le \frac{1}{2}$.  In this case, by \eqref{e:F_2 estimate}, it must be that $|F_2(\Gamma_k^+(t'))-\eta_{k,2}(t')| \ge \frac{1}{4}$.  Thus, $\vec{\eta}_k(t') \in \text{Pos}_{\Gamma_k^+(t'), \frac{3}{x_2}}$ holds trivially.

\textbf{Case 2:} Suppose that $t' \in [t, t+2)$ and $(\Gamma_k^+)'(t') \cdot e_2 \ge \frac{1}{2}$.  If $\vec{\eta}_k(t) \in \text{Pos}_{\Gamma_k^+(t), \frac{3}{x_2}}$, then we note that
\begin{align*}
    \frac{d}{dt}|F_2(\Gamma_k^+(t))-\eta_{k,2}(t)|\Big|_{t=t'} & = \frac{d}{dt}(F_2(\Gamma_k^+(t'))-\eta_{k,2}(t'))\Big|_{t=t'} \\
    & \ge \left(-5\frac{d_1^2}{d_2^2}\left(\frac{1}{R_0} + C\varepsilon_0^2 \right)-(-\Gamma_k''(t'))(\Gamma_k^+)'(t')\cdot e_2\right)\\
    & \ge \left(1-5\frac{d_1^2}{d_2^2}\left(\frac{\varepsilon_0}{|\overline{s}|+4} + C\varepsilon_0^2\right)\right), 
\end{align*}
where we have used \eqref{e:curvature pay off} to estimate $\Gamma_k''(t')(\Gamma_k^+)'(t')\cdot e_2 \ge 1$. In particular, for $0<\varepsilon_0$ sufficiently small 
\begin{align*}
\frac{d}{dt}|F_2(\Gamma_k^+(t'))-\eta_{k,2}(t')|\Big|_{t=t'} >0.
\end{align*}
Putting it all together, if $(\Gamma_k^+)'(t) \cdot e_2 \ge \frac{1}{2}$ then $|F_2(\Gamma_k^+(t)) - \eta_{k,2}(t)|$ is increasing until $t' \ge t$ such that $(\Gamma_k^+)'(t') \cdot e_2 = \frac{1}{2}$.  And for all $t' \in [t, t+2)$ such that $\Gamma_k^+(t') \in B_2(\Gamma_k^+(t))$ and $(\Gamma_k^+)'(t') \cdot e_2 = \frac{1}{2}$, $t' \in \text{Pos}_{\Gamma_k^+(t'), \frac{3}{x_2}}$.  This proves the claim.

Now, we obtain our contradiction.  By our claim and \eqref{e:curvature pay off}, for all $t' \in [t,t+2)$, we have that $(\Gamma_k^+)''(t')\ge 2$. But, if this is the case, then
\begin{align*}
    \Big|\angle((\Gamma_k^{+})'(t'), (\Gamma_k^+)'(t))\Big| = | \int_{t_1}^{t}(\Gamma_k^+)''(s)\,ds| \ge 2(t'-t).
\end{align*}
We see that the angle $\angle((\Gamma_k^{+})'(t'), (\Gamma_k)^+)'(t))$ must exceed $\pi$ for some $t' \in (t, t+2)$.  That is, $\Gamma_k^+$ must violate weak graphicality (Lemma \ref{l:properties of Ik}(3)).  Therefore, no such $t$ can exist. This proves the lemma.
\end{proof}

Now, we list some easy corollaries and porisms of Lemma \ref{l:preliminary gradient control Comparison}.

\begin{corollary}[Lipschitz control in cones]\label{c:lipschitz control in cones}
 Let $A \in S^+_{2\times 2}(\mathbb{R})$ and $\alpha \in (0,1)$.  For all $0<\varepsilon_0$ sufficiently small and all $k\in \mathbb{N}$ sufficiently large, $$\partial I_k \cap \Big[\Big(B_{\varepsilon_0}(\{x_1 = \ell_k^+\}) \cap \Omega_+\Big) \cup \Big(B_{\varepsilon_0}(\{x_1 = \ell_k^-\}) \cap \Omega_-\Big)  \cup \Omega_{k,\varepsilon_0}\Big]$$ is contained in the graph of a function $u_k:\mathbb{R}\rightarrow \mathbb{R}$ and $\emph{Lip}(u_k) \le 3\varepsilon_0$.
\end{corollary}

\begin{proof}
Without loss of generality, we may assume that $0<\varepsilon_0$ is sufficiently small so that Lemma \ref{l:preliminary gradient control Comparison} holds for all $0<\varepsilon \le 2\varepsilon_0$. Recall that this implies $C(\varepsilon_0) \le \frac{4}{3}$ in \eqref{e:F_2 estimate}. Let $R_0(\varepsilon_0,\overline{s})=\frac{|\overline{s}|+4}{\varepsilon_0}$.   

For this $\varepsilon_0, R_0$ Lemma \ref{l:properties of Ik} and Lemma \ref{l:center strip control} give that for all sufficiently large $k \in \mathbb{N}$ we have that
    \begin{align*}
        \partial I_k \cap \Big(([-8R_0, 8R_0]) \times \mathbb{R}\Big) = \text{graph}_{\partial H(e_1, \overline{s}) \cap \Big(([-8R_0, 8R_0]) \times \mathbb{R}\Big)}u_k
    \end{align*}
and $\text{Lip}(u_k; [-8R_0, 8R_0]) \le \frac{1}{2}\varepsilon_0$ and $\norm{u_k}_{L^{\infty}([-8R_0, 8R_0])} \le \frac{1}{2}\varepsilon_0$.

Now, we see that
\begin{align*}
    B_2(\Omega_{k,\varepsilon_0} \cap \{|x_2| \ge 8R_0\}) \subset \Omega_{k,2\varepsilon_0}.
\end{align*}
Thus, for all points $x\in \partial I_k \cap \Omega_{k,\varepsilon_0} \cap \{|x_2| \ge 8R_0\}$, $B_2(x) \subset \Omega_{k,2\varepsilon_0} \cap \{|x_2| \ge R_0\}$.  Therefore, for $x\in \partial I_k \cap \Omega_{k,\varepsilon_0} \cap \{|x_2| \ge 8R_0\}$ and the $t \in \mathbb{R}$ satisfying $\Gamma_k^{\pm}(t) = x$, Lemma \ref{l:preliminary gradient control Comparison} gives 
\begin{align*}
        |\eta_{k,2}(t) - F_2(\Gamma_k^{\pm}(t))| \le \frac{3}{|x_2|}
    \end{align*}
for all sufficiently large $k \in \mathbb{N}$.  That is, by \eqref{e:F_2 estimate} for all such $t$ 
\begin{align*}
    |(\Gamma_k^+)'(t)\cdot e_1| \le \frac{4}{3}\frac{d_1^2}{d_2^2}\varepsilon_0 + \frac{3}{4}\varepsilon_0 <3\varepsilon_0,
\end{align*}
where we have used the fact that $0<\varepsilon_0$ is chosen to be sufficiently small so that in $\Omega_{k,\varepsilon_0}$, $C(\varepsilon_0)\le \frac{4}{3}$ and $x_2 \ge R_0 =\frac{|\overline{s}|+4}{\varepsilon_0}$.  This gives the corollary.
\end{proof}

\begin{remark}
    Note that in the case that the matrix $A$ is strictly anisotropic ($d_1^2 <d_2^2$), we may choose $\varepsilon_0$ sufficiently small so that $\frac{d_1^2}{d_2^2}C(\varepsilon_0)<1$. Therefore, choosing $R(\varepsilon_0, \overline{s})$ sufficiently large so that $\frac{d_1^2}{d_2^2}C(\varepsilon_0)\varepsilon_0 + \frac{3}{R}<\varepsilon_0$, we see that in the strictly anisotropic case we obtain global Lipschitz control on $\partial I_k$, since if $|u_k'|<\varepsilon_0$, then the graph cannot escape the control region $\Omega_{k,\varepsilon_0}$.  This argument is not available in the isotropic case $d_1^2 = d_2^2$.
\end{remark}

\begin{porism}[Second Comparison Result]\label{l:second comparison easy cases}
    Let $A\in S^+_{2\times 2}(\mathbb{R})$ and $\alpha \in (0,1)$.  Let $\Gamma_k^{\pm}:\mathbb{R} \rightarrow \mathbb{R}^2$ be as in Lemma \ref{l:properties of Ik}. For all $0<\varepsilon_0$ sufficiently small there exists an $0<R_0(\varepsilon_0, \alpha)<\infty$ such that the following statement holds for all sufficiently large $k \in \mathbb{N}$.
    \begin{enumerate}
        \item If $\Gamma_k^{\pm}(t) \in \{|x_2| \ge R_0\} \setminus \Omega_{k, \varepsilon_0}$ and $\Gamma_k^{\pm}(t)\cdot e_1 \ge \ell_k^{\pm}$, then 
    \begin{align*}
    (\Gamma_k^{+})'(t)\cdot e_1 > 0, \qquad -(\Gamma_k^{-})'(t)\cdot e_1 > 0.
    \end{align*}
    \item If $\Gamma_k^{\pm}(t) \in \{|x_2| \ge R_0\} \setminus \Omega_{k,\varepsilon_0}$ and $\Gamma_k^{\pm}(t)\cdot e_1 \le \ell_k^{\pm}$, then 
    \begin{align*}
       (\Gamma_k^{+})'(t)\cdot e_1 < 0, \qquad -(\Gamma_k^{-})'(t)\cdot e_1 < 0.
    \end{align*}
    \end{enumerate} 
\end{porism}

\begin{proof}
By reflection, we reduce to arguing for $\Gamma^+_k$.  Let $0<\varepsilon_0$ be sufficiently small so that the argument of Lemma \ref{l:preliminary gradient control Comparison} holds. Let $R:= \frac{|\overline{s}|+4}{\varepsilon_0}$. Assume for the sake of contradiction that there exists a $t \in (0,\infty)$ such that
    \begin{align*}
        B_2(\Gamma_k^{+}(t)) \in \{|x_2| \ge R\} \setminus \Omega_{k,\varepsilon_0} \quad \text{and} \quad \Gamma_k^{+}(t)\cdot e_1 \ge \ell_k^+,
    \end{align*} and $(\Gamma_k^{+})'(t)\cdot e_1 \le 0$.   By the smoothness of $\Gamma_k^{\pm}$ in $\{|x_2| \ge R_0\}$, we may take the infimum of all such $t$ and reduce to considering $(\Gamma_k^{+})'(t)\cdot e_1 = 0$.  Furthermore, since $R_0= \frac{|\overline{s}|+4}{\varepsilon_0}$, the assumption that $B_2(\Gamma_k^{+}(t)) \in \{|x_2| \ge R_0\} \setminus \Omega_{k,\varepsilon_0}$ immediately implies that for all $t' \in [t, t+2)$,  $\Gamma_k^+(t') \in B_2(\Gamma_k^{+}(t))$ and $\Gamma_k^+(t') \cdot e_1 \ge 4$.  
    
    Combining these conditions, we see that for all sufficiently large $k \in \mathbb{N}$, we have $\ell_k^+ \le 1$ by Lemma \ref{l:continuity of Ck}, and
\begin{align*}
    (\Gamma_k^+)''(t') & = d_1^2\ell_k^+ - d_1^2(\Gamma_k^+(t)\cdot e_1)\eta_{k, 1}(t) - d_2^2(\Gamma_k^+(t)\cdot e_2)\eta_{k,2}(t) \\
    & \le d_1^2(1 - 4\eta_{k,1}(t')).
\end{align*}
In particular, for $t'\in [t, t+2)$ such that $\eta_{k,1}(t') \ge \frac{3}{4}$, recalling our normalizing assumption that $d_1 =1$, we obtain $(\Gamma_k^+)''(t') \le -2$.  On the other hand, for $t'\in [t, t+2)$ such that $\eta_{k,1}(t') \le \frac{3}{4}$, we have $\eta_{k,2}(t') \ge \frac{\sqrt{7}}{4}$, so  
\begin{align*}
    (\Gamma_k^+)''(t') & = d_1^2\ell_k^+ - d_1^2(\Gamma_k^+(t)\cdot e_1)\eta_{k, 1}(t) - d_2^2(\Gamma_k^+(t)\cdot e_2)\eta_{k,2}(t) \\
    & \le d_1^2(1 - R_0\frac{\sqrt{7}}{4}),
\end{align*}
where we have used the weak graphicality of $\partial I_k$ (Lemma \ref{l:properties of Ik}) to see that $\eta_{k, 1}(t)\ge 0$. By possibly taking $\varepsilon_0$ smaller so that $\varepsilon_0 \le 1/2$, we have that for these $t'\in [t, t+2)$, $(\Gamma_k^+)''(t') \le 1 - \frac{|\overline{s}|+4}{\varepsilon_0}\frac{\sqrt{7}}{4}\le -2$.

Thus, for all $t'\in [t, t+2)$, $(\Gamma_k^+)''(t') \le -2$.  Therefore, we obtain the same contradiction as in Lemma \ref{l:preliminary gradient control Comparison}.

In the case that 
\begin{align*}
        \Gamma_k^{+}(t) \in \{|x_2| \ge R_0\} \setminus \Omega_{k,\varepsilon_0} \quad \text{and} \quad \Gamma_k^{+}(t)\cdot e_1 \le \ell_k^+,
    \end{align*} and $(\Gamma_k^{+})'(t)\cdot e_1 \ge 0$, the argument is symmetric.
\end{proof}

\begin{porism}[Graphicality]\label{porism graphicality}
Let $A\in S^+_{2\times 2}(\mathbb{R})$ and $\alpha \in (0,1)$.  Let $\Gamma_k^{\pm}:\mathbb{R} \rightarrow \mathbb{R}^2$ be as in Lemma \ref{l:properties of Ik}. For all sufficiently large $k \in \mathbb{N}$, there exists a function $u_k: \mathbb{R}\rightarrow \overline{\mathbb{R}}$ such that 
\begin{align*}
    \emph{Image}(\Gamma_k^{\pm})=\partial I_k = \emph{graph}_{\partial H(e_1, \overline{s})}u_k.
\end{align*}
\end{porism}

\begin{remark}
    The content of Porism \ref{porism graphicality} is that $\partial I_k$ is graphical, not merely weakly graphical.
\end{remark}

\begin{proof}
    Let $0<\varepsilon_0$ be sufficiently small and $0<R_0(\varepsilon_0, \alpha)<\infty$ be as in the proof of Porism \ref{l:second comparison easy cases}.  Without loss of generality, we may assume that $0<\varepsilon_0$ is sufficiently small that we may invoke the results of Corollary \ref{c:lipschitz control in cones} for all sufficiently large $k \in \mathbb{N}$.  Thus, it is sufficient to show that $(\Gamma_k^+)'(t) \cdot e_2 >0$ in the region
    \begin{align*}
        \{|x_2|\ge R_0\} \setminus \Omega_{k,\varepsilon_0}.
    \end{align*}
    By reflection, we reduce to arguing for $\Gamma_k^+$.
    
    Recall that in Porism \ref{l:second comparison easy cases}, we took $R_0 = \frac{|\overline{s}+4|}{\varepsilon_0}$.  In particular, 
    \begin{align*}
        \{x_2\ge R_0\} \setminus \Omega_{k,\varepsilon_0}  \subset \{x_2 \ge R_0, x_1 \ge 4\} \cup \{x_2 \ge R_0, x_1 \le 2\overline{s} - 4\}.
    \end{align*}

But, for any point $\Gamma_k^+(t) = (x_1, x_2) \in \{x_2 \ge R_0, x_1 \ge 4\}$,  $-e_2 \in \text{Pos}_{\Gamma_k^+(t)}$.  Moreover, by Lemma \ref{r:description of vector fields}, we have for all sufficiently large $k \in \mathbb{N}$
\begin{align*}
    F_2(x_1, x_2) \ge \cos\left(\pi - \frac{(-d_1^2\ell_k^+ + d_2^2x_2)}{d_1^2 x_1 + d_2^2 x_2 \frac{\pi}{4}}\right) >-1.
\end{align*}
Therefore, if $(\Gamma_k^+)'(t) = e_1$, it must be that $(\Gamma_k^+)''(t)>0$.  But, this gives a contradiction with weak graphicality (Lemma \ref{l:properties of Ik}). 

The argument for $\Gamma_k^+(t) = (x_1, x_2) \in \{x_2 \ge R, x_1 \ge 1\}$ is symmetric, since $e_2 \in \text{Neg}_{\Gamma_k^+(t)}$.
\end{proof}

\begin{theorem}[Third Comparison Result]\label{t:finer estimates on gradient}
Let $A \in S^+_{2\times 2}(\mathbb{R})$ as in Definition \ref{d:standing assumptions}.  Let $0<\varepsilon_0$ be sufficiently small so that Corollary \ref{c:lipschitz control in cones} holds for all $0<\varepsilon_1 \le \varepsilon_0$. Let $k \in \mathbb{N}$ be sufficiently large so that $\partial I_k = \emph{graph}(u_k)$ as in Porism \ref{porism graphicality}.

Let $1<C$ and $0<h_0<h_1$. For $0<h_1<<\varepsilon_0$ sufficiently small, there exists a $k_2(d_1, d_1, h_0,h_1,C) \in \mathbb{N}$ such that for all $k \ge k_2$, the following properties hold.
\begin{enumerate}
       \item If $0<x_2<\infty$, then
       \begin{enumerate}
           \item $u_k(x_2)+\overline{s} \in [\ell_k^{+} + h_0, \ell_k^{+} + h_1]$ implies  
\begin{align}\label{e:fine grad upper est}
       u_k'(x_2) \le C\frac{d_1^2}{d_2^2}\frac{(u_k(x_2) + \overline{s}) - \ell_k^+}{x_2}.
   \end{align}
   \item $u_k(x_2)+\overline{s} \in [\ell_k^+ - h_0, \ell_k^+ - h_1],$ then 
\begin{align}\label{e:fine grad lower est 2}
       u_k'(x_2) \ge C\frac{d_1^2}{d_2^2}\frac{(u_k(x_2) + \overline{s}) - \ell_k^+}{x_2}.
\end{align}
       \end{enumerate}
   \item If $-\infty<x_2<0$, then
       \begin{enumerate}
           \item $u_k(x_2)+\overline{s} \in [\ell_k^{-} + h_0, \ell_k^{-} + h_1]$ implies  
\begin{align}\label{e:fine grad upper est 3}
       -u_k'(x_2) \le C\frac{d_1^2}{d_2^2}\frac{(u_k(x_2) + \overline{s}) - \ell_k^-}{|x_2|}.
   \end{align}
   \item $u_k(x_2)+\overline{s} \in [\ell_k^- - h_0, \ell_k^- - h_1],$ then 
\begin{align}\label{e:fine grad lower est 4}
       -u_k'(x_2) \ge C\frac{d_1^2}{d_2^2}\frac{(u_k(x_2) + \overline{s}) - \ell_k^-}{|x_2|}.
\end{align}
       \end{enumerate}
   \end{enumerate}
\end{theorem}

\begin{proof}
\textbf{Preliminary reductions and definitions.} The arguments for $\Omega_-$ follow from the argument for $\Omega_+$ by reflection.  Further, the $(u_k(x_2) + \overline{s}) - \ell_k^{\pm}<0$ and $(u_k(x_2) + \overline{s}) - \ell_k^{\pm}>0$ cases are identical.  We therefore only present the $(u_k(x_2) + \overline{s}) - \ell_k^+>0$ argument.

First, let the constant $C(d_1, d_2)$ be as in Lemma \ref{l:G derivatives bound}. Next, we choose $h_1<<\varepsilon_0$ such that
\begin{align*}
    \frac{d_1}{5}\sqrt{\frac{C-1}{C(d_1, d_2)}}\sqrt{h_1} \le \varepsilon_0.
\end{align*}
For convenience, we write $C(d_1, d_2, C) := \frac{d_1}{5}\sqrt{\frac{C-1}{C(d_1, d_2)}}$ and define \begin{align*}
    0<\varepsilon_1 := C(d_1, d_2, C)\sqrt{h_0} \le  C(d_1, d_2, C)\sqrt{h_1} \le \varepsilon_0
\end{align*}
For this $\varepsilon_1$, Corollary \ref{c:lipschitz control in cones} gives that for all sufficiently large $k \in \mathbb{N}$, 
\begin{align*}
   |u_k'|\le 3C(d_1, d_2,C)\sqrt{h_0} \quad \text{in} \quad  B_{\varepsilon_1}(\{x_1 = \ell_k^+\})\cup \Omega_{k,\varepsilon_1}.
\end{align*}
Note that this implies that for for all $0<t$ such that $\Gamma_k^+(t) = (x_2, \overline{s}+\tilde{u}_k(x_2))\in B_{\varepsilon_1}(\{x_1 = \ell_k^+\})\cup \Omega_{k,\varepsilon_1},$
\begin{align*}
    (\Gamma_k^+)'(t) \cdot e_1 \ge \tilde{u}_k'(x_2) \sqrt{1-9\varepsilon_1^2}. 
\end{align*}

Now, by possibly taking $0<h_1$ smaller, we can ensure that 
\begin{align}\label{e:third comparison ensures pos curvature}
    C(\varepsilon_1) \le \sqrt{2}, \quad \text{and}\quad  C\sqrt{1-9\varepsilon_1^2} - C(\varepsilon_1) \ge \frac{1}{2}(C-1)>0, 
\end{align}
where $C(\varepsilon_1)$ is the function appearing in \eqref{e:F_2 estimate} of Lemma \ref{r:description of vector fields} and $C>1$ is the constant in the hypotheses of the theorem.

Finally, we let $M:= \frac{10}{(C-1)^{\frac{1}{2}}}\sqrt{3C(d_1, d_2, C)}$ and $0<R_1=R_1(C, h_0, h_1, d_1, d_2)<\infty$ be such that 
\begin{align}\label{e:third comparison large M containment}
\Big(B_{Mh^{-1/4}_0 + h_1}(\{x_1 = \ell_k^+\}) \cap \{x_2\ge R_1\}\Big) &\subset \Omega_{k,\varepsilon_1} \cap \{\frac{|x_1|}{|x_2|}\le  \varepsilon_1\}\\ \label{e:third comparison R large compared to h_0}
\frac{100}{R_1} & \le d_1^2(C-1)h_0.
\end{align}

\textbf{The argument.} Now, we take $k \in \mathbb{N}$ sufficiently large such that, $\partial I_k$ is a graph (Porism \ref{porism graphicality}) and $|u_k(x_2)-\ell_k^+| < \frac{1}{2}h_0$ for all $x_2 \in (-R_1, R_1)$ (Lemma \ref{l:properties of Ik}(5), Lemma \ref{l:continuity of Ck}, and Corollary \ref{c:lipschitz control in cones}), where $1<<R_1<\infty$ is as above. For these $k\in \mathbb{N}$, we assume for the sake of contradiction that at some $x_2\ge R_1$, $u_k(x_2) \in [\ell_k^+ + h_0, \ell_k^+ + h_1]$ and  
    \begin{align*}
        u_k'(x_2) \ge C\frac{d_1^2}{d_2^2}\frac{(u_k(x_2)+\overline{s}) - \ell_k^+}{x_2}.
    \end{align*}

Since $\partial I_k = \Gamma_k^\pm(\mathbb{R}) = \text{graph}_{\partial H_2(e_1, \overline{s})}u_k$, we let $t_{x_2} \in (R_1, \infty)$ be the unique time such that $\Gamma_k^+(t_{x_2}) = (x_2, u_k(x_2))$.  Then, by our assumptions on $R_1$ and $k$, we have 
\begin{enumerate}
    \item By \eqref{e:third comparison large M containment}, $\Gamma_k^+(t_{x_2}) \in \Omega_{k,\varepsilon_1} \cap \{\frac{|x_1|}{x_2}\le \varepsilon_1\}$. 
    \item And, $\vec{\eta}_k(t_{x_2}) \in \text{Pos}_{\Gamma_{k}^+(t_{x_2})}$, since by \eqref{e:third comparison ensures pos curvature}
    \begin{align*}
    (\Gamma_k^+)'(t_{x_2}) \cdot e_1 \ge C\frac{d_1^2}{d_2^2}\frac{(u_k(x_2)+\overline{s}) - \ell_k^+}{x_2}\sqrt{1-9\varepsilon_1^2} > C(\varepsilon_1)\frac{d_1^2}{d_2^2}\frac{(u_k(x_2)+\overline{s}) - \ell_k^+}{x_2}>0. 
\end{align*}
\end{enumerate} 
Combining these two with the curvature estimates of Lemma \ref{l:curvature bounds by angle}, \eqref{e:third comparison ensures pos curvature} again gives
\begin{align} \nonumber
(\Gamma_k^+)''(t_{x_2}) & \ge (1-4\varepsilon_1)d_2^2x_2\Big|\eta_{k,2}(t_{x_2}) - F_2(\Gamma_k^+(t_{x_2}))\Big|\\ \nonumber
    & = (1-4\varepsilon_1)d_2^2x_2\Big|(\Gamma_k^+)'(t_{x_2}) \cdot e_1 + F_2(\Gamma_k^+(t_{x_2}))\Big|\\\nonumber
    & \ge (1-4\varepsilon_1)d_2^2x_2\left|C\sqrt{1-9\varepsilon_1^2} - C(\varepsilon_1)\right|\frac{d_1^2}{d_2^2}\frac{(u_k(x_2)+\overline{s}) - \ell_k^+}{x_2}\\ \label{e: third comparison curvature bound at tx2}
    & \ge (1-4\varepsilon_1)d_1^2\frac{1}{2}(C-1)\Big((u_k(x_2)+\overline{s})-\ell_k^+\Big) \ge (1-4\varepsilon_1)d_1^2\frac{1}{2}(C-1)h_0.
\end{align}

Now, by \eqref{e:third comparison large M containment}, we may apply Lemma \ref{l:G derivatives bound} to all $t \ge t_{x_2}$ such that $\Gamma_k^{+}(t) \in B_{Mh_0^{-\frac{1}{4}}}(\Gamma_k^+(t_{x_2}))$. Thus, obtain 
\begin{align}\nonumber
    |\frac{\partial}{\partial v}G_1(\Gamma_k^+(t))| & \le  5 \frac{d_1^2}{d_2^2}\left(\frac{1}{\Gamma_k^+(t) \cdot x_2}+C(d_1, d_2)\frac{|\Gamma_k^+(t) \cdot x_1|}{| \Gamma_k^+(t) \cdot x_2|}\Big|\frac{\ell_k^+ - \Gamma_k^+(t) \cdot x_1}{\Gamma_k^+(t) \cdot x_2}\Big|\right)\\ \nonumber
    & \le 5(\frac{1}{R_1} + C(d_1, d_2)\varepsilon_1^2)\\ \label{e:third comparison Gderivative bound}
    & \le (\frac{d_1^2}{20}(C-1)h_0 +\frac{d_1^2}{20}(C-1)h_0) \qquad \qquad \text{by }\eqref{e:third comparison R large compared to h_0} 
\end{align}
for all $v\in \mathbb{S}^1$.

Thus, in particular, combining \eqref{e: third comparison curvature bound at tx2} and \eqref{e:third comparison Gderivative bound} we have that
\begin{align*}
(\Gamma_{k}^+)''(t_{x_2}) - \frac{\partial}{\partial (\Gamma_k^+)'(t_{x_2})}\vec{G}(\Gamma_{k}^+(t_{x_2})) > (\frac{d_1^2}{3} - \frac{d_1^2}{10})(C-1)h_0 > \frac{1}{5}d_1^2(C-1)h_0>0.
\end{align*}

Now, we claim that for all $t \ge t_{x_2}$ such that 
\begin{align}\label{e:containment for 3rd comparison}
    \Gamma_k^+(t) \in B_{Mh_0^{-\frac{1}{4}}}(\{x_1 = \ell_k^+\}) \cap \{x_2 \ge R\} \subset \Big(\Omega_{k,\frac{1}{2}\varepsilon_1} \cap \{\frac{|x_1|}{|x_2|}< \varepsilon_1\}\Big),
\end{align}
we have 
\begin{align}\label{e:curvature claim for 3rd comparison}
    (\Gamma_{k}^+)''(t) & > \frac{1}{5}d_1^2(C-1)h_0 (1+\frac{1}{2}(t-t_{x_2})).
\end{align}
Let $U \subset [t_{x_2}, \infty)$ be the collection of all points which satisfy \eqref{e:containment for 3rd comparison} and \eqref{e:curvature claim for 3rd comparison}. Since both are open conditions, by smoothness of $\Gamma_k^{\pm}$ this set contains a relatively open neighborhood $U$ around $t_{x_2}$.  Let $t_0 \ge t_{x_2}$ be the supremum of $U$. Then, by Lemma \ref{l:curvature bounds by angle} we see
\begin{align*}
    (\Gamma_k^+)''(t_0) & \ge (1-4\varepsilon_0)d_2^2(\Gamma_k^+)_{e_2}\Big((\Gamma_k^+)'(t_0)\cdot e_1 - G_1(\Gamma_k^+(t_0))\Big)\\
    & \ge (1-4\varepsilon_0)d_2^2x_2\Big((\Gamma_k^+)'(t_{x_2})\cdot e_1 - G_1(\Gamma_k^+(t_{x_2})) + \int_{t_{x_2}}^{t_0}(\Gamma_k^+)''(t')\cdot e_1 - \partial_{(\Gamma_k^+)'(t')}G_1(\Gamma_k^+(t'))dt'\Big)\\
    & \ge (1-4\varepsilon_0)d_1^2\frac{1}{2}(C-1)h_0\\
    & \quad + (1-4\varepsilon_0)d_2^2x_2\int_{t_{x_2}}^{t_0} (\frac{1}{5}d_1^2(C-1)h_0 (1+\frac{1}{2}(t-t_{x_2}))) - \frac{1}{10}d_1^2(C-1)h_0 dt'\\
    & \ge (1-4\varepsilon_0)d_1^2\frac{1}{2}(C-1)h_0\\
    & \quad + (1-4\varepsilon_0)d_2^2x_2(\frac{1}{5}d_1^2(C-1)h_0(t_0- t_{x_2})  + \frac{1}{20}d_1^2(C-1)h_0(t_0-t_{x_2})^2).
\end{align*}
In particular, we see that all $t \ge t_{x_2}$ such that $\Gamma_k^+(t) \in B_{Mh_0^{-\frac{1}{4}}}(\Gamma_k^+(t_{x_2}))$ satisfy \eqref{e:curvature claim for 3rd comparison}, as claimed.

In particular, since by assumption, $\Gamma_k^+(t_{x_2})\cdot e_1 \in [\ell_k^+ + h_0, \ell_k^+ + h_1]$, the curvature estimates of \eqref{e:curvature claim for 3rd comparison} must hold for all $t \in [t_{x_2}, t_{x_2} + Mh_0^{-\frac{1}{4}})$. Therefore, we calculate for $t \in [t_{x_2}, t_{x_2} + Mh_0^{-\frac{1}{4}})$
\begin{align*}
    \angle((\Gamma_k^+)'(t_{x_2}), (\Gamma_{k}^+)'(t)) & = \int_{t_{x_2}}^t (\Gamma_k^+)''(s)\, ds\\
    & \ge \int_{t_{x_2}}^t \frac{1}{5}d_1^2(C-1)h_0 (1+\frac{1}{2}(t-t_{x_2}))\, ds\\
    & \ge \frac{1}{5}d_1^2(C-1)h_0  ((t-t_{x_2})+\frac{1}{4}(t-t_{x_2})^2).
\end{align*}
Letting $t \rightarrow t_{x_2}+ Mh_0^{-\frac{1}{4}}$, we see 
\begin{align*}
    \angle((\Gamma_k^+)'(t_{x_2}), (\Gamma_{k}^+)'(t_{x_2} + Mh_0^{-\frac{1}{4}})) & \ge \frac{1}{5}d_1^2(C-1)h_0  (Mh_0^{-\frac{1}{4}}+\frac{1}{4}M^2h_0^{-\frac{1}{2}})\\
    & \ge \frac{1}{5}d_1^2(C-1)(h_0^{\frac{3}{4}}+\frac{1}{4}M^2h_0^{\frac{1}{2}}).
\end{align*}
But, by choice of $M$, $$\frac{1}{5}d_1^2(C-1)\frac{1}{4}M^2h_0^{\frac{1}{2}} > 3C(d_1, d_2, C)h_0^{\frac{1}{2}},$$ a contradiction with the gradient bounds obtained from our choice of $\varepsilon_1$. Thus, the lemma holds.
\end{proof}

\section{Truncation}\label{s:truncation}

In the previous section, we showed that $\partial I_k$ is the graph of a function $u_k:\mathbb{R}\rightarrow \overline{\mathbb{R}}$.  However, we did not obtain global Lipschitz bounds on $u_k$ or global height bounds.  To achieve both global Lipschitz bounds and global height bounds -- thereby producing a sequence of sets $\tilde{E}_k$ in the Small Perturbation Regime -- we must truncate the sets $I_k$.  Throughout this section, we shall use Lemma \ref{l:rescaling d_1 is ok} to reduce to the normalized case
\begin{align*}
    1=d_1 \le d_2.
\end{align*}
We shall still keep track of the dependence on $d_1, d_2$.

\begin{definition}\label{d:truncation functions}(Symmetric Truncations)
Let $A\in S^+_{2\times 2}(\mathbb{R})$ and $\alpha \in (0,1/2]$ be fixed.  Let $u_k:\mathbb{R}\rightarrow \overline{\mathbb{R}}$ as in Porism \ref{porism graphicality} so that $I_k = E_{u_k}$ for all sufficiently large $k \in \mathbb{N}$. 

For any $0<h$, we define the following truncations of $u_k$.
\begin{align*}
    T_h(u_{k}) := \begin{cases} u_k & \text{where } u_k \in [-h,h]\\
    h \frac{u_k}{|u_k|} & \text{where } u_k \not\in [-h,h]
\end{cases}
\end{align*}
\end{definition}

The main result of this section is to show the following theorem. 

\begin{theorem}[Truncation]\label{t:full truncation result}
Let $A\in S^+_{n\times n}(\mathbb{R})$, $A_{1,2}$ be as in Definition \ref{d:standing assumptions}(4), and $\alpha \in (0,1/2]$ be fixed.  Let $u_k:\mathbb{R}\rightarrow \overline{\mathbb{R}}$ as in Porism \ref{porism graphicality} so that $I_k = E_{u_k}$ for all sufficiently large $k \in \mathbb{N}$. 

There exists a sequence $h_k \searrow 0^+$ such that
\begin{align*}
    \lim_{k \rightarrow \infty}\frac{\gamma_{A_{1,2}}(I_k \Delta E_{T_{h_k}(u_k)})}{\gamma_{A_{1,2}}(I_k \Delta H_2(e_1, \overline{s}))^8} = 0,\\
    \lim_{k \rightarrow \infty}\frac{\left|\emph{Per}_{\gamma_{A_{1,2}}}(I_k) - \emph{Per}_{\gamma_{A_{1,2}}}(E_{T_{h_k}(u_k)})\right|}{\gamma_{A_{1,2}}(I_k \Delta H_2(e_1, \overline{s}))^8} = 0.
\end{align*}
\end{theorem}

The proof of Theorem \ref{t:full truncation result} relies upon estimates on the change in mass and perimeter under truncations. For clarity of argument, we shall reduce to proving these estimates in detail for upper truncations in the quadrant $\{x_2 \ge 0\} \cap \{x_1 \ge \overline{s}\}$. The arguments are exactly the same for lower truncations in $\{x_2 \ge 0\} \cap \{x_1 \le \overline{s}\}$.  As usual, we obtain identical results in $\{x_2 \le 0\}$ by reflection.

\begin{definition}\label{d:truncation functions 1-sided}(One-sided truncations)
Let $A\in S^+_{n\times n}(\mathbb{R})$, $A_{1,2}$ be as in Definition \ref{d:standing assumptions}(4), and $\alpha \in (0,1/2]$ be fixed.  Let $u_k:\mathbb{R}\rightarrow \overline{\mathbb{R}}$ as in Porism \ref{porism graphicality} for all sufficiently large $k \in \mathbb{N}$.  

Assume that there exists a $0<h_2$ such that 
    \begin{align*}
        \mathcal{H}^1(\{u_k>h_2\} \cap \{x_2 \ge 0\})>0
    \end{align*}
for all $k \in \mathbb{N}$ sufficiently large.
\begin{enumerate}
    \item  For any $0<H<\frac{1}{2}h_2$, we define the right $H$-\textit{upper truncation} of $u_k$ as follows
\begin{align*}
    \tilde{u}_{k,H,} := \begin{cases}
    u_k & \text{in } \{x_2 \le 0\}\\
    u_k & \text{in } \{x_2 \ge 0\}\text{ where } u_k \le H\\
    H & \text{in } \{x_2 \ge 0\}\text{ where } u_k > H.
\end{cases}
\end{align*}
\item For any $0<H<\frac{1}{2}h_2$, we define
\begin{align*}
    x_{2,k,H} & =
        \inf \{0<x_2 : u_k(x_2) = H\}\\
    x_2^{k,H} & =
        \sup \{0<x_2<x_{2,k,2H} : u_k(x_2) = H\}.
\end{align*}  
\end{enumerate} 
\end{definition}

Now, we use Theorem \ref{t:finer estimates on gradient} to obtain some estimates on the relationship between $x_{2}^{k,H}$ and $x_{2,k,CH}$ for $0<C<1/2$.

\begin{lemma}\label{l:separation of exit times}
Let $A, A_{1,2}$ be as in Definition \ref{d:standing assumptions}.  Let $u_k:\mathbb{R}\rightarrow \overline{\mathbb{R}}$ be as in Porism \ref{porism graphicality}.  Let $0<H<<1$ be sufficiently small such that we may apply Theorem \ref{t:finer estimates on gradient} with $H=h_1$.

Assume that $0<x_{2, k, H} <\infty$ for all sufficiently large $k \in \mathbb{N}$. Then, for all sufficiently large $k \in \mathbb{N}$, we have that $0<x_2^{k, \frac{1}{50d_2^3}H} <\infty$ and
\begin{align*}
    \frac{x_{2,k, H}}{x_2^{k, \frac{1}{50d_2^3}H}} \ge 10d_2^2.
\end{align*}
\end{lemma}

\begin{proof}
Observe that for any function $f:\mathbb{R}^+\rightarrow \mathbb{R}$ satisfying
    \begin{align*}
        f'(x) \le C\frac{f(x)}{x},
    \end{align*}
    then for any $0<\rho<R<\infty$
    \begin{align*}
        \ln(\frac{f(R)}{f(\rho)}) \le C\ln(R/\rho)\\
        \frac{f(R)}{f(\rho)} \le \left(\frac{R}{\rho}\right)^C.
    \end{align*}
Therefore, consider $f(x_2)=u_k(x_2)+ \overline{s}-\ell_k^+$ with $C = \frac{3}{2}$, $h_1 =H$, and $h_0 = \frac{1}{50d_2^3}H \le \frac{1}{2}H$.  This last estimate is where we use the normalizing assumption that $1 \le d_2$. Now, we see that the assumption that $x_{2, k, H}< \infty$ immediately implies that $x_2^{k, \frac{1}{50d_2^3}H}$ exists, is finite, and satisfies
\begin{align*}
    10d_2^2\le (50d_2^3)^{\frac{2}{3}} \le \frac{x_{2,k, H}}{x_{2}^{k, \frac{1}{25d_2^3}H}}
\end{align*}
for all sufficiently large $k \in \mathbb{N}$.
\end{proof}

Note that Lemma \ref{l:separation of exit times} gives us \textit{multiplicative} growth in $x_2$ to obtain a multiplicative height increase (or decrease) in $u_k - \ell_k$. Lemma \ref{l:preliminary gradient control Comparison} is not strong enough to obtain this multiplicative factor, as
\begin{align*}
    f'(x) \le C\frac{f(x)}{x} + \frac{3}{x} \quad \longrightarrow \quad f(R) \le \left(\frac{R}{r}\right)^{\frac{1}{C}}f(r) + \frac{3}{C}(R-r)R^{\frac{1}{C}}.
\end{align*}

\begin{lemma}[Right upper $H$-truncation area estimates]\label{l:mass error is negligible}
Let $A, A_{1,2}$ be as in Definition \ref{d:standing assumptions}.  Let $k \in \mathbb{N}$ be sufficiently large such that $u_k:\mathbb{R}\rightarrow \overline{\mathbb{R}}$ exists as in Porism \ref{porism graphicality}.  Let $0<H<<1$ be sufficiently small such that we may apply Theorem \ref{t:finer estimates on gradient} with $h_1 =H$. Let $\tilde{u}_{k,H}$ be the function in Definition \ref{d:truncation functions 1-sided}. 
    
    Then, 
    \begin{align*}
        \lim_{k \rightarrow \infty}\frac{\gamma_{A_{1,2}}(E_{u_k} \Delta E_{\tilde{u}_{k,H}})}{\gamma_{A_{1,2}}(I_k \Delta H_2(e_1, s))^8} =0.
    \end{align*}
\end{lemma}

\begin{proof}
We may assume that $|x_{2,k, H}|<\infty$ for a sequence of $k \rightarrow \infty$, since otherwise $u_k\le H$, $\tilde{u}_{k,H} = u_k$ for all sufficiently large $k \in \mathbb{N}$, and there is nothing to prove.  

We begin estimating $\gamma_{A_{1,2}}(E_{u_k} \Delta E_{\tilde{u}_{k,H}})$, as follows.
\begin{align}\nonumber
    \gamma_{A_{1,2}}(E_{u_k} \Delta E_{\tilde{u}_{k,H}}) & = \frac{d_1d_2}{2\pi}\int_{\mathbb{R}}\int_{\tilde{u}_{k,H}(x_2)}^{u_k(x_2)} e^{-\frac{1}{2}(d_1^2 x_1^2 + d_2^2 x_2^2)}dx_1dx_2\\\nonumber
    & = \frac{d_1d_2}{2\pi}\int_{x_{2,k,H}}^{\infty}\int_{\tilde{u}_{k,H}(x_2)}^{u_k(x_2)} e^{-\frac{1}{2}(d_1^2 x_1^2 + d_2^2 x_2^2)}dx_1dx_2\\\nonumber
    & \le  \frac{d_2}{\sqrt{2\pi}}\int_{x_{2,k,H}}^{\infty} e^{-\frac{1}{2}(d_2^2 x_2^2)}dx_2\\\label{e:mass estimate upper bound}
    & \le  e^{-\frac{1}{2}(\frac{x_{2,k,H}}{d_2})^2},
    \end{align}
where in the last line we have used change of variables and the Chernoff bound
\begin{align}\label{e:Chernoff bound} 
    \frac{1}{\sqrt{2\pi}}\int_t^{\infty}e^{-\frac{1}{2}x^2}\,dx \le e^{-\frac{1}{2}t^2}.
\end{align}

Now, we estimate the denominator. Let $h_0 := \frac{1}{50d_2^{3}}H$ and fix $0<\varepsilon_0$ be sufficiently small so that we may invoke the results of Corollary \ref{c:lipschitz control in cones}.  By our normalizing assumption that $1\le d_2$, we have
\begin{align*}
    h_0:=\frac{1}{50d_2^3}H \le H =:h_1.
\end{align*}
Therefore, since we assume that $x_{2,k,H}<\infty$ for all sufficiently large $k \in \mathbb{N}$, $x_{2}^{k,h_0}<\infty$ and, by Lemma \ref{l:properties of Ik}(5), $x_{2,k,h_0} \rightarrow \infty$ as $k\rightarrow \infty$.  Therefore, we have that for all sufficiently large $k \in \mathbb{N}$,
\begin{align*}
    |u_k'| \le 3\varepsilon_0 \quad \text{in}\,\Omega_{k,\varepsilon_0},\qquad
    B_{2h_0}(\{x_1 = \ell_k^+\}) \cap \{|x_2| \ge x_{2}^{k,h_0}\} \subset \Omega_{k, \varepsilon_0}.
\end{align*}
Then, for all sufficiently large $k \in \mathbb{N}$ such that $|\ell_k^+ - \overline{s}|\le \frac{1}{2}h_0$, we have 
\begin{enumerate}
    \item $u_{k}([x_{2}^{k,h_0}, x_{2}^{k,h_0} + \frac{h_0}{8\varepsilon_0}]) \subset B_{2h_0}(\{x_1 = \ell_k^+\})$.
    \item $\tilde{u}_k \res [x_{2}^{k,h_0}, x_{2}^{k,h_0} + \frac{h_0}{8\varepsilon_0}] \ge h_0$. This uses the definition of $x_2^{k, h_0}$ as a supremum.
\end{enumerate}  
Therefore, we estimate for these $k \in \mathbb{N}$
\begin{align*}
    \gamma_{A_{1,2}}(I_k \Delta H_2(e_1, \overline{s})) & \ge \frac{d_1d_2}{2\pi}\int_{\mathbb{R}}\int_{\min\{\overline{s}, \tilde{u}_{k,H}(x_2)\}}^{\max\{\overline{s}, \tilde{u}_{k,H}(x_2)\}} e^{-\frac{1}{2}(d_1^2 x_1^2 + d_2^2 x_2^2)}dx_1dx_2\\
    & \ge \frac{d_1d_2}{2\pi}\int_{x_{2}^{k,h_0}}^{x_{2}^{k,h_0}+\frac{h_0}{8\varepsilon_0}} h_0 e^{-\frac{1}{2}(d_1^2|\overline{s}|^2 + d_2^2 x_2^2)}\, dx_2 \ge  \frac{d_1d_2}{2\pi}\frac{h_0^2}{8\varepsilon_0}e^{-\frac{1}{2}d_1^2 |\overline{s}|^2}e^{-\frac{1}{2}d_2^2 |x_{2}^{k,h_0}+\frac{h_0}{8\varepsilon_0}|^2}.
\end{align*}
Now, we recall Lemma \ref{l:properties of Ik}(5), again, to take $k \rightarrow \infty$ sufficiently large so that $|x_{2}^{k,h_0}+\frac{h_0}{8\varepsilon_0}|\le \sqrt{2}|x_{2}^{k,h_0}|$.  Thus, we see that
\begin{align}\label{e:mass estimate lower bound}
    \gamma_{A_{1,2}}(I_k \Delta H_2(e_1, \overline{s}))
    & \ge \frac{d_1d_2}{2\pi}\frac{h_0^2}{8\varepsilon_0}e^{-\frac{1}{2}d_1^2 |\overline{s}|^2}e^{-d_2^2 |x_{2}^{k,h_0}|^2}.
\end{align}

Putting \eqref{e:mass estimate upper bound} and \eqref{e:mass estimate lower bound} together, we have
\begin{align}\nonumber
    \lim_{k \rightarrow \infty}& \frac{\gamma_{A_{1,2}}(E_{u_k} \Delta E_{\tilde{u}_{k,H}})}{\gamma_{A_{1,2}}(I_k \Delta H_2(e_1, s))^8}\\ \nonumber
    & \le C(d_1, d_2, \varepsilon_0, h_0, \overline{s}) \liminf_{k \rightarrow \infty}\frac{e^{-\frac{1}{2}(\frac{x_{2,k,h_1}}{d_2})^2}}{\left(e^{-d_2^2 |x_{2}^{k,h_0}|^2}\right)^8}\\ \nonumber
& \le C(d_1, d_2, \varepsilon_0, h_0, \overline{s}) \liminf_{k \rightarrow \infty}\frac{e^{-50d_2^2|x_{2}^{k,h_0}|^2}}{e^{-8d_2^2|x_{2}^{k,h_0}|^2}}= 0 \qquad \text{by Lemma } \ref{l:separation of exit times}.
\end{align}
\end{proof}

\begin{lemma}[Right upper $H$-truncation perimeter estimates]\label{l:perimeter error bound}
Let $A, A_{1,2}$ be as in Definition \ref{d:standing assumptions}.  Let $u_k:\mathbb{R}\rightarrow \overline{\mathbb{R}}$ be as in Porism \ref{porism graphicality}.  Let $\tilde{u}_{k,H}$ be the right upper truncation in Definition \ref{d:truncation functions 1-sided}.

Then, for all $0<H<<1$ sufficiently small and all $k \in \mathbb{N}$ sufficiently large,
\begin{align*}
\lim_{k\rightarrow \infty}\frac{\Big|\operatorname{Per}_{\gamma_{A_{1,2}}}(I_k) -\operatorname{Per}_{\gamma_{A_{1,2}}}(E_{\tilde{u}_{k,H}}) \Big|}{\gamma_{A_{1,2}}(I_k\Delta H_2(e_1, \overline{s}))^8} =0.
\end{align*}
\end{lemma}

\begin{proof}
As in Lemma \ref{l:mass error is negligible}, we may assume that for some $0<H'$ there exists a sequence (also indexed by $k \rightarrow \infty$) such that $$J_{k,H'} := \{x_2: u_k(x_2)\not = \tilde{u}_{k,H'}(x_2)\} \not = \emptyset$$ for all $k$, otherwise there is nothing to prove.  We let $0<H$ satisfy 
\begin{enumerate}
    \item $0<H$ is sufficiently small so that we may apply Theorem \ref{t:finer estimates on gradient} with $h_1 = H$.
    \item $0<H \le \frac{1}{3}H'$.
\end{enumerate}
Note that the second assumption guarantees that $x_2^{k,H}<\infty$ for all $k \in \mathbb{N}$.

Additionally, we let $0<\varepsilon_0$ be fixed, sufficiently small so that the results of Corollary \ref{c:lipschitz control in cones} and Porism \ref{porism graphicality} hold with this $\varepsilon_0$ for all sufficiently large $k \in \mathbb{N}$.  

Consider the following sets
\begin{align*}
    G_k &:= \{x_2: (x_2, u_k(x_2)) \in \Omega_{k,\varepsilon_0}\}, \\
    B_k & := \{x_2: (x_2, u_k(x_2)) \not \in \Omega_{k,\varepsilon_0}\}.
\end{align*}
Using the convention that if $u_k = +\infty$ on $[x_{2}, \infty)$ then $\nabla u_k \res [x_{2}, \infty) = 0$, we estimate that for sufficiently large $k \in \mathbb{N}$ as in Porism \ref{porism graphicality},
\begin{align*}
    \text{Per}_{\gamma_{A_{1,2}}}(I_k; J_{k, H} \times \mathbb{R}) & = \frac{d_1d_2}{(2\pi)^{\frac{1}{2}}}\int_{J_{k,H}}\sqrt{1+|u_k'(x_2)|^2}e^{-\frac{1}{2}(d_1^2|\overline{s}+u_{k}(x_2)|^2 + d_2^2x_2^2)}dx_2\\
    & \le \frac{d_1d_2}{(2\pi)^{\frac{1}{2}}} \int_{x_{2,k,H}}^{\infty} \sqrt{1+|u_k'(x_2)|^2} e^{-\frac{1}{2}(d_1^2|\overline{s}+u_{k}(x_2)|^2 + d_2^2x_2^2)}dx_2\\
    & \le \frac{d_1d_2}{(2\pi)^{\frac{1}{2}}}\int_{G_k \cap (x_{2,k,H},\infty)} \sqrt{1+9\varepsilon_0^2} e^{-\frac{1}{2}(d_1^2|\overline{s}+u_{k}(x_2)|^2 + d_2^2x_2^2)}dx_2\\
& \qquad + \frac{d_1d_2}{(2\pi)^{\frac{1}{2}}}\int_{B_k \cap (x_{2,k,H},\infty)} \sqrt{1+|u_k'(x_2)|^2} e^{-\frac{1}{2}(d_1^2|\overline{s}+u_{k}(x_2)|^2 + d_2^2x_2^2)}dx_2,
\end{align*}
where we have used Lemma \ref{c:lipschitz control in cones} to obtain the last line.

Now, we decompose the last term and use projection.
\begin{align} \nonumber
    &\frac{d_1d_2}{(2\pi)^{\frac{1}{2}}}\int_{B_k \cap (x_{2,k,H},\infty)}  \sqrt{1+|u_k'(x_2)|^2} e^{-\frac{1}{2}(d_1^2|\overline{s}+u_{k}(x_2)|^2 + d_2^2x_2^2)}dx_2\\ \nonumber
    & = \frac{d_1d_2}{(2\pi)^{\frac{1}{2}}}\int_{B_k \cap (x_{2,k,H},\infty) \cap \{|u_k'|\le 1\}}  \sqrt{1+|u_k'(x_2)|^2} e^{-\frac{1}{2}(d_1^2|\overline{s}+u_{k}(x_2)|^2 + d_2^2x_2^2)}dx_2\\\nonumber
    & \qquad  + \frac{d_1d_2}{(2\pi)^{\frac{1}{2}}}\int_{B_k \cap (x_{2,k,H},\infty)\cap \{|u_k'|>1\}}  \sqrt{1+|u_k'(x_2)|^2} e^{-\frac{1}{2}(d_1^2|\overline{s}+u_{k}(x_2)|^2 + d_2^2x_2^2)}dx_2\\ \label{e:truncation perimeter upper estimate}
    & \le  d_1\left(\sqrt{2}e^{-\frac{1}{2}(|\frac{x_{2,k,H}}{d_2}|^2)}  + \sqrt{2}e^{-\frac{1}{2}(d_2^2|x_{2,k,H}|^2)}\right),
\end{align}
where the first term is estimated by the gradient bound and \eqref{e:Chernoff bound} and the second estimate comes from projection onto the vertical line $\mathbb{R} \times \{x_{2,k,H'}\}$ (by Lemma \ref{l:second comparison easy cases} the multiplicity function in the projection is always $1$) and the gradient bound.

Now, we obtain a lower bound on $\text{Per}_{\gamma_{A_{1,2}}}(I_k; J_{k, H} \times \mathbb{R})$. Note that for all $k \in \mathbb{N}$ sufficiently large, $$[x_{2}^{k,H}, x_{2}^{k,H} +1] \times [\overline{s} + H, \overline{s} + H + 3\varepsilon_0] \subset \Omega_{k,\varepsilon_0}.$$ Therefore, by the definition of $x_{2}^{k,H}$ as the supremum and the Lipschitz estimates of Corollary \ref{c:lipschitz control in cones}, for all such $k \in \mathbb{N}$
\begin{align*}
    [x_2^{k,H}, x_2^{k,H} +1] \subset G_k \cap J_{k,H}.
\end{align*}

Thus, we may obtain our underestimate, as follows. By integration,
\begin{align}\nonumber
    \text{Per}_{\gamma_{A_{1,2}}}(I_k; J_{k, H} \times \mathbb{R})& \ge \frac{d_1d_2}{(2\pi)^{\frac{1}{2}}}\int_{[x_2^{k,H}, x_2^{k,H} +1]}\sqrt{1+|u_k'(x_2)|^2}e^{-\frac{1}{2}(d_1^2|\overline{s}+u_{k}(x_2)|^2 + d_2^2x_2^2)}dx_2\\ \label{e:perimeter truncation est lower perimeter}
    & \ge \frac{d_1d_2}{(2\pi)^{\frac{1}{2}}}e^{-\frac{1}{2}d_1^2|\overline{s}-H|^2}e^{-\frac{1}{2}d_2^2|x_2^{k,H} +1|^2}.
\end{align}

Now, we use a similar argument to see that
\begin{align}\label{e:perimeter truncation est lower perimeter 2}
    \frac{d_1d_2}{(2\pi)^{\frac{1}{2}}}e^{-\frac{1}{2}d_1^2|\overline{s}+H|^2}e^{-\frac{1}{2}d_2^2|x_2^{k,H} +1|^2} & = \text{Per}_{\gamma_{A_{1,2}}}(E_{\tilde{u}_{k,H}}; [x_2^{k,H, +},x_2^{k,H} +1] \times \mathbb{R})\\
    & \le \text{Per}_{\gamma_{A_{1,2}}}(E_{\tilde{u}_{k,H}}; J_{k,H} \times \mathbb{R})\\
    & \le \frac{d_1d_2}{(2\pi)^{\frac{1}{2}}}\int_{[x_{2,k,H}, \infty)} e^{-\frac{1}{2}d_2^2|x_2|^2} \,dx_2\\ \label{e:perimeter truncation est upper perimeter 2}
    & \le d_1 e^{-\frac{1}{2}|\frac{x_{2,k,H}}{d_2}|^2},
\end{align}
where, again, we have used the Chernoff bound \eqref{e:Chernoff bound}.

Thus, combining \eqref{e:truncation perimeter upper estimate}\eqref{e:perimeter truncation est lower perimeter}\eqref{e:perimeter truncation est lower perimeter 2}\eqref{e:perimeter truncation est upper perimeter 2}, we immediately see that
\begin{align*}
    |\text{Per}_{\gamma_{A_{1,2}}}(I_k) - \text{Per}_{\gamma_{A_{1,2}}}(E_{\tilde{u}_{k,H}})| \le C(d_1, d_2, \overline{s},\varepsilon_0)e^{-\frac{1}{2}|\frac{x_{2,k,H}}{d_2}|^2},
\end{align*}
where, again, we have used our normalizing assumption that $1 \le d_2$.  Therefore, borrowing the denominator estimates in Lemma \ref{l:mass error is negligible} shows
\begin{align}\nonumber
    \lim_{k \rightarrow \infty}& \frac{\text{Per}_{\gamma_{A_{1,2}}}(I_k) - \text{Per}_{\gamma_{A_{1,2}}}(E_{\tilde{u}_{k,H}})}{\gamma_{A_{1,2}}(I_k \Delta H_2(e_1, s))^8}\\ \nonumber
    & \le C(d_1, d_2, \overline{s},\varepsilon_0) \liminf_{k \rightarrow \infty}\frac{e^{-\frac{1}{2}\left|\frac{x_{2,k,H}}{d_2}\right|^2}}{e^{-8d_2^2 \left|x_{2}^{k,\frac{1}{50d_2^3}H}\right|^2}}\\ \nonumber
& \le C(d_1, d_2, \overline{s},\varepsilon_0) \liminf_{k \rightarrow \infty}\frac{e^{-50d_2^2\left|x_{2}^{k,\frac{1}{50d_2^3}H}\right|^2}}{e^{-8d_2^2\left|x_{2}^{k,\frac{1}{50d_2^3}H}\right|^2}}= 0 \qquad \text{ by Lemma \ref{l:separation of exit times}},
\end{align}
where we have used Lemma \ref{l:properties of Ik} to give $\lim_{k \rightarrow \infty}x_{2}^{k,H} = +\infty$.
\end{proof}

\subsection{Proof of Theorem \ref{t:full truncation result}}
First, we note that if for all $0<h$,
\begin{align*}
    \mathcal{H}^1(\{ |u_k|>h\}) =0
\end{align*}
for all sufficiently large $k \in \mathbb{N}$, then we may take $h_k$ for any sequence $h_k \rightarrow 0^+$ such that $h_k \ge \sup_{\mathbb{R}} u_k$ for all sufficiently large $k \in \mathbb{N}$.  Thus, $T_{h_k}(u_k) = u_k$ and there is nothing to prove.

Therefore, we shall assume that there exists an $0<h$
\begin{align*}
    \mathcal{H}^1(\{ |u_k|>h\}) >0
\end{align*}
for all sufficiently large $k \in \mathbb{N}$. We consider 
\begin{align*}
    h_2^+ := \sup_{0\le h<\infty}\{h: \mathcal{H}^1(\{ u_k>h\} \cap \{x_2 \ge 0\}) >0
\text{ for all } k\in \mathbb{N} \}\\
    h_2^- := \sup_{0\le h<\infty}\{h: \mathcal{H}^1(\{ u_k>h\} \cap \{x_2 \le 0\})>0 \text{ for all } k\in \mathbb{N} \},
\end{align*}
where we define $h_2^\pm = 0$ if the relevant set is empty. And,
\begin{align*}
    h_3^+ := \sup_{0\le h<\infty}\{h: \mathcal{H}^1(\{ u_k<-h\} \cap \{x_2 \ge 0\}) >0
\text{ for all } k\in \mathbb{N} \}\\
    h_3^- := \sup_{0\le h<\infty}\{h: \mathcal{H}^1(\{ u_k<-h\} \cap \{x_2 \le 0\})>0 \text{ for all } k\in \mathbb{N} \},
\end{align*}
where we define $h_3^\pm = 0$ if the relevant set is empty. 

Let $V \subset \{h_2^+, h_2^-, h_3^+, h_3^-\}$ be the collection of all non-zero entries.  We define $h_4 := \min V$.  Then, for all $h_j < \frac{1}{3}h_4$, we may apply Lemma \ref{l:perimeter error bound} \textit{mutatis mutandis} for right lower truncations and left upper and lower truncations.  Extending Lemma \ref{l:mass error is negligible} similarly, we obtain the results of Lemma \ref{l:mass error is negligible} and Lemma \ref{l:perimeter error bound} for all sufficiently large $k\in \mathbb{N}$.  

Now, let $h_j \searrow 0^+$ be  any sequence satisfying $h_j < \frac{1}{3}h_4$ for all $j \in \mathbb{N}$.  We may assume further that all $0<h_j$ are sufficiently small so that the Lipschitz control in Corollary \ref{c:lipschitz control in cones} holds with with $\varepsilon_0 =2h_j$. Therefore, for all $j,k \in \mathbb{N}$ we define
\begin{align*}
    k_j:= \max\Big\{k, \inf\big\{k \in \mathbb{N}:& \quad \text{for all }k'\ge k \\
    &\quad \text{Lip}(u_{k'}; [-100,100]) + \norm{u_{k'}}_{L^{\infty}([-100,100])} \le \frac{1}{2}h_j\\
    & \quad \text{Lip}(u_{k'})\le 6h_j \text{ when } (x_2, u_k(x_2)) \in \Omega_{k,2h_j}, \text{ or when }\\
    & \qquad  (x_2, u_k(x_2)) \in \left((B_{2h_j}(\{x_1 = \ell_{k'}^+\}) \cap \Omega_+) \cup (B_{2h_j}(\{x_1 = \ell_{k'}^-\}) \cap \Omega_- )\right).\\
    & \quad \frac{\gamma_{A_{1,2}}(E_{T_{h_j}(u_{k'})}\Delta I_{k'})}{\gamma_{A_{1,2}}(I_{k'} \Delta H_2(e_1, \overline{s}))^8} \le 2^{-j} \\
    & \quad \frac{|\operatorname{Per}_{\gamma_{A_{1,2}}}(E_{T_{h_j}(u_{k'})}) -\operatorname{Per}_{\gamma_{A_{1,2}}}(I_{k'})|}{\gamma_{A_{1,2}}(I_{k'} \Delta H_2(e_1, \overline{s}))^8}\le 2^{-j}\big\}\Big\}.
\end{align*}
By Lemma \ref{l:properties of Ik}, Corollary \ref{c:lipschitz control in cones}, and Lemma \ref{l:mass error is negligible} and Lemma \ref{l:perimeter error bound} (extended \textit{mutatus mutandis} to the truncations $T_{h_j}$), we have that $k_j<\infty$.  Choosing a reindexing of $h_j$ such that 
\begin{align*}
    h_k \in \{h_j\}_j \quad \text{and} \quad h_k = h_j \quad  \text{implies} \quad k \ge k_j,
\end{align*}
the fact that $k_j<\infty$ for all $j \in \mathbb{N}$ implies that there is a way to choose $h_k \rightarrow 0^+$ such that $T_{h_k}(u_k)$ satisfies the claims of the theorem.
\qed

\subsection{Proof of Theorem \ref{t:technical theorem}(2abcd)}\label{ss:Ttheorem 2abcd}

Finally, we are ready to define the sets $\tilde{E}_{k}$. Let $A_{1,2} \in S^+_{2\times 2}(\mathbb{R})$ be as in Definition \ref{d:standing assumptions}(4). Let $\alpha \in (0,\frac{1}{2}]$ be fixed.  Let $I_k$ be as in Lemma \ref{l:properties of Ik}.  Let $\{h_k\}_k$ be the sequence guaranteed by Theorem \ref{t:full truncation result}.  We define
\begin{align*}
    \{\tilde{E}_k\}_k := \{
        E_{T_{h_k}(u_{k})}\}_{k \in \mathbb{N}}.
\end{align*}

\medskip

\textbf{Proof of (2d)}:
By assumption, $h_k$ is sufficiently small so that the Lipschitz control of Corollary \ref{c:lipschitz control in cones} holds with $\varepsilon_0= 2h_k$.  Therefore, since $$\partial E_{T_{h_k}(u_{k})} \subset \left((B_{2h_j}(\{x_1 = \ell_{k'}^+\}) \cap \Omega_+) \cup (B_{2h_j}(\{x_1 = \ell_{k'}^-\}) \cap \Omega_- )\right),$$ 
$\text{Lip}(T_{h_k}(u_{k}))\le 6h_k$.  This immediately gives Theorem \ref{t:technical theorem}(2d).

\medskip

\textbf{Proof of (2c)}:
We estimate
    \begin{align*}
        \lambda_{\gamma_A}(E_k) & = \gamma_{A_{1,2}}(F_k\Delta H_2(e_1, \overline{s})) & \text{Lemma } \ref{l:n-1 symmetrization}\\
        & \le (2+o(k)) \gamma_{A_{1,2}}(H_k\Delta H_2(e_1, \overline{s}))& \text{Corollary } \ref{c:second symmetrization}\\
        & \le 2(2+o(k))\gamma_{A_{1,2}}(I_k\Delta H_2(e_1, \overline{s}))& \text{Lemma } \ref{l:properties of Ik}\\
        & \le 2(2+o(k))\Big(\gamma_{A_{1,2}}(\tilde{E}_{k}\Delta H_2(e_1, \overline{s})) + 2^{-k}\gamma_{A_{1,2}}(I_k\Delta H_2(e_1, \overline{s}))^8\Big) & \text{Lemma } \ref{l:mass error is negligible},
    \end{align*}
    where in the last inequality we have also used the triangle inequality for the total variation norm on measurable sets
    \begin{align*}
        \gamma_{A_{1,2}}(I_k\Delta H_2(e_1, \overline{s})) \le \gamma_{A_{1,2}}(\tilde{E}_{k}\Delta H_2(e_1, \overline{s})) + \gamma_{A_{1,2}}(\tilde{E}_{k}\Delta I_k).
    \end{align*}
Absorbing the last term, we obtain
\begin{align*}
    (1-2^{-k})\lambda_{\gamma_A}(E_k) \le (4 + o(k))(1-2^{-k}) \gamma_{A_{1,2}}(I_k \Delta H_2(e_1, \overline{s}))\le (4 + o(k))\gamma_{A_{1,2}}(\tilde{E}_{k} \Delta H_2(e_1, \overline{s})).
\end{align*}
Therefore, Theorem \ref{t:technical theorem}(2c) holds for all $k \in \mathbb{N}$ sufficiently large.

\medskip

\textbf{Proof of (2b)}:
We estimate,
    \begin{align*}
        \delta_{\gamma_A}(E_k) & \ge \delta_{A_{1,2}}(H_k \Delta H_2(e_1, \overline{s}))& \text{Corollary \ref{c:second symmetrization}(1)}\\ & \ge \delta_{A_{1,2}}(I_k \Delta H_2(e_1, \overline{s})) & \text{Lemma \ref{l:properties of Ik}(1)}\\
        & \ge \delta_{\gamma_{A_{1,2}}}(\tilde{E}_k) - 2^{-k}\lambda_{\gamma_{A_{1,2}}}(I_k \Delta H_2(e_1, \overline{s}))^8 & \text{Lemma \ref{l:perimeter error bound}}.
        \end{align*}
This gives Theorem \ref{t:technical theorem}(2b).

\medskip

\textbf{Proof of (2a)}: Since $\gamma_{A_{1,2}}(I_k) = \gamma_{A}(E_k) = \alpha$ for all $k \in \mathbb{N}$, Theorem \ref{t:technical theorem}(2a) follows from Theorem \ref{t:full truncation result}.
\qed

\section{The Small Perturbation Regime}\label{s:The Small Perturbation Regime}

In this section, we provide Fuglede-type estimates for the Small Perturbation Regime.  Proofs of this type of estimate in \cite{CianchiFuscoMaggiPratelli11,CicaleseLeonardi2012_selectionprinciple,BarchiesiBrancoliniJulin17} rely upon some form of reflectional symmetry to improve the constant in the Sobolev-type inequality.  Our reduction step (Theorem \ref{t:technical theorem}), however, uses $(n-1)$-symmetrization, which eliminates all hope of this kind of argument.  It is notable, however, that the ODE from \eqref{e:ODE} nevertheless allows us to obtain this type of improvement, \textit{despite} our initial $(n-1)$-symmtrization. Thus, this demonstrates that reliance on reflectional symmetry, as is customary in the existing literature, is not essential to the proof.

This section is organized as follows.  First, we develop some elementary Taylor expansion results for $Q(\tilde{E}_k)$ (Lemma \ref{l:fuglede-type estimate}). This gives a crude upper bound on $Q(H_2(e_1, \overline{s}))$ (Lemma \ref{l:proof of remark Q finite}), which we need in our reduction step in Remark \ref{r:Q is finite and consequences}. Next, we develop some results around what happens if $\tilde{u}_k$ is asymptotically linear. The key result, is that if $\tilde{u}_k$ is very close to linear, then $\gamma_{A_{1,2}}(\tilde{E}_k \Delta H_2(e_1, \overline{s})) \simeq \lambda_{\gamma_{A}}(E_k)$ by universal constants (see Lemma \ref{l:close to linear and comparability of lambdas}).  This comparability allows us to prove that for all sufficiently large $k \in \mathbb{N}$, the $\tilde{u}_k$ are ``far" from a linear function (Lemma \ref{l:separation from linear case}).  This quantitative control gives an improved Sobolev-type inequality, which immediately gives the desired results of Theorem \ref{t:small perturbation regime}.

\subsection{Taylor Expansion} For the rest of this subsection, we fix $A \in S^+_{2\times 2}(\mathbb{R})$ such that
\begin{align*}
    A = \begin{pmatrix}
        d_1^2 & 0\\
        0& d_2^2
    \end{pmatrix}.
\end{align*}
Further, we fix $\alpha\in (0,\frac{1}{2}]$ and $\overline{s} \in (-\infty,0]$ such that $\gamma_{A}(H_2(e_1, \overline{s})) = \alpha$.

We begin by establishing asymptotic estimates for both the perimeter deficit and the anisotropic Gaussian asymmetry.

\begin{proposition}[Taylor Expansion Estimates]\label{p:taylor expansion}
Let $A, \alpha, \overline{s}$ be as above.  Let $0<\epsilon$. Let $u\in W^{1, \infty}(\mathbb{R})$ with $\norm{u}_{W^{1,\infty}(\mathbb{R})}<\epsilon$.  Then,
\begin{align*}
\Big|e&^{-\frac{1}{2}\langle A(\overline{s}, x_2), (u(x_2),0)\rangle - \langle A(u(x_2),0), (\overline{s},x_2)\rangle  - \langle A(u(x_2),x_2), (u(x_2),0)\rangle}\\
        & - \Big(1 - (d_1^2\overline{s})u(x_2) +\left(- \frac{d_1^2}{2} + \frac{(d_1^2\overline{s})^2}{2}\right)|u(x_2)|^2\Big)\Big| \le O_1(\varepsilon^3), 
\end{align*}
where the function $O_1(\varepsilon^3)$ is given by
\begin{align*}
    O_1(\varepsilon^3):= C\frac{(d_1^2\overline{s})^2}{2}|u(x_2)|^2 (\varepsilon+ o(\varepsilon)).
\end{align*}
Furthermore, we also have 
\begin{align*}
\sqrt{1+ |\nabla u(x_2)|^2}=1 + \frac{1}{2}|\nabla u(x_2)|^2 -\frac{1}{8}|\nabla u(x_2)|^4+O_2(\varepsilon^6)
\end{align*}
The function $O_1(\epsilon^3)$ depends upon $\overline{s}$, $u(x_2)$, and $d_1$.  Moreover, it is continuous in these variables.  The function $O_2(\epsilon^6)$ is independent of $\overline{s}$ and $A$.
\end{proposition}

\begin{proof}
   \emph{Estimates on the error for the graph term.}  By Taylor expansion $\sqrt{1+x^2}=1+\frac12 x^2-\frac{1}{8}x^4+O(x^6)$ and $\|\nabla u\|_{L^\infty(\R)}\leq \varepsilon$, we have
\begin{align*}
\sqrt{1+ |\nabla u(x_2)|^2}=1 + \frac{1}{2}|\nabla u(x_2)|^2 -\frac{1}{8}|\nabla u(x_2)|^4+O_2(\varepsilon^6)
\end{align*}
for all $x_2 \in \mathbb{R}$.\\
\noindent \emph{Estimates on the weight term.}  We estimate by the power series expansion and then by the orthogonality of $A$ and $A^T=A$.
\begin{align*}
    & e^{\frac{-\langle A(\overline{s},x_2), ( u(x_2), 0)\rangle - \langle A(u(x_2), 0), (\overline{s}, x_2)\rangle  - \langle A(u(x_2),0), (u(x_2), 0)\rangle}{2}}\\
    & = 1+ \frac{-\langle A(\overline{s}, x_2), ( u(x_2), 0)\rangle - \langle A(u(x_2),0), (\overline{s}, x_2)\rangle  - \langle A(u(x_2),0), (u(x_2), 0)\rangle}{2}\\
    & \qquad + \sum_{n=2}^{\infty} \frac{\Big(\frac{-\langle A(\overline{s}, x_2), (u(x_2),0)\rangle - \langle A(u(x_2), 0), (\overline{s}, x_2)\rangle  - \langle A(u(x_2),0), (u(x_2),0)\rangle}{2}\Big)^n}{n!}\\
    & = 1+ \frac{-\langle A(\overline{s}, 0), ( u(x_2),0)\rangle - \langle A(u(x_2), 0), (\overline{s},0)\rangle  - \langle A(u(x_2),0), (u(x_2),0)\rangle}{2}\\
    & \qquad + \sum_{n=2}^{\infty} \frac{\Big(\frac{-\langle A(\overline{s},0), (u(x_2),0)\rangle - \langle A(u(x_2),0), (\overline{s}, 0)\rangle  - \langle A(u(x_2),0), (u(x_2),0)\rangle}{2}\Big)^n}{n!}
    \end{align*}
Now, we note that 
$$\frac{-\langle A(\overline{s},0), (u(x_2),0)\rangle - \langle A(u(x_2),0), (\overline{s},0)\rangle  - \langle A(u(x_2),0), (u(x_2),0)\rangle}{2} =  -d_1^2\overline{s}u(x_2) - \frac{d_1^2}{2}|u(x_2)|^2$$
and
\begin{align*}
&\frac{1}{2!}\left(\frac{-\langle A(\overline{s},0), (u(x_2),0)\rangle - \langle A(u(x_2),0), (\overline{s},0)\rangle  - \langle A(u(x_2),0), (u(x_2),0)\rangle}{2}\right)^2 \\
&= \frac{(d_1^2\overline{s})^2}{2}|u(x_2)|^2 + \frac{d_1^4\overline{s}}{2}u^3(x_2)+ \frac{d_1^4}{8}|u(x_2)|^4.
\end{align*}
Therefore, when $\|u\|_{L^\infty(\R)}\leq \varepsilon$,
\begin{align*}
&\Big|e^{\frac{-\langle A(\overline{s},x_2), ( u(x_2),0)\rangle - \langle A(u(x_2),0), (\overline{s},x_2)\rangle  - \langle A(u(x_2),0), (u(x_2),0)\rangle}{2}}\\
        & - \Big(1 - (d_1^2\overline{s})u(x_2) +\left(- \frac{d_1^2}{2} + \frac{(d_1^2\overline{s})^2}{2}\right)|u(x_2)|^2\Big)\Big| \le O_1(\varepsilon^3). 
\end{align*}  
\end{proof}

\begin{lemma}\label{l:first order estimate}
Let $A, \alpha,\overline{s}$ be as above. Let $0<\epsilon$, and $u\in W^{1, \infty}(\mathbb{R})$ with $\norm{u}_{W^{1,\infty}(\mathbb{R})}<\epsilon$.  Let $E= E_u \subset \mathbb{R}^2$ as in Definition \ref{d:standing assumptions}(9) .  Assume further that $\gamma_{A}(E_u) = \gamma_{A}(H_2(e_1, \overline{s}))$.

Then,
\begin{align*}
\Big|\int_{\mathbb{R}}u(x_2)d\gamma_{A',\overline{s}}(x_2) - \frac{d_1^2\bar{s}}{2}\int_{\mathbb{R}}|u(x_2)|^2d\gamma_{A',\overline{s}}(x_2)\Big|\le O_3(\varepsilon^3).
\end{align*}
Moreover, the error function $O_3(\epsilon^3)$ may be chosen such that 
\begin{align*}
    O_3(\epsilon^3) := C\left(1 + \overline{s}^2\right)\int_{\R}|u(x_2)|^2d\gamma_{A',\overline{s}}(x_2)\Big(\epsilon + o(\epsilon)\Big),
\end{align*}
where the constant $C$ depends continuously on the eigenvalues of $A$ and $\overline{s}$.
\end{lemma}

\begin{proof}
Since $\gamma_{A_{1,2}}(E_u)=\gamma_{A_{1,2}}(H_2(e_1, \overline{s}))$, we obtain the following:
\begin{align*}
    \frac{d_1d_2}{(2\pi)} \int_{\mathbb{R}}\int_{0}^{u(x_2)}e^{\frac{-\langle A(\overline{s}+t, x_2), (\bar{s}+t,x_2)\rangle}{2}}dtdx_2=0.
\end{align*}
From Proposition \ref{p:taylor expansion}, we estimate
\begin{align*}
&\Big|e^{\frac{-\langle A(\bar{s}+t, x_2), (\bar{s}+t, x_2)\rangle}{2}}- \Big(1 - (d_1^2\overline{s})t +(- \frac{d_1^2}{2} + \frac{(d_1^2\overline{s})^2}{2})t^2\Big)e^{\frac{-\langle A(\bar{s},x_2), (\bar{s},x_2)\rangle}{2}}\Big| \le O_1(\varepsilon^3)e^{\frac{-\langle A(\bar{s}, x_2), (\bar{s},x_2)\rangle}{2}},
\end{align*}
for any $|t|\leq |u(x_2)|\leq \varepsilon$. Hence,
\begin{align*}
    \Big| & \int_{\mathbb{R}}\int_{0}^{u(x_2)}e^{\frac{-\langle A(\bar{s}+t,x_2), (\bar{s}+t, x_2)\rangle}{2}}dtdx_2 - \int_{\mathbb{R}}\int_0^{u(x_2)}\Big(1 - (d_1^2\overline{s})t +(- \frac{d_1^2}{2} + \frac{(d_1^2\overline{s})^2}{2})t^2\Big)dtd\gamma_{A',\overline{s}}(x_2)\Big|\\
    &= \Big|\int_{\mathbb{R}}u(x_2)d\gamma_{A',\overline{s}}(x_2)-\frac{d_1^2\bar{s}}{2}\int_{\R}|u(x_2)|^2d\gamma_{A',\overline{s}}(x_2)+\left(- \frac{d_1^2}{6} + \frac{(d_1^2\overline{s})^2}{6}\right)\int_{\R}u^3(x_2)d\gamma_{A',\overline{s}}(x_2)\Big|\\
    & \le \epsilon \int_{\mathbb{R}}O_1(\varepsilon^3)d\gamma_{A',\overline{s}}(x_2).
\end{align*}
By the triangle inequality, we obtain that
\begin{align*}
   \Big|&\int_{\mathbb{R}}u(x_2)d\gamma_{A',\overline{s}}(x_2)-\frac{d_1^2\bar{s}}{2}\int_{\R}|u(x_2)|^2d\gamma_{A',\overline{s}}(x_2)\Big|\\
    & \le \Big| \left(- \frac{d_1^2}{6} + \frac{(d_1^2\overline{s})^2}{6}\right)\int_{\R}u^3(x_2)d\gamma_{A',\overline{s}}(x_2) \Big| + \epsilon \int_{\mathbb{R}}O_1(\varepsilon^3)d\gamma_{A',\overline{s}}(x_2).
\end{align*}
Recalling the definition of $O_1(\varepsilon^3)$ and the assumption that $|u|\le \epsilon$ gives the claim of the lemma.
\end{proof}

\begin{lemma}[Estimates on the Perimeter Deficit]\label{l:deficit estimates}
Let $A,\alpha,\overline{s}$ be as above.  Let $0<\epsilon$, and $u\in W^{1, \infty}(\mathbb{R})$ with $\norm{u}_{W^{1,\infty}(\mathbb{R})}<\epsilon$.  Let $E= E_u$ as in Definition \ref{d:standing assumptions}(9), and assume further that $\gamma_{A}(E_u) = \gamma_A(H_2(e_1, \overline{s}))$.

Then,
\begin{align*}
    \Big|& \delta_{\gamma_A}(E_u) - \\
    & \left( \frac{d_1d_2}{2\pi}\int_{\mathbb{R}}\frac{1}{2}|\nabla u(x_2)|^2d\gamma_{A',\overline{s}}(x_2) + \frac{d_1d_2}{2\pi}\int_{\mathbb{R}}\Big(-\frac{1}{2}d_1^2|u(x_2)|^2 \Big)d\gamma_{A',\overline{s}}(x_2)\right) \Big| \le  O_3(\varepsilon^3),
   \end{align*}
   where we have possibly increased the constant $C$ in the definition of $O_3(\varepsilon^3)$.
\end{lemma}

\begin{proof}
Note that
\begin{align*}
e^{\frac{-\langle A(\bar{s}+u(x_2), x_2), (\bar{s}+u(x_2), x_2)\rangle}{2}}=\left(e^{\frac{-\langle A(\overline{s},x_2), (u(x_2), 0)\rangle - \langle A(u(x_2),0), (\overline{s}, x_2)\rangle  - \langle A(u(x_2),0), (u(x_2),0)\rangle}{2}}\right)e^{\frac{-\langle A(\bar{s}, x_2), (\bar{s}, x_2)\rangle}{2}}
\end{align*}
and
$$\delta_{\gamma_A}(E_u)=\text{Per}_{\gamma_A}(E_u)-\text{Per}_{\gamma_A}(H_2(e_1,\bar{s}))$$
Putting the estimates from our Taylor expansion together, we obtain the following.
\begin{align}\nonumber
    \text{Per}_{\gamma_A}(E_u) &= \frac{d_1d_2}{2\pi}\int_{\mathbb{R}}\sqrt{1+|\nabla u|^2}d\gamma_{A',\overline{s}+u(x_2)}(x_2)\\ \nonumber
    & \ge \frac{d_1d_2}{2\pi}\int_{\mathbb{R}}\Big(1+\frac{1}{2}|\nabla u|^2-\frac{1}{8}|\nabla u|^4+O_2(\varepsilon^6)\Big)\\ \nonumber
    &\quad \cdot \Big(1 - (d_1^2\overline{s})u(x_2)+ \left(- \frac{d_1^2}{2} + \frac{(d_1^2\overline{s})^2}{2}\right)|u(x_2)|^2 -O_1(\varepsilon^3) \Big)d\gamma_{A',\overline{s}}(x_2)\\ \label{e:deficit estimate 1}
    & = \text{Per}_{\gamma_A}(H_2(e_1, \overline{s})) - \frac{d_1d_2}{2\pi}(d_1^2\overline{s})\int_{\mathbb{R}}u(x_2) d\gamma_{A',\overline{s}}(x_2)\\ \nonumber
    & \quad + \frac{d_1d_2}{2\pi}\int_{\mathbb{R}}\frac{1}{2}|\nabla u(x_2)|^2d\gamma_{A',\overline{s}}(x_2)\\ \nonumber
    & \quad + \frac{d_1d_2}{2\pi}\left(- \frac{d_1^2}{2} + \frac{(d_1^2\overline{s})^2}{2}\right)\int_{\mathbb{R}}|u(x_2)|^2d\gamma_{A',\overline{s}}(x_2) + O(\varepsilon^3),
\end{align}
where we may take 
\begin{align}\label{e:error estimate deficit}
    O(\varepsilon^3) = C \Big(\int \epsilon^2 u(x_2) - O_1(\epsilon^3)d\gamma_{A',\overline{s}}(x_2)\Big)
\end{align}
for some constant $C$ which depends continuously on $\overline{s}$ and the eigenvalues of $A$.

From Lemma \ref{l:first order estimate}, we have
\begin{align*}
-(d_1^2\bar{s})\int_{\mathbb{R}}u(x_2)d\gamma_{A',\overline{s}}(x_2)&\ge-\frac{(d_1^2\bar{s})^2}{2}\int_{\mathbb{R}}|u(x_2)|^2d\gamma_{A',\overline{s}}(x_2)-O_3(\varepsilon^3)
\end{align*}
Therefore, applying this bound to both \eqref{e:error estimate deficit} and \eqref{e:deficit estimate 1}, we obtain
\begin{align*}
\text{Per}_{\gamma_A}(E_u)&\ge\text{Per}_{\gamma_A}(H_2(\e_1, \overline{s}))\\
&\quad+ \frac{d_1d_2}{2\pi}\int_{\mathbb{R}}\frac{1}{2}|\nabla u(x_2)|^2d\gamma_{A',\overline{s}}(x_2)-\frac{d_1d_2}{2\pi} \frac{d_1^2}{2} \int_{\mathbb{R}}|u(x_2)|^2d\gamma_{A',\overline{s}}(x_2)- O_3(\varepsilon^3),
\end{align*}
where we have possibly increased the constant $C$ in the definition of $O_3(\varepsilon^3)$.

The opposite inequality follows in an identical manner, using $+O_1(\varepsilon^3)$ in place of $-O_1(\varepsilon^3)$ in the Taylor expansion above \eqref{e:deficit estimate 1} and using the corresponding inequality from Lemma \ref{l:first order estimate}.
\end{proof}

\begin{lemma}[Estimate on the $\gamma_A$-asymmetry]\label{l:fraenkel symmetry}
Let $A,\alpha,\overline{s}$ be as above. Let $0<\epsilon$, and $u \in W^{1,\infty}(\mathbb{R})$ with $\norm{u}_{W^{1,\infty}(\mathbb{R})}\le \epsilon$.  Let $E_u$ as in Definition \ref{d:standing assumptions}, and assume that $\gamma_{A}(E_u) = \gamma_A(H_2(e_1, \overline{s}))$.  Then,
    \begin{align*}
    \Big|\gamma_A\left(E_u \Delta H_2\left(e_1, \bar{s}\right)\right) -  \frac{d_1d_2}{2\pi}\int_{\mathbb{R}}|u(x_2)|d\gamma_{A',\overline{s}}(x_2)\Big| \le  O_4(\varepsilon^2).
\end{align*}
Moreover, we may take $O_5(\epsilon^2)$ to be of the form
\begin{align*}
    O_4(\epsilon^2) = C (1 + \frac{d_1^{2}|\overline{s}|}{2} )\frac{d_1d_2}{2\pi} \int_{\mathbb{R}}|u(x_2)|^2d\gamma_{A',\overline{s}}(x_2) + O_3(\varepsilon^3).
\end{align*}
In particular, $O_4(\epsilon^2)$ depends continuously on $\overline{s}$ and the eigenvalues of $A$.
\end{lemma}

\begin{proof}
First, we write
\begin{align*}
    \gamma_A\left(E_u \Delta H_2\left(e_1, \bar{s}\right)\right)& = \frac{d_1d_2}{2\pi}\int_{\mathbb{R}}\int_{\min\{u(x_2), 0\}}^{\max\{u(x_2),0\}} e^{-\frac{\langle A(\bar{s}+t, x_2), (\bar{s}+t, x_2)\rangle}{2}}dt d\mathcal{H}^{1}(x_2).
\end{align*}
Now, we estimate from Proposition \ref{p:taylor expansion},
\begin{align*}
    \Big|\frac{d_1d_2}{2\pi}&\int_{\mathbb{R} }\int_{\min\{u(x_2), 0\}}^{\max\{u(x_2),0\}} e^{-\frac{\langle A(\bar{s}+t, x_2), (\bar{s}+t,x_2)\rangle}{2}}dt d\mathcal{H}^{1}(x_2)\\
    & - \frac{d_1d_2}{2\pi}\int_{\mathbb{R} }\int_{\min\{u(x_2), 0\}}^{\max\{u(x_2),0\}} 1 - (d_1^2\overline{s})t - \frac{1}{2}d_1^2t^2 + \frac{(d_1^2\overline{s})^2}{2}t^2dt d\gamma_{A',\bar{s}}(x_2) \Big|\\
    & \le  \frac{d_1d_2}{2\pi}\int_{\mathbb{R} }\int_{\min\{u(x_2), 0\}}^{\max\{u(x_2),0\}} O_1(\varepsilon^3) dt d\gamma_{A',\bar{s}}(x_2)\\
    & = O_3(\varepsilon^3).
\end{align*}
In particular, integrating in $t$ and using the triangle inequality,
\begin{align*}
    \Big|& \gamma_A\left(E_u \Delta H_2\left(e_1, \bar{s}\right)\right) - \frac{d_1d_2}{2\pi}\int_{\mathbb{R}} |u(x_2)| d\gamma_{A',\bar{s}}(x_2) \Big| \\
    & \le  \frac{d_1d_2}{2\pi}\int_{\mathbb{R}}(\frac{d_1^2\overline{s}}{2})|u(x_2)|^2 - \frac{1}{6}d_1^2|u(x_2)|^3 + \frac{(d_1^2\overline{s})^2}{6}|u(x_2)|^3d\gamma_{A',\bar{s}}(x_2) + O_3(\varepsilon^3)\\
    & \le O_4(\varepsilon^2),
\end{align*}
where we have used $|u|\le \varepsilon$.
\end{proof}

\begin{theorem}[Lower bound in the small perturbation regime]\label{l:fuglede-type estimate}
Let $A,\alpha,\overline{s}$ be as above.  There exists a function $\epsilon: (0,1/2)\times \mathbb{R} \rightarrow (0,1]$ with the following properties.
    \begin{enumerate}
        \item $\epsilon(\delta, \overline{s})$ is continuous in $\overline{s} \in \mathbb{R}$.
        \item For any $0<\delta$ and any set $E_u$ as in \eqref{d:E_u def} with 
        \begin{enumerate}
            \item $\|u\|_{W^{1,\infty}(\R)}<\varepsilon_0(\delta, \overline{s})$,
            \item $\gamma_A\left(E_u\right)=\gamma_A(H_2(e_1,\bar{s}))$,
            \item $\lambda_{\gamma_A}\left(E_u\right)^2 \neq 0$,
        \end{enumerate}
        the following inequality holds
    \begin{align*}
    \frac{\delta_{\gamma_A}(E_u)}{\lambda^2_{\gamma_A}(E_u)} &\ge  \frac{2\pi e^{\frac{d_1^2\overline{s}^2}{2}}}{d_1}\left[\frac{\int_{\mathbb{R}}\frac{1}{2}|\nabla u(x_2)|^2 d\gamma_{A',\overline{s}}(x_2)}{\Big(1+O_5(\varepsilon)\Big)\int_{\mathbb{R}}|u(x_2)|^2d\gamma_{A',\overline{s}}(x_2)} -\frac{d_1^2}{2\Big(1+O_5(\varepsilon)\Big)} -O_6(\varepsilon)\right],
\end{align*}
where
\begin{align*}
    |O_6(\epsilon)| \le C(1+|\overline{s}|^2)(\epsilon +o(\epsilon)).
\end{align*}
\end{enumerate}
\end{theorem}

\begin{proof}
Note that trivially, $\lambda_{\gamma_{A}}(E_u) \le \gamma_A\left(E_u \Delta H_2\left(e_1, \bar{s}\right)\right)$.  Thus, we observe that from Lemma \ref{l:deficit estimates}, Lemma \ref{l:fraenkel symmetry}, and H\"older's inequality that 
\begin{align*}
    \frac{\delta_{\gamma_A}(E_u)}{\lambda^2_{\gamma_A}(E_u)} & \ge \frac{\frac{d_1d_2}{(2\pi)^{1/2}}\int_{\mathbb{R}}\frac{1}{2}|\nabla u(x_2)|^2 -\frac{1}{2}d_1^2|u(x_2)|^2 d\gamma_{A',\overline{s}}(x_2) - O_3(\varepsilon^3)}{\Big(\frac{d_1d_2}{2\pi}\int_{\mathbb{R}}|u(x_2)|d\gamma_{A',\overline{s}}(x_2)+O_4(\varepsilon^2)\Big)^2}\\ 
    &\ge \frac{\frac{d_1d_2}{(2\pi)^{1/2}}\int_{\mathbb{R}}\frac{1}{2}|\nabla u(x_2)|^2 -\frac{1}{2}d_1^2|u(x_2)|^2  d\gamma_{A',\overline{s}}(x_2) - O_3(\varepsilon^3)}{\frac{d_1^2d_2^2}{(2\pi)^{2}}\gamma_{A',\overline{s}}(\mathbb{R})\left(\int_{\mathbb{R}}|u(x_2)|^2d\gamma_{A',\overline{s}}(x_2)+ O(\varepsilon^3)\right)},
\end{align*}
where $O(\epsilon^3)$ in the denominator is of the form
    \begin{align*}
    O(\epsilon^3)= (\gamma_{A',\overline{s}}(\mathbb{R}))^{-1} \cdot \Big(\int_{\mathbb{R}}|u(x_2)|d\gamma_{A',\overline{s}}(x_2) \cdot O_4(\varepsilon^2) + O_4(\epsilon^2)^2\Big).    
    \end{align*}
Now, we rewrite the denominator, claiming that
\begin{align}\label{e:problem for continuity}
    \int_{\mathbb{R}}|u(x_2)|^2d\gamma_{A',\overline{s}}(x_2)+O(\varepsilon^3) \le \Big(1+O_5(\varepsilon)\Big)\int_{\mathbb{R}}|u(x_2)|^2d\gamma_{A',\overline{s}}(x_2),
\end{align}
where $O_5(\epsilon)= O(\epsilon^3) \cdot (\int_{\mathbb{R}}|u(x_2)|^2d\gamma_{A',\overline{s}}(x_2))^{-1}$.  Recalling the definition of $O_4(\epsilon^2)$ in Lemma \ref{l:fraenkel symmetry} and $O_3(\epsilon^3)$ in Lemma \ref{l:first order estimate}, gives that there exists a $C<\infty$ such that for all sufficiently small $\epsilon$ and all nearby $\overline{s}$ we may take $O_5(\epsilon) = C\epsilon \gamma_{A',\overline{s}}(\mathbb{R})$.  This function is obviously continuous in $\overline{s}$, the eigenvalues of $A$, and, by possibly choosing a larger $C$, we may ensure that the bounds from above hold with this new $O_5(\epsilon)$.

Thus, we continue with our estimates obtaining
\begin{align*}
    \frac{\delta_{\gamma_A}(E_u)}{\lambda^2_{\gamma_A}(E_u)} &\ge \frac{\frac{d_1d_2}{(2\pi)^{\frac{1}{2}}}}{\frac{d_1^2d_2^2}{(2\pi)^{2}}\gamma_{A',\overline{s}}(\mathbb{R})}\left[\frac{\int_{\mathbb{R}}\frac{1}{2}|\nabla u(x_2)|^2 -\frac{1}{2}d_1^2|u(x_2)|^2  d\gamma_{A',\overline{s}}(x_2) - O_3(\varepsilon^3)}{\Big(1+O_5(\varepsilon)\Big)\int_{\mathbb{R}}|u(x_2)|^2d\gamma_{A',\overline{s}}(x_2)}\right]\\ 
    &\ge \frac{\frac{d_1d_2}{(2\pi)^{\frac{1}{2}}}}{\frac{d_1^2d_2^2}{(2\pi)^{2}}\gamma_{A',\overline{s}}(\mathbb{R})} \left[\frac{\int_{\mathbb{R}}\frac{1}{2}|\nabla u(x_2)|^2 d\gamma_{A',\overline{s}}(x_2)}{\Big(1+O_5(\varepsilon)\Big)\int_{\mathbb{R}}|u(x_2)|^2d\gamma_{A',\overline{s}}(x_2)} -\frac{d_1^2}{2\Big(1+O_5(\varepsilon)\Big)} -O_6(\varepsilon)\right],
\end{align*}
where $O_6(\epsilon) = C\frac{O_3(\epsilon^3)}{\int_{\mathbb{R}}|u(x_2)|^2d\gamma_{A',\overline{s}}(x_2)}$. Recalling the definition of $O_3(\epsilon^3)$ from Lemma \ref{l:first order estimate}, this gives
\begin{align*}
    O_6(\epsilon) = C(1+|\overline{s}|^2)\frac{\int_{\mathbb{R}}|u(x_2)|^2d\gamma_{A',\overline{s}}(x_2)(\epsilon + o(\epsilon))}{\int_{\mathbb{R}}|u(x_2)|^2d\gamma_{A',\overline{s}}(x_2)} \le C(1+|\overline{s}|^2)(\epsilon +o(\epsilon)),
\end{align*}
where $C$ is a constant that depends continuously on $A$. In particular, $O_6(\epsilon)$ is continuous in $\overline{s}$. Finally, we note that
\begin{align*}
    \frac{\frac{d_1d_2}{(2\pi)^{\frac{1}{2}}}}{\frac{d_1^2d_2^2}{(2\pi)^{2}}\gamma_{A',\overline{s}}(\mathbb{R})} &= \frac{\frac{1}{(2\pi)^{\frac{1}{2}}}}{\frac{d_1d_2}{(2\pi)^{2}}\gamma_{A',\overline{s}}(\mathbb{R})}= \frac{\frac{1}{(2\pi)^{\frac{1}{2}}}}{\frac{d_1}{(2\pi)^{\frac{2+1}{2}}}e^{-\frac{d_1^2\overline{s}^2}{2}}} =\frac{(2\pi)e^{\frac{d_1^2\overline{s}^2}{2}}}{d_1}.
\end{align*}
This gives the claim.
\end{proof}

Now, we justify Remark \ref{r:Q is finite and consequences}.
\begin{lemma}\label{l:proof of remark Q finite}
    Let $n \ge 2$, $A\in S^+_{n\times n}(\mathbb{R})$ be as in Definition \ref{d:standing assumptions}(3), and $\overline{s} \in(-\infty, 0]$.  Then, $Q(H_n(e_1,\overline{s}))<\infty$.
\end{lemma}

\begin{proof}
    By taking product sets we see $Q(H_{n+1}(e_1,\overline{s}))\le Q(H_{n}(e_1,\overline{s}))$ for all $n \ge 2$.  Therefore, we reduce to proving the $n=2$ case.  First, we let $A$ be anisotropic.  Let $f_k\in W^{1,\infty}(\mathbb{R}) \cap L^{\infty}(\mathbb{R})$ by given by
    \begin{align*}
        f_k(x):= 
            \frac{1}{k}\max\{(1-|x|), 0\} + c_k,
    \end{align*}
    where $c_k$ is chosen such that $\gamma_{A}(E_{f_k}) = \gamma_{A}(H_2(e_1,\overline{s}))$. Clearly, then $Q(H_2(e_1,\overline{s})) \le \lim_{k \rightarrow \infty}\frac{\delta_{\gamma_A}(E_{f_k})}{\lambda_{\gamma_A}(E_{f_k})^2}.$  Recalling that for anisotropic $A$ the only admissible half-planes are $H_2(e_1,\overline{s}), H_2(-e_1,|\overline{s}|)$ by \eqref{unique minimizer}, and recalling the definition of $E_{f_k}$ in Definition \ref{d:standing assumptions}(9), we see that for all sufficiently large $k \in \mathbb{R}^n$, $\lambda_{\gamma_A}(E_{f_k}) = \gamma_A(E_{f_k} \Delta H_2(e_1,\overline{s}))$.  Thus, we have that $Q(H_2(e_1,\overline{s})) \le \lim_{k \rightarrow \infty}\frac{\delta_{\gamma_A}(E_{f_k})}{\gamma_A(E_{f_k} \Delta H_2(e_1,\overline{s}))^2}$.

    Now, by direct calculations, we estimate
    \begin{align*}
        \norm{f_k \circ \frac{1}{d_2}}_{L^1(\mathbb{R}, \gamma_1)} \approx_{d_1, d_2, \overline{s}} \left(\frac{1}{k}e^{-\frac{1}{2}|\overline{s}|^2}\right),\qquad \norm{f_k \circ \frac{1}{d_2}}_{L^2(\mathbb{R}, \gamma_1)} \approx_{d_1, d_2, \overline{s}} \left(\frac{1}{k^2}e^{-|\overline{s}|^2}\right)^{\frac{1}{2}}.
    \end{align*}
    Hence, $\norm{f_k \circ \frac{1}{d_2}}_{L^1(\mathbb{R}, \gamma_1)} \simeq_{d_1, d_2, \overline{s}} \norm{f_k\circ \frac{1}{d_2}}_{L^2(\mathbb{R}, \gamma_1)}$ by constants independent of $k \in \mathbb{N}$.  

    Therefore, following through the argument of Lemma \ref{l:fuglede-type estimate}, we have
\begin{align*}
    \frac{\delta_{\gamma_A}(E_{f_k})}{\lambda_{\gamma_A}(E_{f_k})^2} &\approx_{d_1, d_2, \overline{s}}  \frac{2\pi e^{\frac{d_1^2\overline{s}^2}{2}}}{d_1}\left[\frac{\int_{\mathbb{R}}\frac{1}{2}|\nabla f_k(x_2)|^2 d\gamma_{A',\overline{s}}(x_2)}{\Big(1+O_5(\varepsilon)\Big)\int_{\mathbb{R}}|f_k(x_2)|^2d\gamma_{A',\overline{s}}(x_2)} -\frac{1}{2\Big(1+O_5(\varepsilon)\Big)}d_1^2 -O_6(\varepsilon)\right].
\end{align*}
Computing that 
\begin{align*}
    \norm{\nabla f_k}_{L^2(\mathbb{R}),\gamma_{A',\overline{s}}}^2 \approx_{d_1, d_2, \overline{s}} (\frac{1}{k})^2e^{|\overline{s}|^2}
\end{align*}
gives $\frac{\delta_{\gamma_A}(E_{f_k})}{\lambda_{\gamma_A}(E_{f_k})}^2 \lesssim_{d_1, d_2, \overline{s}} 1$.

The argument for isotropic $A$ follows analogously, if we can show that $$\lambda_{\gamma_A}(E_{f_k}) \approx_{d_1, \overline{s}} \gamma_{A}(E_{f_k}\Delta H_2(e_1, \overline{s})).$$  To see that these are comparable, first let $1<C_1<\infty$ be such that for all sufficiently large $k \in \mathbb{N}$,
\begin{align*}
    \gamma_{A}(E_{f_k}\Delta H_2(e_1, \overline{s})) \in (\frac{1}{C_1k}e^{-\frac{1}{2}d_1^2|\overline{s}|^2}, \frac{C_1}{k}e^{-\frac{1}{2}d_1^2|\overline{s}|^2}).
\end{align*}
Next, we let $1<C_2<\infty$ such that for all sufficiently large $k \in \mathbb{N}$, we may estimate
\begin{align*}
    \gamma_{A}(\{(x_1, x_2) \in E_{f_k}: x_1 \in (\overline{s} + c_k + \frac{1}{2k}, \overline{s} + c_k + \frac{1}{k}\}) \in (\frac{1}{C_2k}e^{-\frac{1}{2}d_1^2|\overline{s}|^2}, \frac{C_2}{k}e^{-\frac{1}{2}d_1^2|\overline{s}|^2}).
\end{align*}
Finally, we fix $0<\epsilon_3<<1$ and let $1<C_3<\infty$ be such that for any $0<\epsilon \le \epsilon_3$ and all sufficiently large $k \in \mathbb{N}$,
\begin{align*}
    \gamma_A(\{(x_1, x_2) \not\in E_{f_k}: x_1 \in (\overline{s} + c_k, \overline{s} + c_k + \frac{\epsilon}{k}\}) \in ( \frac{\epsilon}{C_3k}e^{-\frac{1}{2}d_1^2|\overline{s}|^2} e^{-\frac{1}{2}d_1^2}, \frac{\epsilon C_3}{k}e^{-\frac{1}{2}d_1^2|\overline{s}|^2} e^{-\frac{1}{2}d_1^2}).
\end{align*}
Since by definition as an infimum, $\lambda_{\gamma_A}(E_{f_k})\le \gamma_{A}(E_{f_k}\Delta H_2(e_1, \overline{s}))$, it only remains to show that $\lambda_{\gamma_A}(E_{f_k}) \ge \frac{1}{200 C_2C_3}\frac{1}{k}e^{-\frac{1}{2}d_1^2|\overline{s}|^2} e^{-\frac{1}{2}d_1^2}$.  

Suppose for the sake of contradiction that $\lambda_{\gamma_A}(E_{f_k}) < \frac{1}{200 C_2C_3}\frac{1}{k}e^{-\frac{1}{2}d_1^2|\overline{s}|^2} e^{-\frac{1}{2}d_1^2}$.  Let $H_2(\omega, \overline{s})$ be the minimizing half-plane.  We claim that there exist two points $$(x_1', x_2'), (y_1', y_2') \in (B_{\frac{1}{100C_3k}}(\{x_1 = \overline{s} + c_k\}) \setminus E_{f_k}) \cap H_2(\omega,\overline{s})^c$$ such that $x_2<0$ and $x_2'>0$.  Indeed, if no such points exists, then $$B_{\frac{1}{100C_3k}} (\{x_1 = \overline{s} + c_k\}) \setminus (E_{f_k} \cup \{x_1 >0\}) \subset H_2(\omega, \overline{s}),$$ and similarly for $\{x_1<0\}$, and hence
\begin{align*}
    \lambda_{\gamma_A}(E_{f_k}) & = \gamma_A(E_{f_k} \Delta H_2(\omega, \overline{s})) \ge \gamma_{A}(B_{\frac{1}{100C_2C_3k}} (\{x_1 = \overline{s} + c_k\}) \setminus (E_{f_k} \cup \{x_1>0\})\\
    & \ge \frac{1}{100C_2C_3k}e^{-\frac{1}{2}d_1^2|\overline{s}|^2}e^{-\frac{1}{2}d_1^2},
\end{align*}
a contradiction. Therefore, the two points $(x_1', x_2'), (y_1', y_2')$ exist, as claimed.

Now, we simply observe that $H_2(\omega, \overline{s})^c$ is convex, and therefore the line connecting $(x_1', x_2')$ and $(y_1', y_2')$ must be contained in $H_2(\omega, \overline{s})^c$.  Moreover, the whole infinite half-strip $\{(x_1, x_2) : x_1 \ge \max\{x_1', y_1'\} \text{ and } x_2 \in x_2', y_2')\} \subset H_2(\omega, \overline{s})^c$, since any half-plane $H_2(\omega', \overline{s})$ which intersects this strip must fail to intersect either
\begin{align*}
P^-_k:= E_{f_k} \cap \{x_1 >\overline{s}+c_k\} \cap \{x_2<0\}
\end{align*}
or 
\begin{align*}
P^+_k := E_{f_k} \cap \{x_1 >\overline{s}+c_k\} \cap \{x_2>0\}. 
\end{align*}

Since these sets satisfy $\gamma_A(P_k^{\pm}) \ge \frac{1}{2C_2}\frac{1}{k} e^{-\frac{1}{2}d_1^2|\overline{s}|^2}$, by our assumption on $\lambda_{\gamma_A}(E_{f_k})$, the half-infinite strip must be in $H_{2}(\omega', \overline{s})$.

In particular the set $E_{f_k} \cap \{(x_1, x_2) : x_1 \ge \max\{x_1', y_1'\} \text{ and } x_2 \in (x_2', y_2')\} \subset H_2(\omega, \overline{s})^c$ and hence $\lambda_{\gamma_2}(E_{f_k}) \ge \frac{1}{C_1k}e^{-\frac{1}{2}d_1^2|\overline{s}|^2}$, a contradiction.  

Thus, $\lambda_{\gamma_A}(E_{f_k}) \approx_{d_1, \overline{s}} \gamma_A(E_{f_k} \Delta H_{2}(e_1, \overline{s}))$. Following the argument in the anisotropic case, now gives the claim of the lemma.
\end{proof}

\subsection{A Study of the Asymptotically Linear Case}

We shall use the following notation for the rest of this section. Recall that the Hermite polynomials form an orthonormal basis of
$L^2(\mathbb{R},\gamma_{1})$ (see, for example, \cite[Definition 1.3.1 and Corollary 1.3.3]{MR1642391}). As a normalizing assumption, by change of variables we may assume that $d_1=1$. Therefore, for any $w \in L^2(\mathbb{R},\gamma_{1})$ we may write
\[
w(y)
=
\sum_{k=0}^{\infty}
a_{i} H_{i}(y),
\]
where $H_{i}(y)$ is the standard Hermite polynomial of degree $i$. By orthonormality, we have
$$a_{i}=\int_{\mathbb{R}}
w(y)H_{i}(y)\,d\gamma_{1}(y).$$

We shall apply this decomposition to the functions $w_k:= \tilde{u}_k \circ \frac{1}{d_2}$, for $\tilde{u}_k$ as given by Theorem \ref{t:technical theorem}. For these $\{w_k\}_k$, we will study, $a_{1, k}$, the $1$st order coefficient for $w_k$.

\begin{lemma}[Convergence]\label{l:convergence to linear case}
Let $A\in S^+_{n\times n}(\mathbb{R})$, $A_{1,2}$ as in Definition \ref{d:standing assumptions}(4), $\alpha\in (0,1)$ and $\overline{s} \in \mathbb{R}$ be related as in Definition \ref{d:standing assumptions}.  For $\tilde{u}_k\in L^\infty(\mathbb{R}) \cap W^{1,\infty}(\mathbb{R})$ the functions given by Theorem \ref{t:technical theorem} such that
    \begin{align*}
        \partial \tilde{E}_k = \emph{graph}_{\partial H_2(e_1, \overline{s})}\tilde{u}_k.
    \end{align*}
    Define $w_k:= \tilde{u}_k \circ \frac{1}{d_2}$ and $v_k:= \frac{w_k}{\norm{w_k}_{L^2(\mathbb{R}, \gamma_1)}} \in L^2(\mathbb{R}, \gamma_1)$.
    
    If there exists a subsequence such that $|a_{1, k'}| \rightarrow 1$ as $k' \rightarrow \infty$, then 
    \begin{enumerate}
        \item $v_{k'}$ converges strongly to the function $H_1(x)=x$ in $L^2(\mathbb{R},\gamma_1)$.
        \item For all $i \in \mathbb{N}$, $D^iv_{k'} \rightarrow D^iH_1$ strongly in $L^2([-100d_2, -100^{-1}d_2]\cup[100^{-1}d_2,100d_2], \gamma_1)$.
    \end{enumerate}
\end{lemma}

\begin{proof}
To see the claim, since $v_k$ is bounded in $L^2(\R, \gamma_1)$, there exists a subsequence $v_k$ and a function $v \in L^1(\mathbb{R},\gamma_1)$ such that $v_k\to v$ weakly in $L^2(\R,\gamma_1)$. Note also that, after possibly taking a further subsequence, we may assume that
$$\left(\int_{\R}v_k(x)xd\gamma_1(x)\right)^2> \left(1-\frac1k\right).$$
Moreover, $v_k$ and $x$ are both in $L^2(\R, \gamma_1)$. Taking $k\to \infty$, we have
\begin{equation}\label{xv_k conv}
\lim_{k\to \infty}\int_{\R}v_k(x)xd\gamma_1(x)=\int_{\R}v(x)xd\gamma_1(x)
\end{equation}
and
$$1\leq \left(\int_{\R}v(x)xd\gamma_1(x)\right)^2\leq \left(\int_{\R}v^2(x)d\gamma_1(x)\right)\left(\int_{\R}x^2d\gamma_1(x)\right)=\int_{\R}v(x)^2d\gamma_1(x).$$
Thus, by lower-semicontinuity of the norm under weak convergence, $\norm{v}_{L^2(\mathbb{R},\gamma_1)}=1$ and equality holds in the Cauchy-Schwarz inequality. Hence,
$$v(x)=cx\mbox{\quad a.e.}$$
for some constant $c$. Then
$$c^2=\int_{\R}c^2x^2d\gamma_1(x)=\int_{\R}v(x)^2d\gamma_1(x)=1\implies c=\pm1.$$
Since, by construction, $\gamma_{A_{1,2}}(\tilde{E}_k \cap \Omega_+)> \gamma_{A_{1,2}}(\tilde{E}_k \cap \Omega_-)$, it must be that $c=1$. In particular, by \eqref{xv_k conv}, we have $\int_{\R}v_k v\, d\gamma_1\to \|v\|_{L^2(\R,\gamma_1)}^2=1$. Therefore,
$$\|v_k- v\|^2_{L^2(\R,\gamma_1)}\to 0,$$
i.e., $v_k\to v$ in $L^2(\R,\gamma_1)$ and hence $v_k\to v$ a.e. in $\R$. 

Strong convergence to $v= H_1$ immediately implies the following weak convergence.  Since $\tilde{u}_k\in W^{1,\infty}(\R)$, $\tilde{u}_k$ is Lipschitz on $\R$ and hence $v_k$ is also Lipschitz on $\R$.
By strong convergence in $L^2(\R,\gamma_1)$, we have for any $\phi \in C^1_c(\mathbb{R})$
$$\int_{\R}Dv_k\varphi d\gamma_1=-\int_{\R}v_kD\left(\varphi \frac{1}{\sqrt{2\pi}}e^{-x^2/2}\right) dx\to -\int_{\R}vD\left(\varphi \frac{1}{\sqrt{2\pi}}e^{-x^2/2}\right) dx=\int_{\R}Dv\varphi d\gamma_1.$$
Now, consider the set 
\begin{align*}
    K:= [-100d_2,-100^{-1}d_2]\cup[100^{-1}d_2, 100d_2],  
\end{align*}
and recall by our construction of $\tilde{u}_k = T_{h_k}(u_k)$ and the definition of $h_k, k_j$, that $\tilde{u}_k = u_k$ on $[-100, 100]$. Therefore, by smoothness of $u_k$ (Lemma \ref{l:properties of Ik}(3)), $D^iv_k$ exists on $A$ for all $i \in \mathbb{N}$. Thus, for any $\phi \in C^1_c(K)$, the argument above works with $D^iv_k, D^iv$ in place of $v_k, v$. Thus, $D^iv_k \rightharpoonup 0$ in $L^2(K,\gamma_1) \simeq L^2(K, \mathcal{L}^1)$.  Therefore, since by a corollary to the Banach-Steinhaus theorem any weakly convergence subsequence is norm-bounded, 
$$\sup_{k'}\norm{D^iv_{k'}}_{L^{2}(K)}<\infty.$$  
Thus, since by the continuity of each $D^v_k$ on $K$, for all sufficiently large $k \in \mathbb{N}$ there must exists some point $x_k^i$ in each component of $K$ such that $|D^iv_k(x_k^i)| \le 1$.  Thus, by the fundamental theorem of calculus, we may easily obtain that 
\begin{align*}
\norm{D^iv_k}_{L^\infty(K)} \le 1 + 10\sup_{k'}\norm{D^iv_{k'}}_{L^{2}(K)}<\infty.
\end{align*}
In particular, for each $i \in \{1, 2, 3, ...\}$ we see that $\{D^{i-1}v_{k'}\}_{k'}$ are uniformly Lipschitz on $K$.  Therefore, up to choosing a subsequence,  $D^iv_k \rightarrow D^iv$ in $L^\infty(K)$ for all $i \ge 1$.  This gives the claim.
\end{proof}

\begin{lemma}[Continuity]\label{l:close to linear and comparability of lambdas}
    Let $A, \alpha, \overline{s}$ be as in Lemma \ref{l:convergence to linear case}, above.  Recall that $d\gamma_{A_{1,2}', 0} = e^{-\frac{1}{2}d_2^2x_2^2}d\mathcal{H}^1$, as in Definition \ref{d:standing assumptions}(5). 
    
    Let $\tilde{u}_k, w_k, v_k$ be as in Lemma \ref{l:convergence to linear case}. There exists a constant $0<\delta_0<1$ with the following property.
    \begin{enumerate}
        \item If $a_{1, k}>1-\delta_0$, then
\begin{align*}
    \tilde{u}_k(x_2) & = \norm{\tilde{u}_k}_{L^2(\mathbb{R},\gamma_{A_{1,2}', 0})} x_2 + f_k(x_2) \qquad \text{on }\mathbb{R}\\
    \tilde{u}_k'(x_2) & = \norm{\tilde{u}_k}_{L^2(\mathbb{R},\gamma_{A_{1,2}', 0})} + f'_k(x_2) \qquad \quad \text{on } [-100, -100^{-1}]\cup [100^{-1}, 100],\\
    \tilde{u}''_k & = f''_k(x_2) \qquad \quad \text{on } [-100, -100^{-1}]\cup [100^{-1}, 100],
\end{align*}
where \begin{enumerate}
    \item[a.] $\norm{f_k}_{L^2(\mathbb{R},\gamma_{A_{1,2}', 0})}, \norm{Df_k}_{L^2(\mathbb{R},\gamma_{A_{1,2}', 0})},  \norm{D^2f_k}_{L^2(\mathbb{R},\gamma_{A_{1,2}', 0})} \le \frac{1}{10,000}\norm{\tilde{u}_k}_{L^2(\mathbb{R},\gamma_{A_{1,2}', 0})}$.
    \item[b.] $\norm{D^2f}_{L^\infty([-100, -100^{-1}]\cup[100^{-1}, 100])} \le \frac{1}{10,000}\norm{\tilde{u}_k}_{L^2(\mathbb{R},\gamma_{A_{1,2}', 0})}$.
\end{enumerate} 
\item $\frac{1}{5}\lambda_{\gamma_{A}}(E_k) \le \gamma_{A_{1,2}}(\tilde{E}_k \Delta H_2(e_1, \overline{s})) \le 2\lambda_{\gamma_{A}}(E_k)$.
    \end{enumerate}
\end{lemma}

\begin{proof}
    The claim (1a) is immediate from a contradiction compactness argument and the previous lemma.  
    
    We now argue the second claim. Note that the lower bound follows immediately from Lemma \ref{l:properties of Ik}(3) and Lemma \ref{l:mass error is negligible}. To see the upper bound, first recall from Lemma \ref{l:properties of Ik} that
\begin{align*}
\gamma_{A_{1,2}}\left((\tilde{E}_k\Delta H_2(e_1, \overline{s}) \cap \Omega_+\right) \ge (\frac{1}{8}-\frac{c}{2})\lambda_{\gamma_A}(E_k), 
\end{align*}
where we can take $0<c$ to be as small as we wish and we use this $0<c$ to absorb the truncation error from Lemma \ref{l:mass error is negligible} for all sufficiently large $k \in \mathbb{N}$. 

Next, we recall from Corollary \ref{c:second symmetrization}(2) and Lemma \ref{l:properties of Ik} that 
\begin{align*}
    \gamma_{A_{1,2}}(I_k \cap \Omega_+) = \gamma_{A_{1,2}}(H_k \cap \Omega_+) = \frac{1}{2}\alpha + C\lambda_{\gamma_{A}}(E_k)
\end{align*}
for some $C_k \in (\frac{1}{4}-c,\frac{1}{2})$, where again we may take $0<c$ as small as we wish for all sufficiently large $k \in \mathbb{N}$. Thus, again using the $0<c$ to absorb the truncation error from Lemma \ref{l:mass error is negligible}, we can write 
\begin{align*}
\gamma_{A_{1,2}}\left((\tilde{E}_k\Delta H_2(e_1, \overline{s})) \cap \Omega_+\right) = C_k \lambda_{\gamma_{A}}(E_k) + 2 \gamma_{A_{1,2}}\left((H_{2}(e_1, \overline{s}) \setminus \tilde{E}_k) \cap \Omega_+\right).
\end{align*}

In particular, for $0<\delta_0<1$ as in (1ab) above, we have that $\tilde{u}_k>0$ on the interval $[100^{-1}, 100]$.  Thus, for all sufficiently large $k \in \mathbb{N}$, we use the Taylor expansion from Lemma \ref{l:fraenkel symmetry} to estimate that
\begin{align*}
&\gamma_{A_{1,2}}\left((H_{2}(e_1, \overline{s}) \setminus \tilde{E}_k) \cap \Omega_+\right) \\
& \le e^{-\frac{1}{2}d_1^2|\overline{s}|^2}\left(\frac{d_1d_2}{2\pi}\int_{0}^{\frac{1}{100}}|\tilde{u}_k|\,d\gamma_{A_{1,2}', 0} + \frac{d_1^2}{2\pi}\int_{100}^{\infty}|\tilde{u}_k|\,d\gamma_{A_{1,2}', 0}\right) + O_4(\epsilon^2)\\
& \le e^{-\frac{1}{2}d_1^2|\overline{s}|^2}\norm{\tilde{u}_k}_{L^2(\mathbb{R},\gamma_{A_{1,2}', 0})}\left(\frac{d_1d_2}{2\pi}\int_{0}^{\frac{1}{100}}x_2 + \frac{f_k(x_2)}{\norm{\tilde{u}_k}_{L^2(\mathbb{R},\gamma_{A_{1,2}', 0})}}\,d\gamma_{A_{1,2}', 0}(x_2)\right)\\
& \quad + e^{-\frac{1}{2}d_1^2|\overline{s}|^2}\norm{\tilde{u}_k}_{L^2(\mathbb{R},\gamma_{A_{1,2}', 0})}\left(\frac{d_1d_2}{2\pi}\int_{100}^{\infty}x_2 + \frac{f_k(x_2)}{\norm{\tilde{u}_k}_{L^2(\mathbb{R},\gamma_{A_{1,2}', 0})}}\,d\gamma_{A_{1,2}', 0}\right)\\
& \quad + O_4(\norm{\tilde{u}_k}_{L^2(\mathbb{R},\gamma_{A_{1,2}', 0})}^2)\\
& \le e^{-\frac{1}{2}d_1^2|\overline{s}|^2}\norm{\tilde{u}_k}_{L^2(\mathbb{R},\gamma_{A_{1,2}', 0})}\Big(\frac{2}{100}\gamma_{A_{1,2}', 0}([0,\frac{1}{100}]) + \frac{1}{100}\gamma_{A_{1,2}', 0}([100,\infty))\\ 
&\qquad \qquad + \frac{1}{2\pi}e^{-\frac{1}{2}d_2^2|100|^2}\Big) + O_4(\norm{\tilde{u}_k}_{L^2(\mathbb{R},\gamma_{A_{1,2}', 0})}^2)\\
& \le e^{-\frac{1}{2}|\overline{s}|^2}\frac{3}{100}\norm{\tilde{u}_k}_{L^2(\mathbb{R},\gamma_{A_{1,2}', 0})}
\end{align*}
for all sufficiently large $k \in \mathbb{N}$, where in the last line we have used our normalizing assumption that $d_2 \ge 1$ to obtain $e^{-\frac{1}{2}d_2^2|100|^2}<<\frac{1}{100}$.

Now, we note that by integration by parts
\begin{align*}
\norm{x_2}_{L^1(\mathbb{R},\gamma_1)} & = 2\frac{1}{\sqrt{2\pi}}\int_0^\infty x_2e^{-\frac{1}{2}|x_2|^2}\,dx_2= \sqrt{\frac{2}{\pi}}\\
\norm{x_2}_{L^2(\mathbb{R},\gamma_1)}^2 & = 2\frac{1}{\sqrt{2\pi}}\int_0^\infty x_2(x_2e^{-\frac{1}{2}|x_2|^2})\,dx_2 = 1.
\end{align*}
In particular, by the triangle inequality,
\begin{align}\label{e:L1 and L2 comparable for large a1}
\frac{1}{3}\norm{\tilde{u}_k}_{L^2(\mathbb{R},\gamma_{A_{1,2}', 0})} \le \norm{\tilde{u_k}}_{L^1(\mathbb{R},\gamma_{A_{1,2}', 0})} \le \norm{\tilde{u}_k}_{L^2(\mathbb{R},\gamma_{A_{1,2}', 0})}.  
\end{align}
Therefore, by the Taylor expansion in Lemma \ref{l:fraenkel symmetry}, our previous calculation shows that for all sufficiently large $k \in \mathbb{N}$
\begin{align*}
\gamma_{A_{1,2}}\left((H_{2}(e_1, \overline{s}) \setminus \tilde{E}_k) \cap \Omega_+\right) & \le e^{-\frac{1}{2}|\overline{s}|^2}\frac{3}{100}\norm{\tilde{u}_k}_{L^2(\mathbb{R},\gamma_{A_{1,2}', 0})}\\
& \le e^{-\frac{1}{2}d_1^2|\overline{s}|^2}\frac{9}{100}\norm{\tilde{u}_k}_{L^1(\mathbb{R},\gamma_{A_{1,2}', 0})}\qquad \qquad \eqref{e:L1 and L2 comparable for large a1}\\
& \le \frac{1}{10} \gamma_{A_{1,2}}((\tilde{E}_k \Delta H_2(e_1, \overline{s})) \cap \Omega_+) \qquad (\text{Lemma \ref{l:fraenkel symmetry}}).
\end{align*}
Thus, 
\begin{align*}
\gamma_{A_{1,2}}\left((\tilde{E}_k\Delta H_2(e_1, \overline{s})) \cap \Omega_+\right) \le C_k \lambda_{\gamma_{A}}(E_k) + \frac{1}{5} \gamma_{A_{1,2}}((\tilde{E}_k \Delta H_2(e_1, \overline{s})) \cap \Omega_+)\\
\longrightarrow \gamma_{A_{1,2}}\left((\tilde{E}_k\Delta H_2(e_1, \overline{s})) \cap \Omega_+\right) \le \frac{5}{4}C_k \lambda_{\gamma_{A}}(E_k) \le \frac{5}{8}\lambda_{\gamma_{A}}(E_k).
\end{align*}
Applying an identical argument for $\gamma_{A_{1,2}}(\tilde{E}_k \Delta H_2(e_1,\overline{s})) \cap \Omega_-$ gives the full claim.
\end{proof}

\subsection{Proof of Small Perturbation Case}

\begin{lemma}[Quantitative separation from the linear case]\label{l:separation from linear case}
    Let $A\in S^+_{n\times n}(\mathbb{R})$ be as in Definition \ref{d:standing assumptions}(3) with eigenvalues $$1=d_1^2 \le d_2^2\le 2.$$ Let $\overline{s} \in (-\infty, 0]$, and let $0<\delta_0<1$ be as in Lemma \ref{l:close to linear and comparability of lambdas}.  For any recovery sequence $\{E_k\}_k$ for $Q(H_n(e_1, \overline{s}))$, let $\{\tilde{E}_k\}_k$ be the sequence guaranteed by Theorem \ref{t:technical theorem} and $\tilde{u}_k(x_2)$ be the defining function for $\partial\tilde{E}_k$.  As in Lemma \ref{l:convergence to linear case} and Lemma \ref{l:close to linear and comparability of lambdas}, we define $w_k := \tilde{u}_k \circ \frac{1}{d_2}$. We write $w_k(y) = a_{0, k} + \sum_{i=1}^{\infty}a_{i,k}H_i(y)$. 
    
    Then,
    \begin{align*}
    |a_{1, k}|\le 1-\delta_0
    \end{align*}
    for all sufficiently large $k \in \mathbb{N}$.
\end{lemma}

\begin{proof}
We argue by contradiction.  In particular, suppose that there exists a subsequence (also indexed by $k \in \mathbb{N}$) such that $a_{1, k} > 1-\delta_0$ for all $k \in \mathbb{N}$.  We will show that this contradicts the strong convergence of $D^2v_k \rightarrow 0$ in $L^2(\mathbb{R},\gamma_1)$ from Lemma \ref{l:convergence to linear case}.

\textbf{Finer estimates on $\boldsymbol{\ell_k^{\pm}}$.}  First, we note that $H_2(e_1, \ell_k^{\pm})$ is the unique (up to rotation in the case $d_1=d_2$) global minimizer for $\gamma_{A_{1, 2}}$-perimeter among all sets $E \subset \mathbb{R}^2$ with $\gamma_{A_{1,2}}(E)= 2 \gamma_{A_{1,2}}(I_k \cap \Omega_\pm)$ (see \eqref{unique minimizer}).   Restricting our attention to $\ell_k^+$, we have that $\partial H_2(e_1, \ell_k^+)=\{x_1=\ell_k^+\}$ is also a critical point for perimeter in $\Omega_+$ under the volume constraint $\gamma_{A_{1,2}}(E \cap \Omega_+)=  \gamma_{A_{1,2}}(I_k \cap \Omega_+)$.  Since by reflection we obtain a volume-constrained perimeter \textit{global minimizer} (i.e., an isoperimetric set), it must be that $H_2(e_1, \ell_k^+) \cap \Omega_+$ is not just a critical point, but a \textit{local minimizer} for perimeter in $\Omega_+$ under the volume constraint $\gamma_{A_{1,2}}(E \cap \Omega_+)=  \gamma_{A_{1,2}}(I_k \cap \Omega_+)$.  That is,
\begin{align*}
2\gamma_{A_{1,2}}((I_k \setminus H_2(e_1, \overline{s}))\cap \Omega^+) = \gamma_{A_{1,2}}(H_2(e_1, \ell_k^+) \setminus H_2(e_1, \overline{s})).
\end{align*}
Recalling the volume estimates of \eqref{e:excess mass I_k estimate}, we immediately obtain
\begin{align*}
        \frac{1}{5}\lambda_{\gamma_A}(E_k)\le \frac{d_1d_2}{2\pi}\int_{\overline{s}}^{\ell_k^+}\int_{\mathbb{R}}e^{-\frac{1}{2}(d_1^2x_1^2 + d_2^2x_2^2)}\,dx_2dx_1 & = \frac{d_1}{\sqrt{2\pi}}\int_{\overline{s}}^{\ell_k^+}e^{-\frac{1}{2}(d_1^2x_1^2)}\,dx_1\\
        & \le 2\frac{d_1}{\sqrt{2\pi}}(\ell_k^+ - \overline{s})e^{-\frac{1}{2}d_1^2|\overline{s}|^2},
    \end{align*}
    for all sufficiently large $k \in \mathbb{N}$ by Lemma \ref{l:continuity of Ck}.  That is, 
    \begin{align}\label{e:LB on ell_k height}
        \frac{1}{10}\frac{\sqrt{2\pi}}{d_1}e^{\frac{1}{2}d_1^2|\overline{s}|^2}\lambda_{\gamma_A}(E_k)\le \ell_k^+ - \overline{s}.
    \end{align}

\textbf{The contradiction.} Now, suppose towards a contradiction that there exists a sequence $\tilde{u}_k$ such that $w_k$ satisfies $a_{1,k}>1-\delta_0$. Let $\phi \in C^1_c(\mathbb{R})$ be a non-negative function such that $\Gamma_k(\text{supp}(\phi)) \subset \{1\le x_2\le 100\}$.  As an immediate consequence of the strong convergence of Lemma \ref{l:close to linear and comparability of lambdas} and Lipschitz control from Corollary \ref{c:tamanini convergence} by change of variables, we have for all sufficiently large $k \in \mathbb{N}$
\begin{align*}
    \int \phi(\tau) (\Gamma_k^+)''(\tau)\,d\tau & =  \int \tilde{\phi}(x_2) (\frac{D^2\tilde{u}_k}{\sqrt{1+|D\tilde{u}_k|^2}} - \frac{|D\tilde{u}_k|^2 D^2\tilde{u}_k}{(1+|D\tilde{u}_k|^2)^{\frac{3}{2}}})\, dx_2\\
    & \le \frac{1}{5,000}\norm{\tilde{u}_k}_{L^2(\mathbb{R},\gamma_{d_2})}\int \tilde{\phi}(x_2)\,dx_2,
\end{align*}
where $\tilde{\phi}(x_2) = \phi((\Gamma_{k}^+)^{-1}((\tilde{u}_k(x_2), x_2)))$.  On the other hand, from the ODE \eqref{e:ODE}, change of variables, and a Taylor expansion, we have 
\begin{align*}
    & \int \phi(\tau) (\Gamma_k^+)''(\tau)\,d\tau =  \int \phi_i(\tau) \left(d_1^2\ell_k - d_1^2 (\Gamma_k^+(\tau) \cdot e_1) \eta_{k,1}^+(\tau) - d_2^2(\Gamma_k^+(\tau)\cdot e_2)\eta_{k,2}^+(\tau)\right)\,d\tau\\
    & =  \int \phi(\tau) \Big(d_1^2\ell_k - d_1^2 (\overline{s}+\tilde{u}_k(\pi_{e_2}(\Gamma_k^+(\tau))) \left(\frac{1}{\sqrt{1+|D\tilde{u}_k(\pi_{e_2}(\Gamma_{k}^+(\tau)))|^2}}\right)\\
    & \qquad - d_2^2(\Gamma_k^+(\tau)\cdot e_2)\left(-\frac{D\tilde{u}_k(\pi_{e_2}(\Gamma_{k}^+(\tau)))}{\sqrt{1+|D\tilde{u}_k(\pi_{e_2}(\Gamma_{k}^+(\tau)))|^2}}\right)\Big)\,d\tau\\
    & =  \int \tilde{\phi}(x_2) \Big(d_1^2\ell_k\sqrt{1+|D\tilde{u}_k(x_2)|^2} - d_1^2 (\overline{s}+\tilde{u}_k(x_2)) +d_2^2x_2D\tilde{u}_k(x_2)\Big)\,dx_2\\
    & =  \int \tilde{\phi}(x_2) \Big(d_1^2(\ell_k\sqrt{1+|D\tilde{u}_k(x_2)|^2} - \overline{s})\Big)\,dx_2\\
    & \qquad \int \tilde{\phi}(x_2) \Big(-d_1^2 +d_2^2\Big)\norm{\tilde{u}_k}_{L^2(\mathbb{R},\gamma_1)}cx_2\,dx_2 + \int \tilde{\phi}(x_2) (-d_1^2f_k(x_2) + d_2^2x_2Df_k(x_2))\,dx_2\\
    & \ge \int \tilde{\phi}(x_2) \Big(d_1^2(\ell_k\sqrt{1+|D\tilde{u}_k(x_2)|^2} - \overline{s})\Big)\,dx_2 -(\frac{d_1^2}{10,000} + \frac{d_2^2}{100})\norm{\tilde{u}_k}_{L^2(\mathbb{R},\gamma_1)}) \int \tilde{\phi}(x_2) \,dx_2\\
    & \ge \int \tilde{\phi}(x_2) \Big(d_1^2(\ell_k - \overline{s})\Big)\,dx_2 -(\frac{d_1^2}{10,000} + \frac{d_2^2}{100})\norm{\tilde{u}_k}_{L^2(\mathbb{R},\gamma_1)}) \int \tilde{\phi}(x_2) \,dx_2 - o(\norm{\tilde{u}_k})\int \phi(x_2)\, dx_2,
\end{align*}
where we have used the assumption that $d_1 \le d_2$ to drop the middle term, the Taylor expansion 
\begin{align*}
    \ell_k\sqrt{1+|D\tilde{u}_k|^2} - \overline{s} = \ell_k - \overline{s} + \ell_k(\frac{1}{2} |D\tilde{u}_k|^2 + o(|D\tilde{u}_k|^2)),
\end{align*}
and the estimates on $\norm{D\tilde{u}_k}_{L^2(K, \gamma_{A'_{1,2}, 0})}$ from Lemma \ref{l:close to linear and comparability of lambdas}.

In particular, for any $x_2 \in (1, 100]$ we may choose $\phi_i$ such that $\tilde{\phi}_i \rightarrow \chi_{[1, x_2]}$.  Then, recalling that $\norm{D\tilde{u}_k}_{\infty}\rightarrow 0$ as $k \rightarrow \infty$, we have from \eqref{e:LB on ell_k height}
\begin{align*}
\frac{1}{5,000}\norm{\tilde{u}_k}_{L^2(\mathbb{R},\gamma_{d_2})}(x_2-1) & \ge  d_1^2\Big(\ell_k^+ - \overline{s}\Big)(x_2 - 1) -(\frac{d_1^2}{10,000} + \frac{d_2^2}{100})\norm{\tilde{u}_k}_{L^2(\mathbb{R},\gamma_{d_2})}(x_2-1)\\
&\qquad + o(\norm{\tilde{u}_k}_{L^2(\mathbb{R},\gamma_{d_2})})\\
& \ge d_1^2\Big(\frac{1}{10}\frac{\sqrt{2\pi}}{d_1}e^{\frac{1}{2}d_1^2|\overline{s}|^2}\lambda_{\gamma_A}(E_k)\Big)(x_2 - 1)\\
& \qquad -(\frac{d_1^2}{10,000} + \frac{d_2^2}{100})\norm{\tilde{u}_k}_{L^2(\mathbb{R},\gamma_{d_2})}(x_2-1)+ o(\norm{\tilde{u}_k}_{L^2(\mathbb{R},\gamma_{d_2})}).
\end{align*}
Now, we use Lemma \ref{l:close to linear and comparability of lambdas}(2) and the Taylor expansion estimate in Lemma \ref{l:fraenkel symmetry} to estimate 
\begin{align*}
    \lambda_{\gamma_A}(E_k)& \ge \frac{1}{2}\gamma_{2}(\tilde{E}_k \Delta H_2(e_1, \overline{s})) \qquad \qquad \qquad \qquad \text{Lemma \ref{l:close to linear and comparability of lambdas}(2)}\\
    & = \frac{1}{2}\norm{\tilde{u}_k}_{L^1(\mathbb{R},\gamma_{d_2})}e^{-\frac{1}{2}d_1^2|\overline{s}|^2} + O_4(\epsilon^2) \qquad \qquad \text{Lemma \ref{l:fraenkel symmetry}}\\
    & \ge \frac{1}{6}\norm{\tilde{u}_k}_{L^2(\mathbb{R},\gamma_{d_2})}e^{-\frac{1}{2}d_1^2|\overline{s}|^2} + o(\norm{\tilde{u}_k}_{L^2(\mathbb{R},\gamma_{d_2})}^2) \qquad \eqref{e:L1 and L2 comparable for large a1}.
\end{align*}

Plugging this into our previous estimates and canceling terms, we obtain 
\begin{align*}
    \frac{1}{5,000} & \ge d_1^2\Big(\frac{1}{10}\frac{\sqrt{2\pi}}{d_1} \frac{1}{6}\Big) -(\frac{d_1^2}{10,000} + \frac{d_2^2}{100})+ o(k).
\end{align*}
Therefore, using our assumption that $1=d_1^2 \le d_2^2\le 2$, we reach a contradiction for all sufficiently large $k \in \mathbb{N}$, since 
\begin{align*}
    \frac{1}{5,000} & < \left(\frac{1}{10}\sqrt{2\pi} \frac{1}{6}\right)-\left(\frac{1}{10,000} + \frac{1}{50}\right)\\
    & = \frac{\sqrt{2\pi}}{60}-\left(\frac{1}{50}+ \frac{1}{10,000}\right).
\end{align*}
Thus, no such subsequence $\tilde{u}_k$ exists.  This proves the lemma.
\end{proof}

\begin{theorem}[The Small Perturbation Regime]\label{l:fuglede-type estimate 2}
Let $A, A_{1,2}, \alpha, \overline{s}$ as in Definition \ref{d:standing assumptions}.  Let $\{\tilde{E}_k\}_k$ be the sequence guaranteed by Theorem \ref{t:technical theorem}.  Then, the following inequality holds.
    \begin{align*}
        \liminf_{k\rightarrow \infty}\frac{\delta_{\gamma_{A_{1,2}}}(\tilde{E}_k)}{\gamma_{A_{1,2}}(\tilde{E}_k \Delta H_2(e_1, \overline{s}))^2} \ge \begin{cases}\frac{(2\pi)e^{\frac{d_1^2\overline{s}^2}{2}}}{d_1}\left[\frac{d_2^2-d_1^2}{2}+d_2^2(\delta_0-\frac{1}{2}\delta_0^2)\right]& 1 \le \frac{d_2^2}{d_1^2}\le 2\\
        \frac{(2\pi)e^{\frac{d_1^2\overline{s}^2}{2}}}{d_1}\left[\frac{d_2^2-d_1^2}{2}\right] & \frac{d_2^2}{d_1^2}>2,
        \end{cases}
    \end{align*}
    where $0<\delta_0<1$ is the constant given in Lemma \ref{l:close to linear and comparability of lambdas}.
\end{theorem}

\begin{proof}
By change of variables, we may assume that $1=d_1^2 \le d_2^2$. By Lemma \ref{l:fuglede-type estimate}, it is sufficient to estimate the ratio
\begin{align*}
    \frac{\int_{\mathbb{R}}\frac{1}{2}|\nabla \tilde{u}_k(x_2)|^2 d\gamma_{A',\overline{s}}(x_2)}{\int_{\mathbb{R}}|\tilde{u}_k(x_2)|^2d\gamma_{A',\overline{s}}(x_2)}.
\end{align*} 
First, note by change of variables
\begin{align}\label{e:denominator bound in fuglede}
    \norm{\tilde{u}_k}_{L^2(\mathbb{R}, \gamma_{A'_{1,2}, \overline{s}})}:= e^{-\frac{1}{2}|\overline{s}|^2}\int_{\mathbb{R}}|\tilde{u}_k(x_2)|^2 e^{-\frac{1}{2}d_2^2|x_2|^2}\, dx_2 = \frac{e^{-\frac{1}{2}|\overline{s}|^2}}{d_2}\int_{\mathbb{R}}\tilde{u}_{k}(\frac{1}{d_2}y)\,d\gamma_1(y). 
\end{align}
Now, using the standard Hermite polynomials for $L^2(\mathbb{R},\gamma_1)$ for $u= \tilde{u}_k \circ \frac{1}{d_2} \in W^{1,\infty}(\mathbb{R},\gamma_1)$, we have Parseval-type identities
\begin{align*}
\|u\|^2_{L^2(\R,\gamma_1)}
=
\int_{\mathbb{R}}
|u(y)|^2\,d\gamma_{1}(y) =
\sum_{i=0}^{\infty} a_{i}^2,\qquad 
Du(y)=D\left[\sum_{i=0}^\infty a_{i}H_{i}(y)\right]=\sum_{i=1}^\infty a_{i} DH_{i}(y).
\end{align*}
Therefore, by the iteration identity $DH_{i}=\sqrt{i}H_{i-1}$ for $i\geq 1$ (see, for example, \cite[Lemma 1.3.2]{MR1642391}), we have 
\begin{align*}
\norm{\nabla \tilde{u}_k}^2_{L^2(\mathbb{R}, \gamma_{A'_{1,2}, \overline{s}})}& := e^{-\frac{1}{2}|\overline{s}|^2}\int_{\mathbb{R}}|\nabla \tilde{u}_k|^2\, e^{-\frac{1}{2}d_2^2|x_2|^2}dx_2\\
& = \frac{1}{d_2}e^{-\frac{1}{2}|\overline{s}|^2}\int_{\mathbb{R}}|\nabla \tilde{u}_k(\frac{1}{d_2}y)|^2\, d\gamma_1(y)\\
& = d_2e^{-\frac{1}{2}|\overline{s}|^2}\int_{\mathbb{R}}|\nabla u(y)|^2\, d\gamma_1(y)\\
&= d_2e^{-\frac{1}{2}|\overline{s}|^2} \sum_{i=1}^\infty ia_{i-1,k}^2.
\end{align*}
And, using Lemma \ref{l:first order estimate} with $H_2(e_1,\overline{s}_k) = \alpha_k = \gamma_{A_{1,2}}(\tilde{E}_k)$ and $u = \tilde{u}_k \circ \frac{1}{d_2}$, we have that for any $0<\varepsilon$ and all sufficiently large $k \in \mathbb{N}$
\begin{align*}
|a_{0}|=\left|\int_{\R}\tilde{u}_k(\frac{1}{d_2}y)d\gamma_{1}(y)\right|&=d_2 e^{\bar{s}_k^2/2}\left|\int_{\R}\tilde{u}_k(x_2)d\gamma_{A_{1,2}',\bar{s}_k}(x_2)\right|\\
&\leq d_2e^{\bar{s}_k^2/2}\left(\frac{d_1^2\bar{s}_k}{2}\int_{\mathbb{R}}|\tilde{u}_k(x_2)|^2d\gamma_{A_{1,2}',\overline{s}_k}(x_2)+O_3(\varepsilon^3)\right) \qquad \text{Lemma \ref{l:first order estimate}}\\
&\leq d_2Ce^{\bar{s}_k^2/2}\|\tilde{u}_k\|^2_{L^2(\R,\gamma_{A'_{1,2},\bar{s}_k})}\\
&\leq d_2C\varepsilon e^{\bar{s}_k^2/2}\|\tilde{u}_k\|_{L^2(\R,\gamma_{A'_{1,2},\bar{s}_k})}.
\end{align*}

Now, we break into two cases.  

\textbf{Case 1: $1\le d_2^2\le 2$.}
In this case, by Lemma \ref{l:separation from linear case}, we may estimate
\begin{align*}
d_2 \norm{\tilde{u}_k}_{L^2(\mathbb{R}, \gamma_{A_{1,2}', 0})} & =\|\tilde{u}_k \circ \frac{1}{d_2}\|^2_{L^2(\mathbb{R},\gamma_{1})}\\
&=\sum_{i=0}^\infty a_{i,k}^2\\
&\leq d_2^2C^2\varepsilon^2e^{\bar{s}_k^2}\|\tilde{u}_k\|_{L^2(\mathbb{R},\gamma_{A_{1,2}',\bar{s}_k})}^2 +(1-\delta_0)^2\|\tilde{u}_k\circ \frac{1}{d_2}\|_{L^2(\mathbb{R},\gamma_{1})}^2\\
& \qquad +\sum_{i=2}^\infty a_{i,k}^2.
\end{align*}
Absorbing the first two terms on the right hand side, we see that
\begin{align*}
\left(2\delta_0 - \delta_0^2 -d_2C\varepsilon^2e^{\bar{s}_k^2}\frac{\|\tilde{u}_k\|_{L^2(\mathbb{R},\gamma_{A_{1,2}',\bar{s}_k})}^2}{\|\tilde{u}_k\circ \frac{1}{d_2}\|^2_{L^2(\mathbb{R}, \gamma_{1})}}\right)\|\tilde{u}_k\circ \frac{1}{d_2}\|^2_{L^2(\mathbb{R}, \gamma_{1})}\leq \sum_{i=2}^\infty a_{i,k}^2,  
\end{align*}
where we have used the assumption that $d_2 \ge 1$.  However, unwinding the notation ($A_{1,2}'=d_2$), we see that
\begin{align*}
d_2e^{\frac{\bar{s}_k^2}{2}}\|\tilde{u}_k\|_{L^2(\mathbb{R}, \gamma_{A_{1,2}',\bar{s}_k})}^2= \|\tilde{u}_k\circ \frac{1}{d_2} \|_{L^2(\mathbb{R},\gamma_1)}^2.
\end{align*}
Therefore, we have that for any $0<\varepsilon$ and all sufficiently large $k \in \mathbb{N}$,
\begin{align}\label{e: lower bound on higher order terms 1}
    (2\delta_0-\delta_0^{2} - C^2\varepsilon^2e^{\frac{\bar{s}_k^2}{2}})\|\tilde{u}_k\circ \frac{1}{d_2}\|^2_{L^2(\mathbb{R}, \gamma_{1})}\leq \sum_{i=2}^\infty a_{i,k}^2.
\end{align}

This immediately implies an improvement in the Sobolev-type inequality.  That is,
\begin{align*}
\|\nabla \tilde{u}_k\|^2_{L^2(\mathbb{R},\gamma_{A_{1,2}', \overline{s}})}&= d_2 e^{-\frac{1}{2}|\overline{s}|^2}\Big(a_{1,k}^2+\sum_{i=2}^\infty ia_{i,k}^2\Big)\\
&\geq d_2 e^{-\frac{1}{2}|\overline{s}|^2}\Big(a_{1,k}^2+2\sum_{i=2}^\infty a_{i,k}^2\Big)\\
& \geq d_2 e^{-\frac{1}{2}|\overline{s}|^2}\Big(\sum_{i=1}^\infty a_{i,k}^2 + (2\delta_0-\delta_0^{2}- C^2\varepsilon_0^2e^{\frac{|\overline{s}_k|^2}{2}})\|\tilde{u}_k\circ \frac{1}{d_2}\|^2_{L^2(\mathbb{R}, \gamma_{1})}\Big) \qquad \eqref{e: lower bound on higher order terms 1}\\
& \geq d_2 e^{-\frac{1}{2}|\overline{s}|^2}\Big(\sum_{i=0}^\infty a_{i,k}^2 - d_2C^2\varepsilon^2e^{\bar{s}_k^2} \|\tilde{u}_k\circ\frac{1}{d_2}\|_{L^2(\mathbb{R},\gamma_1)}^2\\
& \qquad + (2\delta_0-\delta_0^{2}-C^2\varepsilon^2e^{\frac{\bar{s}_k^2}{2}})\|\tilde{u}_k\circ \frac{1}{d_2}\|^2_{L^2(\mathbb{R}, \gamma_{1})}\Big).
\end{align*}
Therefore, by \eqref{e:denominator bound in fuglede} and the above inequality, we see that
\begin{align*}
\frac{\int_{\mathbb{R}}|\nabla \tilde{u}(y)|^2 d\gamma_{A_{1,2}',\overline{s}}(y)}{2\Big(1+O_5(\varepsilon)\Big)\int_{\mathbb{R}}|u(y)|^2d\gamma_{A_{1,2}',\overline{s}}(y)} \geq \frac{d_2^2(1-d_2C\varepsilon^2e^{\bar{s}_k^2})}{2\Big(1+O_5(\varepsilon)\Big)}+\frac{d_2^2(2\delta_0-\delta_0^2 - C^2\varepsilon^2e^{\frac{\bar{s}_k^2}{2}})}{2\Big(1+O_5(\varepsilon)\Big)}.
\end{align*}
Letting $\varepsilon \rightarrow 0^+$, Lemma \ref{l:fuglede-type estimate} immediately gives that
\begin{align*}
\liminf_{k \rightarrow \infty}\frac{\delta_{\gamma_A}(\tilde{E}_k)}{\gamma_{A_{1,2}}(\tilde{E}_k\Delta H_2(e_1,\overline{s}))^2}&\geq \frac{(2\pi)e^{\frac{d_1^2\overline{s}^2}{2}}}{d_1}\left[\frac{d_2^2-d_1^2}{2}+d_2^2(\delta_0-\frac{1}{2}\delta_0^2)\right].
\end{align*}

\textbf{Case 2: $2<d_2^2$.}
Following exactly the same argument as above, we obtain without any control on $a_{1,k}$, 
\begin{align*}
\|\nabla \tilde{u}_k\|^2_{L^2(\mathbb{R},\gamma_{A_{1,2}', \overline{s}_k})}&= d_2 e^{-\frac{1}{2}|\overline{s}_k|^2}\Big(a_{1,k}^2+\sum_{i=2}^\infty ia_{i,k}^2\Big)\\
&\geq d_2 e^{-\frac{1}{2}|\overline{s}_k|^2}\Big(a_{1,k}^2+2\sum_{i=2}^\infty a_{i,k}^2\Big)\\
& \geq d_2 e^{-\frac{1}{2}|\overline{s}_k|^2}\Big(\sum_{i=1}^\infty a_{i,k}^2\Big)\\
& \geq d_2 e^{-\frac{1}{2}|\overline{s}_k|^2}\Big(\sum_{i=0}^\infty a_{i,k}^2 - d_2C^2\varepsilon^2e^{\bar{s}_k^2} \|\tilde{u}_k\circ\frac{1}{d_2}\|_{L^2(\mathbb{R},\gamma_1)}^2\Big).
\end{align*}
Plugging this in, we obtain
\begin{align*}
\frac{\int_{\mathbb{R}}|\nabla \tilde{u}(y)|^2 d\gamma_{A_{1,2}',\overline{s}_k}(y)}{2\Big(1+O_5(\varepsilon)\Big)\int_{\mathbb{R}}|u(y)|^2d\gamma_{A_{1,2}',\overline{s}_k}(y)} \geq \frac{d_2^2(1-d_2C\varepsilon^2e^{\bar{s}_k^2})}{2\Big(1+O_5(\varepsilon)\Big)}.
\end{align*}
Letting $\varepsilon \rightarrow 0^+$, Lemma \ref{l:fuglede-type estimate} immediately gives that in this case
\begin{align*}
\liminf_{k \rightarrow \infty}\frac{\delta_{\gamma_A}(\tilde{E}_k)}{\gamma_{A_{1,2}}(\tilde{E}_k\Delta H_2(e_1,\overline{s}))^2}&\geq \frac{(2\pi)e^{\frac{d_1^2\overline{s}^2}{2}}}{d_1}\left[\frac{d_2^2-d_1^2}{2}\right].
\end{align*}
This completes the proof.
\end{proof}

\section{Qualitative Stability}\label{s:qualitative stability}

In this section, we prove Lemma \ref{l:qualitative stability} and Lemma \ref{l:continuity of epsilon}.

\begin{lemma}[Qualitative Stability]\label{l:qualitative stability 2}
    Let $A$ be a symmetric, positive-definite $(n \times n)$-matrix.  For every $0<\varepsilon$ and $\alpha \in (0,1)$, there is a $\overline{\rho}(\alpha, \varepsilon, A, n)>0$ such that if $F\subset \mathbb{R}^{n-1}$ satisfies
    \begin{align*}
        \gamma_A(F)=\alpha, \qquad \delta_{\gamma_A}(F)< \overline{\rho}(\alpha, \varepsilon, A, n),
    \end{align*}
    then $\lambda_{\gamma_A}(F)< \varepsilon$.
\end{lemma}

\begin{proof}
    We argue by contradiction. Suppose that the statement is false and that there exists a pair $(\alpha,\varepsilon)$ for which there exists a sequence of sets $\{F_i\}_i$ such that 
    \begin{align*}
        \gamma_A(F_i)=\alpha, \qquad \delta_{\gamma_A}(F_i)\le i^{-1},
    \end{align*}
    but $\lambda_{\gamma_A}(F_i)\ge \varepsilon$.

    By weak compactness of sets of finite $\gamma_{A}$-perimeter and lower semi-continuity of $\gamma_A$-perimeter under convergence in $\gamma_A$-measure (Theorem \ref{t:weak compactness, LSC}), there exists an $F \subset \mathbb{R}^n$ and a subsequence $\{F_j\}_j$ such that $F_j \rightarrow F$ in $L^1(\mathbb{R}^n,\gamma_A)$ and $\delta_{\gamma_A}(F) \le \liminf_{j}\delta_{\gamma_A}(F_j) = 0$. In particular, $F$ is a $\gamma_A$-perimeter minimizer.  However, by \eqref{unique minimizer}, the unique $\gamma_A$-perimeter-minimizers are sets of the form $H_n(\omega',\overline{s})$ for $\omega' \in V_{d_{\min}}(\sqrt{A})$. Thus, $F_j \rightarrow H_n(\omega',\overline{s})$ in $L^1(\mathbb{R}^n,\gamma_{A})$, contradicting the assumption that $\lambda_{\gamma_A}(F_i)\ge \varepsilon$. Thus, the result holds.
\end{proof}

\begin{definition}\label{d:rho as supremum}
    For any $A,\alpha,\epsilon$ as in the hypotheses of Lemma \ref{l:qualitative stability 2}, we define $\rho(\alpha, \epsilon, A, n)$ to be the supremum of all the $\overline{\rho}(\alpha, \epsilon, A,n)$ such that Lemma \ref{l:qualitative stability 2} holds.  
\end{definition}
    
\begin{lemma}[Properties of the modulus $\rho$]\label{l:properties of modulus}
The function $\rho(\alpha,\epsilon, A, n)$ defined by Definition \ref{d:rho as supremum} satisfies the following properties.
\begin{enumerate}
\item For all $A,\alpha$ as in the hypotheses of Lemma \ref{l:qualitative stability 2}, $\rho(\alpha, \epsilon, A, n)$ is left continuous in $\epsilon$ and monotone non-decreasing in $\epsilon$.  Moreover, there exists an $0<\epsilon_2(\alpha,A,n)$ such that $\rho(\alpha, \epsilon_2, A, n)<\infty$.
\item The modulus $\rho(\alpha, \epsilon, A, n)$ is jointly lower-semicontinuous in $\alpha$ and $A$ on the set
\begin{align*}
    \{(\alpha, A) \in (0,1)\times S_{n \times n}^+(\mathbb{R}): \rho(\alpha, \epsilon, A, n)<\infty \}.
\end{align*}
\end{enumerate}
\end{lemma}

\begin{proof}
\textbf{Proof of (1):} We note that $\rho(\alpha, \epsilon, A, n)$ is monotone non-decreasing in $\epsilon$ by definition as the supremum. Now, we show that for any $0<\alpha<1$, and any $A \in S^+_{n\times n}(\mathbb{R})$ as in Definition \ref{d:standing assumptions}(3), $\rho(\alpha, \epsilon, A, n)<\infty$ for sufficiently small $0<\epsilon$.  By considering $B_r(0)\subset \mathbb{R}^n$ such that $\gamma_A(B_r(0))= \alpha$, we see from \eqref{unique minimizer} that $\infty>\delta_{\gamma_A}(B_r(0))>0$. Since $B_r(0)$ is not a minimizing half-plane, it must be that $\lambda_{\gamma_A}(B_r(0))>0$. Let $0<\epsilon'(\alpha, A, n)<\lambda_{\gamma_A}(B_r(0))$. Then, it must be that $$0<\rho(\alpha, \epsilon'(\alpha, A, n), A, n)<\delta_{\gamma_A}(B_r(0)) <\infty.$$  

Now, let $A,\alpha$ be fixed. We argue that $\rho(\alpha, \epsilon, A, n)$ is left continuous in $\epsilon.$  Suppose that there exists an $0<\epsilon'$ such that
\begin{align*}
    \lim_{\epsilon \nearrow \epsilon'}\rho(\alpha, \epsilon, A, n) <     \lim_{\epsilon \searrow \epsilon'}\rho(\alpha, \epsilon, A, n),
\end{align*}
since otherwise, by monotonicity, there is nothing to prove. By the definition of $\rho(\alpha, \epsilon, A, n)$ as a supremum, for any sequence of $\epsilon_i \nearrow \epsilon'$ and any $$\kappa\in (0,\lim_{\epsilon \nearrow \epsilon'}\rho(\alpha, \epsilon, A, n) -     \lim_{\epsilon \searrow \epsilon'}\rho(\alpha, \epsilon, A, n))$$ sufficiently small, there exists a subsequence $\epsilon_j\nearrow \epsilon'$ and sequence of sets $F_j \subset \mathbb{R}^n$ such that $\gamma_A(F_j)= \alpha$ and
\begin{align*}
    \rho(\alpha, A, \epsilon_j, n) < \delta_{\gamma_A}(F_j)& < \lim_{\epsilon \searrow \epsilon'}\rho(\alpha, \epsilon, A, n)-\kappa\\
    \lambda_{\gamma_A}(F_j) &\ge \epsilon_j.
\end{align*}
In particular, the sequence $F_j$ is uniformly bounded as sets of finite $\gamma_A$-perimeter.  By weak compactness and lower semicontinuity of $\delta_{\gamma_A}$, there exists a set $F_0 \subset \mathbb{R}^n$  and a further subsequence (also indexed by $j$) such that
\begin{align*}
    F_j \overset{L^1(\mathbb{R}^n,\gamma_A)}{\rightarrow } F_0 \qquad \text{and} \qquad \delta_{\gamma_A}(F_0) \le \lim_{\epsilon \searrow \epsilon'}\rho(\alpha, \epsilon, A, n)-\kappa.
\end{align*}
But, the convergence in measure implies that for any individual minimizing half-plane $H_n(\omega, \overline{s})$,
\begin{align*}
    \epsilon_j \le \lambda_{\gamma_A}(F_j) \le \gamma_{A}(F_j \Delta H_n(\omega, \overline{s})) \rightarrow \gamma_{A}(F_0 \Delta H_n(\omega, \overline{s})).
\end{align*}
Letting $j \rightarrow \infty$, we see that $\epsilon' \le \lambda_{\gamma_A}(F_0)$. Thus, it must be that $\rho(\alpha, \epsilon', A, n) \le \delta_{\gamma_A}(F_0) \le \lim_{\epsilon \searrow \epsilon'}\rho(\alpha, \epsilon, A, n)-\kappa$.  Letting 
\begin{align*}
    \kappa \rightarrow  \left|\lim_{\epsilon \searrow \epsilon'}\rho(\alpha, \epsilon, A, n)-\lim_{\epsilon \nearrow \epsilon'}\rho(\alpha, \epsilon, A, n)\right|,
\end{align*} we see that $\rho(\alpha, \epsilon', A, n)$ must be left continuous in $\epsilon$, since by monotonicity
\begin{align*}
    \lim_{\epsilon \nearrow \epsilon'}\rho(\alpha, \epsilon, A, n) & \le \rho(\alpha, \epsilon', A, n)\\
    & \le \lim_{\epsilon \searrow \epsilon'}\rho(\alpha, \epsilon, A, n) - \left|\lim_{\epsilon \searrow \epsilon'}\rho(\alpha, \epsilon, A, n)-\lim_{\epsilon \nearrow \epsilon'}\rho(\alpha, \epsilon, A, n)\right|\\
    & = \lim_{\epsilon \nearrow \epsilon'}\rho(\alpha, \epsilon, A, n).
\end{align*}

\textbf{Proof of (2):} Let $\alpha,A,\epsilon$ be given as in the hypotheses of the lemma such that $\rho(\alpha, \epsilon, A, n)<\infty$.  Let $(\alpha_i, A_i) \rightarrow (\alpha, A)$ be a sequence such that $\lim_{i \rightarrow \infty}\rho(\alpha_i, \epsilon, A_i, n)$ exists, as is the case for, say, a sequence which realizes $\liminf_{(\alpha_i, A_i) \rightarrow (\alpha, A)}\rho(\alpha_i, \epsilon, A_i, n)$. 

Suppose for the sake of contradiction that
\begin{align*}
    \rho(\alpha, \epsilon, A, n) - \lim_{i \rightarrow \infty}\rho(\alpha_i, \epsilon, A_i,n) > \kappa >0.
\end{align*}
Then, for each $i \in \mathbb{N}$, there must exist a set $F_i \subset \mathbb{R}^n$ such that $\gamma_{A_i}(F_i) = \alpha_i$ and
\begin{align*}
    \rho(\alpha_i, \epsilon, A_i, n) & \le \delta_{\gamma_{A_i}}(F_i) <\rho(\alpha_i, \epsilon, A_i, n) + \frac{1}{2}\kappa <\infty,\\
    \lambda_{\gamma_{A_i}}(F_i)&\ge \epsilon>0.
\end{align*}
In particular, using the fact that the $\gamma_{A_i}$-perimeter is continuous in $A_i$, the collection $\{F_i\}_i$ is uniformly locally bounded as sets of finite $\gamma_{A}$-perimeter.  Therefore, by a diagonalization argument, there exists a set $F_0\subset \mathbb{R}^n$ and a subsequence such that
\begin{align*}
    & F_j \overset{L^1(\mathbb{R}^n,\gamma_{A})}{\rightarrow }F_0\\
    \delta_{\gamma_{A}}(F_0)&\le \liminf_{i \rightarrow \infty} \delta_{\gamma_{A_i}}(F_i) \le  \rho(\alpha, \epsilon, A, n) - \frac{1}{2}\kappa\\
     \lambda_{\gamma_{A}}(F_0)& =\lim_{i \rightarrow \infty}\lambda_{\gamma_{A_i}}(F_i)\ge \epsilon>0,
\end{align*}
where we have used the fact that $\gamma_{A_i}(E)$ is continuous in $A_i$ for any $E\subset \mathbb{R}^2$ to obtain the last line.  But, this contradicts the definition of $\rho(\alpha, \epsilon, A,n)$. Thus, the Claim 2 must be true.
\end{proof}

\begin{definition}\label{d:epsilon threshold def 2}
Let $A_{1,2}$ be a diagonal, positive-definite, $(2\times 2)$-matrix with eigenvalues $0<d_1^2 \le d_2^2$. Let $\alpha \in (0,1)$.  Let $\epsilon_0(d_1, d_2, \alpha)$ be defined as in Definition \ref{d:epsilon threshold def}.
\end{definition}

\begin{lemma}\label{l:continuity of epsilon 2}
    For $A_{1,2}, d_1^2, d_2^2, \alpha$, and $\epsilon_0(d_1, d_2, \alpha)$ as in Definition \ref{d:epsilon threshold def}, the function $\epsilon_0(d_1, d_2,\alpha)$ satisfies $0<\epsilon_0(d_1, d_2, \alpha)$ and is upper semicontinuous in $(d_1, d_2, \alpha)$.
\end{lemma}

\begin{proof}
To see that the threshold $\epsilon_0(d_1, d_2, \alpha)$ is strictly positive, we use a simple contradiction argument.  Recall $C_{SPR}(d_1, d_2, \alpha)$ as in \eqref{e:spr lb constant}.  Suppose towards a contradiction that there exists a matrix $A_{1,2}$ and a sequence of sets $\{F_i\}\subset \mathbb{R}^n$ such that $\gamma_{A_{1,2}}(F_i) = \gamma_{A_{1,2}}(H_{2}(e_1, \overline{s}))$ and
\begin{align*}
    0<\gamma_{A_{1,2}}(F_i \Delta H_2(e_1, \overline{s}))\le \frac{1}{i}, \qquad  \frac{1}{100}C_{SPR}(d_1,d_2,\alpha) \ge \frac{\delta_{\gamma_{A_{1,2}}}(F_i)}{\gamma_{A_{1,2}}(F_i\Delta H_2(e_1, \overline{s}))^2}
\end{align*} for all $i \in \mathbb{N}$.  Then, $\{F_i\}_i$ is admissible in the definition of $Q(H_2(e_1,\overline{s}))$. Applying Theorem \ref{t:technical theorem}(2), Corollary \ref{c:reduction ratio estimates}, and Theorem \ref{t:small perturbation regime} gives 
\begin{align*}
    \frac{1}{100}C_{SPR}(d_1, d_2, \alpha) \ge \liminf_{i \rightarrow \infty}\frac{\delta_{\gamma_{A_{1,2}}}(F_i)}{\gamma_{A_{1,2}}(F_i\Delta H_2(e_1, \overline{s}))^2}\ge \frac{1}{25}C_{SPR}(d_1, d_2, \alpha),
\end{align*}
a contradiction.  Thus, $0<\epsilon_0(d_1, d_2, \alpha)$.    

To see that $\epsilon_0(d_1, d_2,\alpha)$ is upper semicontinuous in $d_1, d_2,\alpha$, we recall the notation
\begin{align*}
    A_{d_1,d_2}:= \begin{pmatrix}
        d_1^2 & 0\\
        0 & d_2^2
    \end{pmatrix}
\end{align*} to keep track of the eigenvalues.  Observe that for any $F \subset \mathbb{R}^2$ 
\begin{align}\label{e:continuity 1}
    \lim_{(d_1',d_2') \rightarrow (d_1,d_2)}\gamma_{A_{d_1',d_2'}}(F) & = \gamma_{A_{d_1,d_2}}(F)\\ \label{e:continuity 2}
    \lim_{(d_1',d_2') \rightarrow (d_1,d_2)} \operatorname{Per}_
    {\gamma_{A_{d_1',d_2'}}}(F) & = \operatorname{Per}_{\gamma_{A_{d_1,d_1}}}(F).
\end{align}
Now, we claim that for any sequence $(d_1', d_2',\alpha') \rightarrow (d_1, d_2, \alpha)$ it cannot be that 
\begin{align*}
  \epsilon_0(d_1, d_2, \alpha) < \limsup_{(d_1', d_2', \alpha') \rightarrow (d_1, d_2, \alpha)} \epsilon_0(d_1', d_2', \alpha').
\end{align*}
We argue by contradiction.  Assume that $$0<\eta = \limsup_{(d_1', d_2', \alpha') \rightarrow (d_1, d_2, \alpha)} \epsilon_0(d_1', d_2', \alpha') -  \epsilon_0(d_1, d_2, \alpha).$$  Since $\epsilon_0(d_1, d_2, \alpha)$ is a supremum, there is a set $F$ such that $\gamma_{A_{1,2}}(F) = \alpha$,
\begin{align*}
    \epsilon_0(d_1, d_2, \alpha) \le \gamma_{A_{d_1, d_2}}(F \Delta H_2(e_1, \overline{s})) &< \epsilon_0(d_1, d_2, \alpha) + \frac{\eta}{2}\\
    \frac{\delta_{\gamma_{A_{d_1,d_2}}}(F)}{\gamma_{A_{d_1,d_2}}(F\Delta H_2(e_1, \overline{s}))^2} & < \frac{1}{100}C_{SPR}(d_1, d_2, \alpha).
\end{align*}
Now, we modify $F$ as follows.  Let $x_0,x_1 \in \mathbb{R}^n$ be such that $x_0$ is a Lebesgue point of $F^c$ and $x_1$ is a Lebesgue point of $F$. We define $F(d_1', d_2', \alpha') \subset \mathbb{R}^2$ and the radius $0\le r'=r'(d_1', d_2', \alpha')$ such that by 
\begin{align*}
F':= \begin{cases}
        F \cup B_{r'}(x_0) & \text{ if }\gamma_{A_{d_1',d_2'}}(F \cup B_{r'}(x_0)) = \alpha'\\
        F \setminus B_{r_i}(x_1) & \text{ if }\gamma_{A_{d_1',d_2'}}(F \setminus B_{r'}(x_1)) = \alpha'.
    \end{cases}
\end{align*}
By \eqref{e:continuity 1}, we have that $r' \rightarrow 0$ as $(d_1', d_2', \alpha') \rightarrow (d_1, d_2, \alpha)$ and 
\begin{align*}
    \gamma_{A_{d_1, d_2}}(F \Delta H_2(e_1, \overline{s})) & = \lim_{(d_1', d_2', \alpha') \rightarrow (d_1, d_2, \alpha)} \gamma_{A_{d_1', d_2'}}(F' \Delta H_2(e_1, \overline{s}))\\
    &= \lim_{(d_1', d_2', \alpha') \rightarrow (d_1, d_2, \alpha)} \gamma_{A_{d_1', d_2'}}(F' \Delta H_2(e_1, \overline{s}'))<\limsup_{(d_1', d_2', \alpha') \rightarrow (d_1, d_2, \alpha)} \epsilon_0(d_1', d_2', \alpha'),
\end{align*}
where we use the notation $\overline{s}'$ to be the quantity such that $\gamma_{A_{d_1',d_2'}}(H_2(e_1, \overline{s}'))=\alpha'$. 

Similarly, \eqref{e:continuity 2} and the fact that $r' \rightarrow 0^+$ gives
\begin{align*}
    \lim_{(d_1', d_2', \alpha') \rightarrow (d_1, d_2, \alpha)}\operatorname{Per}_{\gamma_{A_{d_1',d_2'}}}(F') = \operatorname{Per}_{\gamma_{A_{d_1,d_2}}}(F).
\end{align*}
But, by our assumption and \eqref{e:continuity 1}, \eqref{e:continuity 2}, we see that 
\begin{align*}
    \lim_{(d_1', d_2', \alpha') \rightarrow (d_1, d_2, \alpha)}\frac{1}{100}C_{SPR}(d_1', d_2', \alpha') & = \frac{1}{100}C_{SPR}(d_1, d_2, \alpha) \\
    & >\frac{\delta_{\gamma_{A_{d_1,d_2}}}(F)}{\gamma_{A_{d_1,d_2}}(F\Delta H_2(e_1, \overline{s}))^2}\\
    & = \lim_{(d_1', d_2') \rightarrow (d_1, d_2)} \frac{\delta_{\gamma_{A_{d_1',d_2'}}}(F')}{\gamma_{A_{d_1',d_2'}}(F'\Delta H_2(e_1, \overline{s}))^2}\\
    & = \lim_{(d_1', d_2') \rightarrow (d_1, d_2)} \frac{\delta_{\gamma_{A_{d_1',d_2'}}}(F')}{\gamma_{A_{d_1',d_2'}}(F'\Delta H_2(e_1, \overline{s}'))^2}.
\end{align*}
This is a contradiction for sufficiently large $j \in \mathbb{N}$.
\end{proof}

\addtocontents{toc}{\protect\setcounter{tocdepth}{0}}
\section*{Acknowledgments}
The authors would like to thank Tatiana Toro and Francesco Maggi for bringing this problem to their attention. Additionally, K.Y. was partially supported by NSF grants DMS--1954545 and DMS--FRG--1853993.

\addtocontents{toc}{\protect\setcounter{tocdepth}{1}}

\bibliography{references}
\bibliographystyle{amsalpha}

\end{document}